\documentclass[a4paper]{article}

\usepackage{amsthm,amsmath,amssymb}
\usepackage{mathrsfs}
\usepackage{graphicx,graphics}

\usepackage[nohead,margin=1.0in]{geometry}
\usepackage{threeparttable,booktabs}
\usepackage{algorithm, algorithmicx}

\usepackage{multirow}
\usepackage{makecell}
\usepackage[noend]{algpseudocode}
\usepackage[nohead,margin=1.0in]{geometry}
\usepackage{color}
\usepackage[normalem]{ulem}

\usepackage{hyperref}
\hypersetup{hypertex=true,colorlinks=true,linkcolor=blue,anchorcolor=blue,citecolor=blue}
\usepackage{epstopdf}
\graphicspath{{./figures/}}

\algdef{SE}[DOWHILE]{Do}{doWhile}{\algorithmicdo}[1]{\algorithmicwhile\ #1}

\newtheorem{Theorem}{Theorem}[section]
\newtheorem{Lemma}{Lemma}[section]
\newtheorem{Example}{Example}[section]

\newtheorem{Assumption}{Assumption}[section]
\newtheorem{Remark}{Remark}[section]
\newtheorem{Proposition}{Proposition}[section]
\numberwithin{equation}{section}
\allowdisplaybreaks

\title{A Hybrid Iterative Deep Ritz Method for Elliptic Interface Problems\thanks{The work of T. Hu is supported by the National Natural Science Foundation of China (Project 125B2022). The work of B. Jin is supported by Hong Kong RGC General Research Fund (14306824), and ANR / Hong Kong RGC Joint Research Scheme (A-CUHK402/24), NSFC / RGC Joint Research Scheme (N\_CUHK446/25) and a start-up fund, both from The Chinese University of Hong Kong. The work of Y. Xu is partially supported by the National Natural Science Foundation of China (12250013, 12261160361 and 12271367) and General Research Fund (KF2023018 and KF2024068) from Shanghai Normal University.}}

\author{Tianhao Hu\thanks{Department of Mathematics, The Chinese University of Hong Kong, Shatin, N.T., Hong Kong (email: \texttt{thhu@link.cuhk.edu.hk, b.jin@cuhk.edu.hk, fengruwang@cuhk.edu.hk})} \and Bangti Jin\footnotemark[2] \and Fengru Wang\footnotemark[2] \and Yifeng Xu\thanks{Department of Mathematics \& Scientific Computing Key Laboratory of Shanghai Universities, Shanghai Normal University, Shanghai 200234, China. ({\tt yfxu@shnu.edu.cn, yfxuma@aliyun.com})} }
\date{}

\begin{document}
\maketitle

\begin{abstract}
In this work, we propose a hybrid iterative deep Ritz method (H-IDRM) for a class of interface problems for second-order elliptic operators. It is based on a new mixed formulation of the problem and involves solving a sequence of convex minimization problems. We employ a level-set neural network architecture, featuring a level-set representation of the interface, to accommodate the piecewise smoothness of the solution and the flux. The approach involves only volumetric representations instead of duality pairing on the interface and avoids explicit interface sampling that is inconvenient for complex interface geometries. Further, we present an analysis of the method, including the errors arising from the neural network approximation, Monte Carlo approximation, iterative scheme, and penalty parameters. Numerical experiments indicate that the H-IDRM outperforms existing neural solvers on problems with high-dimensional domains, intricate interface geometries, and lower subdomain regularity.
\end{abstract}

\section{Introduction}
Interface problems arise in materials science, solid mechanics, and fluid dynamics etc., including heat conduction with discontinuous conductivities~\cite{CHAO19932021}, elasticity in heterogeneous materials~\cite{Li2006}, and Stokes flow with discontinuous viscosity across fluid-fluid interfaces~\cite{LeVeque1997Immersed}. The solutions to interface problems typically exhibit high regularity within each subdomain but suffer from singularities or internal layers across the interface \cite{Kellog:1971,Kellog:1972,HuangZou:2002,HuangZou:2007}. The inherent global low regularity  and complex interface geometries pose significant challenges to effective numerical resolution.

Let $\Omega \subset \mathbb{R}^d$ be an open bounded domain with a Lipschitz boundary, partitioned into two disjoint, connected subdomains $\Omega_1$ and $\Omega_2$ such that $\overline{\Omega} = \overline{\Omega}_1 \cup \overline{\Omega}_2$, with the interface $\Gamma = \overline{\Omega}_1 \cap \overline{\Omega}_2$. 
For a function $u$ defined on $\Omega$, let $u_k = u|_{\Omega_k}$ be its restriction to $\Omega_k$.
We consider the following linear elliptic interface problem
\begin{equation}\label{eqn:interfaceproblem}
\left\{\begin{aligned}
-\nabla \cdot (a \nabla u) + c u &= f, && \text{in } \Omega, \\
[u] = 0, \quad [a \nabla u \cdot \mathbf{n}] &= 0, && \text{on } \Gamma, \\
 u &= g, && \text{on } \partial \Omega,
\end{aligned}
\right.
\end{equation}
where $[\cdot]$ denotes the jump of a function across the interface $\Gamma$ and $\mathbf{n}$ is the unit normal vector to $\Gamma$ pointing from $\Omega_1$ to $\Omega_2$. The diffusion coefficient $a$ is piecewise constant: $a|_{ \Omega_k} = a_k>0$ for $k=1,2$, the reaction coefficient $0\leq c \in L^\infty(\Omega)$, the source term $f \in L^2(\Omega)$, and the boundary data $g\in H^{1/2}(\partial\Omega)$. Under standard ellipticity and boundedness hypotheses, problem \eqref{eqn:interfaceproblem} has a unique weak solution.

Various numerical techniques have been developed to solve interface problems, including immersed boundary method \cite{Peskin:2002}, immersed interface method \cite{Li2006}, the discontinuous Galerkin technique \cite{CaiYeZhang:2011}, matched interface and boundary method \cite{ZhouZhaoWei:2006}, finite element method \cite{ChenZou:1998,LiLinWu:2003,AdjeridGuo:2023} etc. These methods have demonstrated success in solving interface problems. However, there remain substantial challenges. Constructing meshes for irregular geometries can be expensive, especially when utilizing adaptive grids or interface-fitted meshes. The computational cost tends to escalate for high-dimensional problems due to the notorious ``the curse of dimensionality''. 

The recent advent of deep learning offers an alternative paradigm, in which solutions are approximated by neural networks (NNs) and the parameters in the NNs are determined by minimizing suitable losses. Several classes of neural PDE solvers have been proposed: strong-form methods enforce the PDE pointwise (e.g., physics-informed neural networks  (PINNs)~\cite{RAISSI2019686}), whereas weak-form methods like the deep Ritz method (DRM)~\cite{yu2018deep} are based on suitable variational principles. However, these solvers encounter fundamental difficulties when applied to elliptic interface problems. PINNs are not directly applicable since the PDE does not hold in a strong sense across the interface $\Gamma$. The DRM converges slowly due to the limited global regularity of the solution.

To overcome these limitations, several strategies have been developed, including input-augmented architectures that incorporate geometric features (e.g., signed distance functions) into the NN input~\cite{HU2022111576}, domain decomposition approaches that employ separate NNs in the subdomains to capture the solution discontinuities and enforce the interface conditions via the \(L^2(\Gamma)\) penalty terms (see, e.g., \cite{WU2022111588,doi:10.1137/24M1632309,doi:10.1137/22M1517081,SARMA2024117135,YaoGu:2023}), splitting schemes \cite{TsengLinHu:2023,FanTan:2025} that decompose the solution into a regular part and a weakly singular part, and hybrid schemes \cite{LI2025113847} that combine NN approximations with traditional finite-difference or finite-element discretizations. These approaches have shown promising performance, but still have some limitations. First, the $L^2(\Gamma)$ imposition of interface conditions is inconsistent with the natural functional settings between \(H^{\frac{1}{2}}(\Gamma)\) and \(H^{-\frac{1}{2}}(\Gamma)\). Second, uniform sampling on complex interfaces remains a highly challenging task. Thus, existing methods are largely confined to simple and predefined geometries (e.g., line segments or circles). Third, there is a lack of NN architectures that can naturally deal with interface jump conditions and piecewise smooth regularity within a single NN. Finally, for most existing methods, there is still no rigorous convergence analysis. 

In this work, we develop a novel hybrid iterative deep Ritz method (H-IDRM) for problem \eqref{eqn:interfaceproblem}. First, we recast problem \eqref{eqn:interfaceproblem} as a variant of the primal-dual mixed system (see \eqref{eqn:weakform}) in order to facilitate training: instead of solving the classic saddle-point problem, we solve a sequence of convex minimization subproblems (see \eqref{eqn:symmetric-subproblem}) that are more amenable to stable training. Second, we encode the interface geometry using a smoothed Heaviside function of the level-set variable. By Green's identity, we convert the duality pairing on the interface, naturally defined between $H^{\frac{1}{2}}(\Gamma)$ and $H^{-\frac{1}{2}}(\Gamma)$, into volumetric integrals, which can then be efficiently approximated by Monte Carlo sampling. Third, we  design a novel level-set neural network (LSNN) architecture 
to approximate the piecewise smooth solution and flux across the interface $\Gamma$; see Section \ref{subsec:LSNN}.  In Theorem \ref{thm:levelsetappro}, we also establish the approximation capacity of the LSNN architecture. Moreover, we discuss the convergence of the H-IDRM, which provides relevant theoretical underpinnings. The numerical experiments in Section \ref{SEC:NUM} show that the H-IDRM can accurately resolve elliptic interface problems in several challenging settings, e.g., high-dimensional problems, problems with complex interfaces and lower global regularity, and outperform existing neural solvers, e.g., domain-decomposition PINN and deep Ritz method. To the best of our knowledge, the H-IDRM represents the first neural solver for elliptic interface problems with a rigorous convergence analysis under a reasonable regularity assumption on the solution.

The rest of the paper is organized as follows.
In Section~\ref{SEC:HIDRM}, we develop the H-IDRM, including the mixed formulation, the LSNN architecture, and the iterative training procedure.
Section \ref{SEC:ERR} is devoted to an analysis of the H-IDRM.
In Section~\ref{SEC:NUM}, we present numerical experiments to illustrate the efficiency and accuracy of the H-IDRM, including a comparative study with several existing neural solvers. Throughout, we denote by $(\cdot,\cdot)$ the $L^{2}(\Omega)$ inner product, and by $\langle\cdot,\cdot\rangle$ the duality pairing between $H^{\frac{1}{2}}$ and its dual $H^{-\frac{1}{2}}$ on the interface $\Gamma$ or the boundary $\partial\Omega$. We use $C$, with or without subscript, to denote a generic positive constant, independent of discretization parameters (e.g., network width, sample size, penalty parameters).

\section{Hybrid iterative deep Ritz method}\label{SEC:HIDRM}

In this section, we develop the hybrid iterative deep Ritz method (H-IDRM) for problem \eqref{eqn:interfaceproblem}. It consists of three parts: (i) a new mixed formulation that avoids the saddle-point structure, (ii) a level-set neural network (LSNN) that accommodates piecewise smoothness of the solution, and (iii) an iterative deep Ritz method that solves a sequence of convex minimization subproblems for stable training. We describe the three components separately below.

\subsection{Mixed formulation}

The proposed H-IDRM is based on a mixed formulation of problem \eqref{eqn:interfaceproblem} \cite{ARNOLD1990281,2015Stabilized,MixedandHybridFiniteElementMethod}.
The mixed approach reformulates the second‑order elliptic problem as a first‑order system by introducing the flux as an additional unknown.  
Specifically, let \(\mathbf{p} = a\nabla u\) be the flux and \(H_{\rm div}(\Omega) = \{\mathbf{p} \in L^2(\Omega)^d : \nabla\cdot\mathbf{p} \in L^2(\Omega)\}\).  
Then the saddle‑point formulation reads: find \((u,\mathbf{p})\in L^2(\Omega)\times H_{\operatorname{div}}(\Omega)\) such that
\begin{equation}\label{eq:classicalmixed}
\left\{\begin{aligned}
(a^{-1}\mathbf{p}, \mathbf{q}) + (u, \nabla\cdot\mathbf{q}) &= \langle g, \mathbf{q}\cdot\mathbf{n}_{\partial\Omega} \rangle, \\
-(\nabla\cdot\mathbf{p}, v) + (c u, v) &= (f, v),
\end{aligned}\right.
\quad \forall (v,\mathbf{q})\in L^2(\Omega)\times H_{\operatorname{div}}(\Omega),
\end{equation}
where $\mathbf{n}_{\partial\Omega}$ denotes the unit outward normal to $\partial\Omega$. The first equation follows from the constitutive relation \(\mathbf{p}=a\nabla u\) and integration by parts, and the second equation enforces the original PDE weakly. This formulation approximates the solution $u$ and the flux $\mathbf{p} = a\nabla u$ simultaneously. The flux $\mathbf{p}$ is often the quantity of primary interest in continuum mechanics. Problem \eqref{eq:classicalmixed} is symmetric but indefinite, and lacks coercivity on the space $L^2(\Omega)\times H_{\operatorname{div}}(\Omega)$.
Thus, it possesses a genuine saddle-point structure, and its well-posedness hinges on the Ladyzhenskaya–Babu\v{s}ka–Brezzi inf-sup condition.
%Any discrete subspaces \(V_h\subset L^2(\Omega)\) and \(\mathbf{W}_h\subset H_{\rm div}(\Omega)\) satisfying the discrete LBB condition
%\[
%\inf_{v_h\in V_h\setminus\{0\}} \sup_{\mathbf{q}_h\in \mathbf{W}_h\setminus\{\mathbf{0}\}}
%\frac{(v_h,\nabla\cdot\mathbf{q}_h)}{\|v_h\|_{L^2(\Omega)}\|\mathbf{q}_h\|_{H_{\rm div}(\Omega)}} \ge \beta > 0,
%\]
%with \(\beta\) independent of the mesh size $h$, can be used for constructing FEM approximations.
%Well‑known compatible pairs, e.g., Raviart–Thomas elements and Brezzi–Douglas–Marini elements, satisfy the LBB condition and can deliver optimal convergence rates \cite{MixedandHybridFiniteElementMethod}. However, the method demands compatible finite element spaces to meet the LBB condition, and the resulting algebraic system is large, indefinite, and poorly conditioned, which presents substantial challenges in the linear solve, especially in the 3D case. Also constructing interface‑fitted meshes for complex geometries can be inefficient.
Thus, numerically solving the mixed formulation~\eqref{eq:classicalmixed} using NNs remains very challenging: the indefinite and non-coercive bilinear form precludes a simple, convex energy functional. Indeed, the associated min-max formulation yields highly non-convex loss landscapes, and the saddle point is not a minimum but a stationary point.

We propose a new mixed formulation by suitably incorporating a regularized term. It converts the indefinite system into a coercive one that is well suited for efficient training. 
We employ the broken Sobolev space $V = \{ u : u_k \in H^1(\Omega_k), \; k=1,2 \}$ equipped with the norm $\|u\|_V^2 = \|u_1\|_{H^1(\Omega_1)}^2 + \|u_2\|_{H^1(\Omega_2)}^2$. 
Throughout, for $u \in V$, the gradient $\nabla u$ is understood in the piecewise sense: $\nabla u|_{\Omega_k} = \nabla u_k$, and inner products involving $\nabla u$ over $\Omega$ are interpreted as $(\nabla u, \nabla v) = \sum_{k=1}^2 (\nabla u_k, \nabla v_k)_{\Omega_k}$.
Let $X=V\times H_{\rm div}(\Omega)$ be equipped with the norm
$\|(u,\mathbf{p})\|_{X}^2 = \| u\|_{V}^2  + \|\mathbf{p}\|_{L^2(\Omega)^d}^2 + \|\nabla\cdot\mathbf{p}\|_{L^2(\Omega)}^2$ and  $X'$ the dual space of $X$, with the duality pairing $\langle\cdot,\cdot\rangle_{X,X'}$.

Next we derive an equivalent formulation of problem \eqref{eqn:interfaceproblem}. In Theorem \ref{thm:elliptic}, we prove that the associated bilinear form is coercive, which facilitates the training via the iterative deep Ritz method \cite{HU2025113791}.
\begin{Proposition}
The following mixed variational formulation holds for problem \eqref{eqn:interfaceproblem}: find $(u,\mathbf{p}) \in X$ such that 
\begin{equation}\label{eqn:weakform}
\left\{
\begin{aligned}
(a\nabla u-\mathbf{p}, \nabla v)+(-\nabla\cdot \mathbf{p}+cu,v)&= (f, v), \\
(a^{-1}\mathbf{p}, \mathbf{q}) + ((1+c)^{-1}\nabla\cdot \mathbf{p} + (1+c)^{-1}u, \nabla\cdot \mathbf{q}) &= \langle g, \mathbf{q}\cdot\mathbf{n}_{\partial\Omega} \rangle-((1+c)^{-1}f, \nabla\cdot\mathbf{q}),
\end{aligned}\right.\quad \forall(v, \mathbf{q}) \in X.
\end{equation}
\end{Proposition}
\begin{proof}
In view of the defining relation  $\mathbf{p}=a\nabla u$, for any $v\in V$, we have 
$(a\nabla u-\mathbf{p},\nabla v)=0$. %\blueadd{$\sum_{k=1}^2(a_k\nabla u-\mathbf{p},\nabla v)_{\Omega_k}=0$. $v$ is piecewise $H^1$}. 
This identity and the second equation in \eqref{eq:classicalmixed} give the first equation in \eqref{eqn:weakform}:
\begin{equation}\label{eq:new-first-equation}
(a\nabla u-\mathbf{p},\nabla v)+(-\nabla\cdot\mathbf{p}+cu,v)=(f,v).
\end{equation}
%This equation naturally requires $u$ and $v$ to belong to an $H^1$-type space. 
Fix a weight $\omega \in L^\infty(\Omega)$ to be specified. % and add a weighted version of the second equation in \eqref{eq:classicalmixed} to the first one. 
For any $\mathbf{q}\in H_{\operatorname{div}}(\Omega)$, setting  $v=-\omega\,\nabla\cdot\mathbf{q}$ in the second equation of \eqref{eq:classicalmixed}, which is valid since $\nabla\cdot\mathbf{q}\in L^2(\Omega)$ and $\omega\in L^\infty(\Omega)$, gives
\begin{equation*}
(\omega\nabla\cdot\mathbf{p},\nabla\cdot\mathbf{q})-(\omega c u,\nabla\cdot\mathbf{q})=-(\omega f,\nabla\cdot\mathbf{q}).
\end{equation*}
This identity and the first equation in \eqref{eq:classicalmixed} lead to 
\[(a^{-1}\mathbf{p},\mathbf{q})+(u,\nabla\cdot\mathbf{q})+(\omega\nabla\cdot\mathbf{p},\nabla\cdot\mathbf{q})-(\omega c u,\nabla\cdot\mathbf{q})=\langle g, \mathbf{q}\cdot\mathbf{n}_{\partial\Omega} \rangle-(\omega f,\nabla\cdot\mathbf{q}),\]
or equivalently
\begin{equation}\label{eq:alpha-form}
(a^{-1}\mathbf{p},\mathbf{q})+\bigl(\omega\nabla\cdot\mathbf{p}+(1-\omega c)u,\nabla\cdot\mathbf{q}\bigr)=\langle g, \mathbf{q}\cdot\mathbf{n}_{\partial\Omega} \rangle-(\omega f,\nabla\cdot\mathbf{q}).
\end{equation}
Then setting $\omega=(1+c)^{-1}$ completes the proof of the proposition. The choice of $\omega$ ensures the coercivity of the associated bilinear form; see the proof of Theorem \ref{thm:elliptic} below for details.
\end{proof}
The system \eqref{eqn:weakform}
can be stated in an abstract form. We define the bilinear form $\mathcal{A}:X\times X\rightarrow\mathbb{R}$ and a linear functional $\mathcal{F}\in X'$ respectively by
\begin{equation}\label{def:af}
\left\{\begin{aligned}
\mathcal{A}((u,\mathbf{p}),(v,\mathbf{q})) =& (a\nabla u-\mathbf{p}, \nabla v)+(-\nabla\cdot \mathbf{p}+cu,v) + (a^{-1}\mathbf{p}, \mathbf{q}) \\
&+ \bigl((1+c)^{-1}\,\nabla\cdot\mathbf{p}+(1+c)^{-1} u,\nabla\cdot\mathbf{q}\bigr),\\
\mathcal{F}(v,\mathbf{q}) =& (f,v)+\langle g, \mathbf{q}\cdot\mathbf{n}_{\partial\Omega} \rangle-((1+c)^{-1} f, \nabla\cdot\mathbf{q}).
\end{aligned}\right.
\end{equation}
Thus~\eqref{eqn:weakform} is equivalent to finding $(u,\mathbf{p})\in X$ such that 
$$\mathcal{A}((u,\mathbf{p}),(v,\mathbf{q}))=\mathcal{F}(v,\mathbf{q}),\quad\forall (v,\mathbf{q})\in X.$$ 

Now we state the assumptions on the interface and the solution regularity.
\begin{Assumption}\label{assump:elliptic} The interface $\Gamma$ is Lipschitz, and the solution $u$ of problem~\eqref{eqn:interfaceproblem} is piecewise regular: $u_{k} \in H^{\frac{3}{2}+r}(\Omega_k)$ for $k=1,2$ and some $r > 0$.
\end{Assumption}

\begin{Theorem}\label{thm:elliptic}
Let Assumption~\ref{assump:elliptic} hold and suppose that for $k = 1,2$, 
%{\color{blue} \sout{at least one of the following conditions is satisfied: {\rm(i)} the boundary segment $\Gamma_k := \partial\Omega \cap \partial\Omega_k$ has a positive measure; {\rm(ii)}}} 
there exists $c_0 > 0$ such that each set $\{\boldsymbol{x} \in \Omega_k : c(\boldsymbol{x}) \ge c_0\}$ has a positive measure.
Then the mixed formulation~\eqref{eqn:weakform} is well-posed. 
Moreover, the solution $(u^*,\mathbf{p}^*)$ coincides with the solution pair $(u^*,a\nabla u^*)$ of problem~\eqref{eqn:interfaceproblem}.
\end{Theorem}

\begin{proof}
Since $a,c\in L^\infty(\Omega)$, the Cauchy-Schwarz inequality yields the continuity of the bilinear form $\mathcal{A}:X\times X\to \mathbb{R}$.
By taking $(v,\mathbf{q})=(u,\mathbf{p})$, we obtain
\begin{align*}\mathcal{A}((u,\mathbf{p}),(u,\mathbf{p}))
=&[(a\nabla u,\nabla u)-(\mathbf{p},\nabla u)+(a^{-1}\mathbf{p},\mathbf{p})]\\
&+[(cu,u)-((1+c)^{-1} cu,\nabla\cdot\mathbf{p})+((1+c)^{-1}\nabla\cdot\mathbf{p},\nabla\cdot\mathbf{p})]:={\rm I} + {\rm II}.
 \end{align*}
We recast the term $\rm I$ as
\[{\rm I}=\left\|\tfrac{1}{\sqrt{2}}a^{\frac12}\nabla u-\tfrac{1}{\sqrt{2}}a^{-\frac12}\mathbf{p}\right\|_{L^2(\Omega)^d}^2+\tfrac12(a\nabla u,\nabla u)+\tfrac12(a^{-1}\mathbf{p},\mathbf{p}).
\]
The term  ${\rm II}$ can be decomposed into
\[
\begin{aligned}
{\rm II} =&\left\|c(\sqrt{2+2c})^{-1}\,u-(\sqrt{2+2c})^{-1}\,\nabla\cdot\mathbf{p}
\right\|_{L^2(\Omega)}^2\\\
&+\left(
c(c+2)(2+2c)^{-1}u,u
\right) +\left(
(2+2c)^{-1}\nabla\cdot\mathbf{p},
\nabla\cdot\mathbf{p}
\right).
\end{aligned}
\]
By the elementary inequalities $(2+2c)^{-1}c(c+2)\geq 2^{-1}c$ and $(2+2c)^{-1}\geq (2+2\|c\|_{L^\infty(\Omega)})^{-1}$, we deduce \begin{align*}
\mathcal{A}((u,\mathbf{p}),(u,\mathbf{p}))
\geq C\big(\|\nabla u\|_{L^2(\Omega)^d}^2 + \|\mathbf{p}\|_{L^2(\Omega)^d}^2 + (cu,u) + \|\nabla\cdot\mathbf{p}\|_{L^2(\Omega)}^2\big).
\end{align*}
Next we prove the coercivity. %{\color{blue} \sout{Under condition (i), the Poincar\'e inequality yields
%\begin{equation*}
%\mathcal{A}((u,\mathbf{p}),(u, \mathbf{p})) \geq C\big(\|\nabla u\|_{L^2(\Omega)^d}^2 + \|\mathbf{p}\|_{L^2(\Omega)^d}^2 + \|u\|_{L^2(\Omega)}^2 + \|\nabla\cdot\mathbf{p}\|_{L^2(\Omega)}^2\big) = C\|(u,\mathbf{p})\|_X^2.
%\end{equation*}
%Under condition (ii), }}
We establish the Poincar\'e-type inequality
\begin{equation}\label{eqn:l2control}
\|u\|_{L^2(\Omega)}^2\leq C(\left\|\nabla u\right\|_{L^2(\Omega)^d}^2+(cu,u)).
\end{equation}
We prove the inequality by contradiction. Suppose that no such constant exists. Then for each $n\in\mathbb{N}$, there exists $u_n\in V$ with $\|u_n\|_{L^2(\Omega)}=1$ such that $1=\|u_{n}\|_{L^{2}(\Omega)} > n(\|\nabla u_n\|_{L^2(\Omega)^d}^2 + (cu_n,u_n))$.
Thus $\|\nabla u_n\|_{L^2(\Omega)^d}\to 0$, and by the compactness of the embedding  $V\hookrightarrow L^2(\Omega)$, a subsequence, still denoted by $(u_n)_n$, converges strongly in $L^2(\Omega)$ to some $\bar{u}$ with $\|\bar{u}\|_{L^2(\Omega)}=1$.
Since $(u_n)_n$ is also a Cauchy sequence in $V$, $\nabla\bar{u}=0$ in each $\Omega_k$ and then $\bar{u}$ is piecewise constant.
However, the convergence $(cu_n,u_n)\to(c\bar{u},\bar{u})=0$ and the positivity condition on $c$ imply that $\bar{u}=0$ in $\Omega$, which contradicts $\|\bar{u}\|_{L^2(\Omega)}=1$. This shows the inequality \eqref{eqn:l2control} and the coercivity. Meanwhile, by the continuity of the normal-trace mapping $\gamma_n: H_{\rm div}(\Omega)\to H^{-\frac{1}{2}}(\partial\Omega)$ with $\mathbf{q} \mapsto \mathbf{q}\cdot\mathbf{n}_{\partial \Omega}$ \cite[Theorem I.2.5, p.27]{GiraultRaviart:1986}, we have $\mathcal{F}\in X'$. The Lax-Milgram lemma then guarantees the well-posedness of~\eqref{eqn:weakform}. The equivalence is verified by showing that $(u^*,a\nabla u^*)$ satisfies~\eqref{eqn:weakform}.
The first equation of~\eqref{eqn:weakform} follows directly from \eqref{eqn:interfaceproblem}.
For the second equation, by integration by parts, we have
\begin{equation*}
\begin{aligned}
&(\nabla u^*, \mathbf{q}) + ((1+c)^{-1}\nabla\cdot (a\nabla u^*) + (1+c)^{-1}u^*, \nabla\cdot \mathbf{q})\\
=&(\nabla u^*, \mathbf{q})+(u^*,\nabla\cdot\mathbf{q})+ ((1+c)^{-1}\nabla\cdot (a\nabla u^*) - (1+c)^{-1}cu^*, \nabla\cdot \mathbf{q})\\
=&\langle g, \mathbf{q}\cdot\mathbf{n}_{\partial\Omega} \rangle-((1+c)^{-1}f, \nabla\cdot\mathbf{q}).
\end{aligned}
\end{equation*}
So the pair $(u^*,a\nabla u^*)$ solves~the system \eqref{eqn:weakform}. 
The proof is completed by the uniqueness of the solution to the system \eqref{eqn:weakform}.
\end{proof}

\begin{Remark}\label{remark:plusepsilon}
In the degenerate case $c=0$, we can employ a regularized mixed formulation with a small $\epsilon>0$:
%\bluedel{To restore the coercivity, with a small $\epsilon>0$, we modify the mixed form to}
\begin{equation}
\left\{
\begin{aligned}
(a\nabla u-\mathbf{p}, \nabla v)+(-\nabla\cdot \mathbf{p}+\epsilon u,v)&= (f, v), \\
(a^{-1}\mathbf{p}, \mathbf{q}) + (\nabla\cdot\mathbf{p} + u, \nabla\cdot\mathbf{q}) &= -(f, \nabla\cdot\mathbf{q}).
\end{aligned}
\right.
\end{equation}
The effectiveness of the approach will be validated numerically in Section~\ref{SEC:NUM}.
\end{Remark}

% \begin{Remark}\label{rmk:nonzero}
% %Throughout the theoretical analysis, we have assumed the homogeneous Dirichlet boundary condition $g=0$ so that the trial space for the primal variable is the linear space $V_0$, which allows the direct application of the Lax-Milgram theorem. 
% We can also derive a mixed formulation for a nonzero boundary condition $u = g$ on $\partial\Omega$. Indeed, using the relation $\mathbf{p} = a\nabla u$, by multiplying by $a^{-1}\mathbf{q}$, integrating over the domain $\Omega$ and the divergence theorem, we get
% \[
% (a^{-1}\mathbf{p}, \mathbf{q}) + (u, \nabla\cdot\mathbf{q}) = \langle g, \mathbf{q}\cdot\mathbf{n} \rangle,
% \]
% Then the second equation in the system \eqref{eqn:weakform} reads
% \begin{equation}
% (a^{-1}\mathbf{p}, \mathbf{q}) + ((1+c)^{-1}\nabla\cdot \mathbf{p} + (1+c)^{-1}u, \nabla\cdot \mathbf{q}) = \langle g, \mathbf{q}\cdot\mathbf{n} \rangle - ((1+c)^{-1}f, \nabla\cdot\mathbf{q}).
% \end{equation}
% Thus, the linear functional $\mathcal{F}$ in \eqref{def:af} involves an additional boundary term $\langle g, \mathbf{q}\cdot\mathbf{n} \rangle$.
% \end{Remark}

The mixed formulation \eqref{eqn:weakform} enjoys several distinct features. It avoids the $H^{\frac{1}{2}}(\Gamma)$-$H^{-\frac{1}{2}}(\Gamma)$ duality pairing, which ensures the consistency between the theoretical formulation and practical implementation. Moreover, the monotonicity of the associated operator $\mathcal{A}$ stabilizes the NN training; see Section \ref{subsec:IDRM} for further discussion. 

\subsection{Level-set neural network}\label{subsec:LSNN}

Existing neural PDE solvers often employ a fully connected NN (FCNN) $u_\theta: \mathbb{R}^d \to \mathbb{R}$ to approximate PDE solutions. It is defined recursively by 
\begin{equation*}
\boldsymbol{y}_0(\boldsymbol{x}) = \boldsymbol{x}, \quad 
\boldsymbol{y}_\ell(\boldsymbol{x}) = \sigma(A_\ell \boldsymbol{y}_{\ell-1}(\boldsymbol{x}) + b_\ell), \quad \ell=1,\ldots,{D}-1,
\end{equation*}
and $u_\theta(\boldsymbol{x}) = A_{ D}\boldsymbol{y}_{D-1}(\boldsymbol{x}) + b_{D}$, where $A_\ell \in \mathbb{R}^{N_\ell \times N_{\ell-1}}$ and $b_\ell \in \mathbb{R}^{N_\ell}$ are trainable parameters, collectively denoted by $\theta$, and $\sigma: \mathbb{R} \to \mathbb{R}$ is a smooth nonpolynomial function which applies componentwise to a vector. We denote by $\mathcal{N}_\sigma(D,W,B)$ the class of NNs with depth $D$, width $W = \max_{1\le \ell \le D} N_\ell$, and parameter bound $B \ge \max_{\ell} \{ \|A_\ell\|_{\infty}, \|b_\ell\|_{\infty} \}$. However, standard FCNNs are ill-suited to interface problems: they cannot resolve the disparity between high regularity within the subdomains and lower regularity across the interface $\Gamma$. 

This observation motivates developing a NN architecture that accommodates the piecewise regularity of the solution $u$ and the flux $\mathbf{p}$. We introduce level-set NNs (LSNNs) which use a level-set function $\phi$ to represent the interface $\Gamma$:
$\Omega_1 = \{\boldsymbol{x} \in \Omega : \phi(\boldsymbol{x}) > 0\}$, $
\Omega_2 = \{\boldsymbol{x} \in \Omega : \phi(\boldsymbol{x}) < 0\}$ and $\Gamma = \{\boldsymbol{x} \in \Omega : \phi(\boldsymbol{x}) = 0\}.$
The LSNN architecture is defined by
\begin{equation*}
u_{\theta}(\boldsymbol{x}) = \tilde{u}_{\tilde\theta}\left(\boldsymbol{x}, t(\boldsymbol{x})\right), \quad \mbox{with } t(\boldsymbol{x})=(1+e^{\alpha \phi(\boldsymbol{x})})^{-1},
\end{equation*}
where $\tilde{u}_{\tilde\theta}: \mathbb{R}^{d+1} \to \mathbb{R}$ is a standard FCNN, and $\alpha>0$. We denote the LSNN architecture by $\mathcal{N}_{\sigma}(D,W,B;\alpha)$ ($\mathcal{N}_\sigma$ for short), and the vectorial version with $d$ outputs by $\mathcal{N}_{\sigma}^d(D,W,B;\alpha)$ ($\mathcal{N}_{\sigma}^d$ for short), which is employed to approximate the flux $\mathbf{p}$.

The construction incorporates the level-set function $\phi(\boldsymbol{x})$ through a smoothed Heaviside transformation $t(\boldsymbol{x})$, and encodes the interface geometry into input features.
The parameter $\alpha>0$ governs the steepness of the transition across the interface $\Gamma$; a large $\alpha$ implies a sharper transition layer. It allows the LSNN to resolve sharp gradients while preserving the global smoothness. The LSNN provides parametric separation by supplying distinct input signals in $\Omega_1$ ($\phi>0$) and $\Omega_2$ ($\phi<0$), which facilitates specialized representations within each subdomain while maintaining solution continuity across the interface $\Gamma$. The implicit interface representation accommodates complex geometries and topological changes without explicit meshing. The construction of the LSNN resembles enriched approximation spaces in traditional interface-fitted methods \cite{Zilian2009,Osher2001LevelSet}. 

\subsection{Iterative deep Ritz method}\label{subsec:IDRM}

Now we provide details on the iterative optimization procedure and the empirical loss. The mixed formulation~\eqref{eqn:weakform} is amenable to the iterative deep Ritz method (IDRM) since the bilinear form $\mathcal{A}$ defined in~\eqref{def:af} is monotone~\cite{HU2025113791} in the sense that  
$$\mathcal{A}((u-v,\mathbf{p}-\mathbf{q}),(u-v,\mathbf{p}-\mathbf{q})) \geq C\|(u-v,\mathbf{p}-\mathbf{q})\|_X^2,\quad \forall (u,\mathbf{p}), (v,\mathbf{q})\in X.$$ 
See the proof of Theorem \ref{thm:elliptic} for the estimate.

The bilinear form $\mathcal{A}$ is generally not symmetric, i.e., $\mathcal{A}((u,\mathbf{p}),(v,\mathbf{q})) \neq \mathcal{A}((v,\mathbf{q}),(u,\mathbf{p}))$, due to the presence of the coupling terms $(\mathbf{p}, \nabla v)$, $(\nabla\cdot\mathbf{p}, v)$, and $(u, \nabla\cdot\mathbf{q})$. Thus one cannot directly minimize the energy $\frac{1}{2}\mathcal{A}((u,\mathbf{p}),(u,\mathbf{p})) - \mathcal{F}(u,\mathbf{p})$.
This issue can be overcome by the IDRM \cite{HU2025113791}.
The IDRM decomposes $\mathcal{A}$ into symmetric and antisymmetric parts, and treats the antisymmetric part iteratively. Specifically, let  $\mathcal{A}_s((u,\mathbf{p}),(v,\mathbf{q})) = \frac{1}{2}\bigl[\mathcal{A}((u,\mathbf{p}),(v,\mathbf{q})) + \mathcal{A}((v,\mathbf{q}),(u,\mathbf{p}))\bigr]$ be the symmetric part of $\mathcal{A}$, which inherits the coercivity of $\mathcal{A}$. Then at iteration $k$, we solve the symmetric problem
\begin{equation}\label{eqn:weak-symmetric}
\mathcal{A}_s((u_{k+1}-u_k,\mathbf{p}_{k+1}-\mathbf{p}_k),(v,\mathbf{q})) = \mathcal{F}(v,\mathbf{q}) - \mathcal{A}((u_k,\mathbf{p}_k),(v,\mathbf{q})), \quad \forall (v,\mathbf{q}) \in X.
\end{equation}
Since $\mathcal{A}_s$ is symmetric and coercive, problem \eqref{eqn:weak-symmetric} is equivalent to the following convex minimization:
\begin{equation}\label{eqn:symmetric-subproblem}
(u_{k+1},\mathbf{p}_{k+1}) = \operatorname*{argmin}_{(u,\mathbf{p})\in X} \widetilde{L}_k(u,\mathbf{p}) 
\end{equation}
with the functional $\widetilde{L}_k(u,\mathbf{p})$ given by %\blueadd{(Why not use $\mathcal{A}_s$ in $\widetilde{L}_k$?)}
\begin{equation*} \widetilde{L}_k(u,\mathbf{p}) = \tfrac{1}{2}\mathcal{A}((\delta u,\delta \mathbf{p}),(\delta u,\delta \mathbf{p})) + \mathcal{A}((u_k,\mathbf{p}_k),(\delta u,\delta\mathbf{p})) - \mathcal{F}(\delta u,\delta \mathbf{p}),
\end{equation*}
with $(\delta u,\delta\mathbf{p})=(u-u_k,\mathbf{p}-\mathbf{p}_k)$.
The first term is quadratic in the increment $(\delta u, \delta \mathbf{p})$, and the second term incorporates the asymmetry of $\mathcal{A}$. This procedure can be viewed as a fixed-point iteration: at convergence, the increment vanishes and the identity $\mathcal{A}((u^*,\mathbf{p}^*),(v,\mathbf{q})) = \mathcal{F}(v,\mathbf{q})$ is recovered. Note that \eqref{eqn:symmetric-subproblem} is precisely a deep Ritz step.

In practice, we modify $\widetilde{L}_k(u,\mathbf{p})$ to improve the training stability. First, we use a step-size $\lambda > 0$ to control the magnitude of each update. Then the loss is evaluated at the scaled increment $\lambda(\delta u, \delta \mathbf{p})$: a larger $\lambda$ penalizes large deviations from the current iterate more strongly, effectively restricting the update to a local neighborhood. Second, we add a stabilization term $\sigma_s\|a\nabla u - \mathbf{p}\|^2$ to enforce the constitutive relation $\mathbf{p} = a\nabla u$. These enhancements lead to a surrogate loss
\begin{equation}\label{eqn:surrogate}
L_k(u,\mathbf{p}) = \tfrac{\lambda^2}{2}\mathcal{A}((\delta u,\delta \mathbf{p}),(\delta u,\delta \mathbf{p})) + \lambda\,\mathcal{A}((u_k,\mathbf{p}_k),(\delta u,\delta \mathbf{p})) - \lambda\,\mathcal{F}(\delta u,\delta \mathbf{p})+ \sigma_s \|a\nabla u - \mathbf{p}\|_{L^2(\Omega)^d}^2,
\end{equation}
%{\color{blue} \sout{where $\Phi_{\sigma_s}$ denotes the stabilized energy
%\begin{equation*}
%\Phi_{\sigma_s}(u,\mathbf{p}) = \frac{1}{2}\,\mathcal{A}((u,\mathbf{p}),(u,\mathbf{p})) + \sigma_s \|a\nabla u - \mathbf{p}\|_{L^2(\Omega)^d}^2.
%\end{equation*}}}
Note that each subproblem~\eqref{eqn:surrogate} is convex in $(u,\mathbf{p})$, and thus well suited for gradient-based training. In sum, the IDRM transforms the asymmetric problem into a sequence of convex subproblems, thereby ensuring stable and robust convergence; see Theorem~\ref{thm:modelconv} below.

The mixed formulation \eqref{eqn:weakform} weakly imposes the Dirichlet boundary datum $g$ via the term $\langle g,\mathbf{q}\cdot\mathbf{n}_{\partial\Omega}\rangle$ in $\mathcal{F}$, which may be insufficient to achieve the desired accuracy. Thus, in practice, %In practice, we the accurate enforcement via the dual pairing with $\mathbf{p}$ in \eqref{eqn:surrogate} relies on good approximation of the flux variable in numerical simulation. satisfied to the desired accuracy.} 
we further enforce the Dirichlet boundary condition $u|_{\partial\Omega}=g$ by adding the penalty $\frac{\sigma}{2}\|u-g\|_{L^2(\partial\Omega)}^2$ to the loss $L_k$.
The complete iterative procedure is given in Algorithm~\ref{alg:alg1}.

\begin{algorithm}[hbt!]
\caption{Hybrid IDRM for the elliptic interface problem.}
\label{alg:alg1}
\begin{algorithmic}[1]
\State Initialize the NNs $u_{0}:\Omega \rightarrow \mathbb{R}$, $\mathbf{p}_0:\Omega \rightarrow \mathbb{R}^d$, set the step size $\lambda$ and tolerances $\epsilon_{\rm M}$ and $\epsilon_{\rm tol}$, and set $k=0$.
\Do
\State Define the surrogate loss~\eqref{eqn:surrogate}.
\State Find $(u_{k+1},\mathbf{p}_{k+1})\in \mathcal{N}_{\sigma}\times\mathcal{N}_{\sigma}^d$ such that
\begin{equation}\label{eqn:learning-error}
{L}_{k}(u_{k+1},\mathbf{p}_{k+1})\leq \min_{(u,\mathbf{p})\in X}{L}_{k}(u,\mathbf{p})+\epsilon_{\rm M}.
 \end{equation}
\State $k=k+1$.
\doWhile{$L_{k-1}(u_{k},\mathbf{p}_{k})> \epsilon_{\rm tol}$.}
\end{algorithmic}
\end{algorithm}

In practice, we discretize the integrals in $L_k(u,\mathbf{p})$ by the Monte Carlo method. Let $U(\Omega)$ and $U(\partial\Omega)$ denote the uniform distributions over $\Omega$ and $\partial\Omega$, respectively. We draw $N_d$ independent and identically distributed (i.i.d.) interior samples $\mathbb{X} = \{X_i\}_{i=1}^{N_d} \sim U(\Omega)$ and $N_b$ i.i.d. boundary samples $\mathbb{Y} = \{Y_j\}_{j=1}^{N_b} \sim U(\partial\Omega)$.
Then the empirical surrogate loss $\widehat{L}_k(u,\mathbf{p})$ is given by
\begin{equation*}
    \widehat{L}_{k}(u,\mathbf{p})=\frac{\lambda^2}{2}\widehat{\mathcal{A}}((\delta u,\delta \mathbf{p}), (\delta u,\delta \mathbf{p}))+\lambda\widehat{\mathcal{A}}((u_k,\mathbf{p}_k),(\delta u,\delta \mathbf{p}))-\lambda\widehat{\mathcal{F}}(\delta u,\delta \mathbf{p})    +\sigma_s\frac{|\Omega|}{N_d}\sum_{i=1}^{N_d}|a\nabla u-\mathbf{p}|^2(X_i),
\end{equation*}
where the discrete bilinear form $\widehat{\mathcal{A}}$ and linear functional $\widehat{\mathcal{F}}$ are defined respectively by
\begin{align}
\widehat{\mathcal{A}}((u,\mathbf{p}),(v,\mathbf{q})) =& \frac{|\Omega|}{N_d}\sum_{i=1}^{N_d}\Big[(a\nabla u-\mathbf{p})\cdot\nabla v + (-\nabla\cdot\mathbf{p}+c u)v \\
&+ a^{-1}\mathbf{p}\cdot\mathbf{q} + (1+c)^{-1}(\nabla\cdot\mathbf{p}+u)\nabla\cdot\mathbf{q}\Big](X_i), \nonumber \\
\widehat{\mathcal{F}}(v,\mathbf{q}) =& \frac{|\Omega|}{N_d}\sum_{i=1}^{N_d}\big[f v-(1+c)^{-1}f\nabla\cdot\mathbf{q}\big](X_i) + \frac{|\partial\Omega|}{N_b}\sum_{j=1}^{N_b} g(Y_j)\,\mathbf{q}(Y_j)\cdot\mathbf{n}(Y_j).
\end{align}

We minimize the empirical loss $\widehat{L}_k$ over the NN classes:
\begin{equation}
(\widehat{\theta}_{k+1},\widehat{\eta}_{k+1}) = \operatorname*{argmin}_{\theta,\eta}\widehat{L}_{k}\,\big(u_{\theta},\mathbf{p}_{\eta}\big),
\end{equation}
and denote the resulting approximation at iteration $k$ by $ (u_{\widehat{\theta}_{k+1}},\mathbf{p}_{\widehat{\eta}_{k+1}})\in \mathcal{N}_\sigma\times\mathcal{N}_\sigma^d$.

\section{Convergence analysis}\label{SEC:ERR}
In this section, we discuss the convergence of Algorithm~\ref{alg:alg1} and provide bounds on the learning error $\epsilon_{\rm M}$ in \eqref{eqn:learning-error}.
At each iteration $k$, the algorithm minimizes the surrogate loss $L_k(u,\mathbf{p})$ over the NN classes. We denote the exact minimizer of $L_k$ over the space $X$ by
\[
(u_{k+1}^*,\mathbf{p}_{k+1}^*) = \operatorname*{argmin}_{(u,\mathbf{p})\in X} L_k(u,\mathbf{p}).
\]
The NN approximation $(u_{\widehat{\theta}_{k+1}}, \mathbf{p}_{\widehat{\eta}_{k+1}})$ satisfies
\[
L_k(u_{\widehat{\theta}_{k+1}}, \mathbf{p}_{\widehat{\eta}_{k+1}}) \le L_k(u_{k+1}^*,\mathbf{p}_{k+1}^*) + \epsilon_{\rm M}.
\]
The monotonicity of the operator $\mathcal{A}$ ensures that the associated convergence analysis of IDRM in \cite{HU2025113791} carries over to the mixed formulation \eqref{eqn:weakform} directly. Theorem~\ref{thm:modelconv} states the local linear convergence of $(u_k,\mathbf{p}_k)$ toward $(u^*,\mathbf{p}^*)$ during an initial phase, given a sufficiently small step size and a tolerance $\epsilon_{\rm M}$.

\begin{Theorem}\label{thm:modelconv}
Let $(u^*, \mathbf{p}^*)$ be the solution of problem~\eqref{eqn:interfaceproblem}, and let $c_{\rm s}$ and $c_{\rm t}$ be two constants (see the proof of \cite[Theorem 4.1]{HU2025113791}).  
Suppose that the sequence $\{(u_k,\mathbf{p}_k)\}_{k=0}^\infty$ is generated by Algorithm~\ref{alg:alg1} with a constant step size $\lambda_k = c_{\rm s}$.  
Fix $\eta\in(0,1)$ and define
\[
k^{*} = \max\bigl\{k\in\mathbb{N} : -L_j(u_{j+1}^*,\mathbf{p}_{j+1}^*) \ge 2c_{\rm t}\eta^{-1}\epsilon_{\rm M}^{1/2},\ \forall 0\le j\le k\bigr\}.
\]
Then there exists $q\in (0,1)$ such that for all $0 \le k \le k^{*}$,
\[
\|u_k - u^*\|_{V} + \|\mathbf{p}_k - \mathbf{p}^*\|_{H_{\rm div}(\Omega)} \le (\eta q)^{k}\bigl(\|u_0 - u^*\|_{V} + \|\mathbf{p}_0 - \mathbf{p}^*\|_{H_{\rm div}(\Omega)}\bigr).
\]
\end{Theorem}
\begin{Remark}
    In \cite[Theorem 4.1]{HU2025113791}, the boundedness of $\{(u_k,\mathbf{p}_k)\}_{k=0}^\infty$ is required for a general Banach space. In a Hilbert space setting, the assumption is not needed.
\end{Remark}

Theorem~\ref{thm:modelconv} shows that the learning error $\epsilon_{\rm M}$ governs the convergence of Algorithm~\ref{alg:alg1}. In the standard error analysis, the learning error $\epsilon_{\rm M}$ can be decomposed into approximation, statistical, and optimization errors~\cite{10.1093/imanum/draf129, HU2025113791,Jiao2022Rate,JMLR:v26:24-1258}. The optimization error is controlled by the stopping criterion of Algorithm~\ref{alg:alg1} and is absorbed into $\epsilon_{\rm M}$. We bound the approximation and statistical errors below.

\begin{Lemma}\label{lem:decomp}
Let $(u_{k+1}^*,\mathbf{p}_{k+1}^*)$ be the unique minimizer of the surrogate loss $L_k(u,\mathbf{p})$ over $X$. For the learning error $\epsilon_{\rm M}$, we have
\begin{align*}
\mathbb{E}_{\mathbb{X},\mathbb{Y}}\left[\epsilon_{\rm M}\right]\le& c\Big(\underbrace{ \inf_{u_{\theta} \in \mathcal{N}_\sigma}\|u_{\theta}-u_{k+1}^*\|_{V}+\inf_{\mathbf{p}_{\eta} \in \mathcal{N}_\sigma^d}\|\mathbf{p}_{\eta}-\mathbf{p}_{k+1}^*\|_{H_{\rm div} (\Omega)}}_{\mathcal{E}_{\rm app}}\\
&+\underbrace{\mathbb{E}_{\mathbb{X},\mathbb{Y}}\sup _{(u_{\theta},\mathbf{p}_{\eta})\in \mathcal{N}_\sigma\times\mathcal{N}_\sigma^d}\left|L_{k}(u_{\theta},\mathbf{p}_{\eta})-\widehat{L}_{k}(u_{\theta},\mathbf{p}_{\eta})\right|}_{\mathcal{E}_{\rm stat}}\Big).
\end{align*}
\end{Lemma}

First we bound the approximation error $\mathcal{E}_{\rm app}$ in terms of the NN architecture.
\begin{Assumption}\label{assump:levelset}
There exists a level-set function $\phi \in C^\infty(\mathbb{R}^d)$ satisfying
{\rm(i)} $\nabla \phi(\boldsymbol{x}) \neq \boldsymbol{0}$ for all $\boldsymbol{x} \in \overline{\Omega}$ and
 {\rm(ii)} there exists $c =c(\Omega)> 0$ such that for small enough $\epsilon > 0$,
$\tfrac{\epsilon}{c}\leq \operatorname{meas}\big(\{\boldsymbol{x} \in \Omega : |\phi(\boldsymbol{x})| < \epsilon\}\big) \le c\epsilon$.
\end{Assumption}

Condition (ii) in Assumption \ref{assump:levelset} ensures that the measure of the transition layer $\{\boldsymbol{x}\in\Omega: |\phi(\boldsymbol{x})|<\epsilon\}$ scales linearly with respect to its thickness $\epsilon$. The linear scaling is crucial for controlling the approximation error concentrated near the interface $\Gamma$, since it allows bounding the contribution of the layer uniformly as $\epsilon \to 0$. We have the following approximation result. The proof is given in Appendix \ref{append}.
\begin{Theorem}\label{thm:levelsetappro}
Let Assumptions~\ref{assump:elliptic} and~\ref{assump:levelset} hold.
Then there exists an LSNN $u_{\theta}\in\mathcal{N}_{\sigma}(C\log(d + 2 + r),W, W^{\frac{9}{2}+\frac{2r+2}{d+1}};\tanh)$, such that for any $\nu > 0$, and $W$ large enough, with $\gamma=\min\{\frac{2r+1}{2d},\frac{1}{2}\}$, and $C=C(d,r,\nu,\Omega,u^*,\|\phi\|_{C^1(\Omega)})$, the following error bound holds
\begin{equation}\label{eqn:approximation}
\| u_{\theta} - u^* \|_V \leq C W^{-\frac{\gamma(r-\nu)}{(d+1)(1+\gamma r+1.5\gamma)}}.
\end{equation}
\end{Theorem}
The error bound in Theorem~\ref{thm:levelsetappro} is stated for $u$ only. The proof yields a similar estimate in the $H_{\rm div}(\Omega)$ norm. Together, these two estimates provide a bound on the approximation error $\mathcal{E}_{\rm app}$. The statistical error $\mathcal{E}_{\rm stat}$ follows from an argument using Rademacher complexity \cite[Proposition 6.4]{HU2025113791}~\cite{Bartlett2003RademacherAG}.

\begin{Proposition}\label{prop:sta}
Suppose $\Omega \subset (-1,1)^d$ and let $N_d, N_b = \mathcal{O}(N)$. 
Then the statistical error $\mathcal{E}_{\rm stat}$ satisfies
\begin{equation*}
\mathcal{E}_{\rm stat}
\leq C(\Omega,a,c,d,\lambda)DW^{\max\{3D-2,4\}}B^{3D}\sqrt{\frac{\log \left((d+1) DWB N\right)}{N}},
\end{equation*}
\end{Proposition}

\section{Numerical experiments and discussions}\label{SEC:NUM}

Now we present numerical results from several experiments to illustrate the accuracy of the H-IDRM. 

\subsection{Experimental setup}\label{subsec:setup}

First we describe the experimental setup. Throughout, the solution $u$ and the flux $\mathbf{p}$ are parameterized by the LSNNs in Section~\ref{subsec:LSNN}, with the $\tanh$ activation and the scaling parameter $\alpha = 10$. We draw $N_d = 10{,}000$ and $N_b = 5{,}000$ samples uniformly from $\Omega$ and $\partial\Omega$, respectively, to form the empirical surrogate loss $\widehat{L}_k$ in Algorithm~\ref{alg:alg1}. The Dirichlet boundary condition is enforced by adding an $L^2(\partial\Omega)$ penalty with $\sigma = 100$.
Algorithm~\ref{alg:alg1} employs the Adam optimizer~\cite{KingmaBa:2015} for 50,000 epochs. %\bluedel{, with the learning rates for $u$ and $\mathbf{p}$ individually tuned.}
The IDRM step size $\lambda$ is fixed at $ 1$, and the stability weight is $\sigma_s = 1$.
The penalty parameter $\epsilon$ in Remark \ref{remark:plusepsilon} is set to $10^{-3}$ when $c=0$. {We present} the two relative errors $e_u = \|u^* - u_{\theta}\|_{L^2(\Omega)}/\|u^*\|_{L^2(\Omega)}$ and $e_{\mathbf{p}} = \|\mathbf{p}^* - \mathbf{p}_{\eta}\|_{L^2(\Omega)^d}/\|\mathbf{p}^*\|_{L^2(\Omega)^d}$ to compare the H-IDRM with three neural PDE solvers: %\blueadd{(Comments: why not use DRM with LSNN and DD-PINN with LSNN for comparison?) The DD-PINN already augments the network input with a subdomain indicator $z=\pm1$, which plays a role analogous to the level-set input $t(\boldsymbol{x})$ of the LSNN; hence equipping the DD-PINN with the LSNN would only modify its input feature. Similarly, the DRM loss remains a non-convex energy functional with the discontinuous coefficient $a$, and the training instability of the DRM is a consequence of the formulation rather than of the architecture.}
\begin{itemize}
\item[(i)] \textbf{Deep Ritz method (DRM)}: The loss function is given by
$$L_{{\rm DRM}}(u)=\frac{1}{2}(a\nabla u,\nabla u)+\frac{1}{2}(cu,u)-(f,u)+\frac{\rho_1}{2}\|u-g\|_{L^2(\partial\Omega)}^2$$
where $\rho_1$ is a penalty parameter. We approximate the solution $u$ with the standard FCNN.
\item[(ii)] \textbf{Domain-decomposition PINN (DD-PINN)}: The loss function is given by
$$\begin{aligned}
L_{{\rm DDPINN}}(u_1,u_2)=&\|-\nabla \cdot (a_1 \nabla u_1) +c u_1-f\|_{L^2(\Omega_1)}^2+\|-\nabla \cdot (a_2 \nabla u_2) +c u_2-f\|_{L^2(\Omega_2)}^2\\&+\rho_2\|u_1-u_2\|_{L^2(\Gamma)}^2
+\rho_3\|a_1\nabla u_1\cdot\mathbf{n}_1-a_2\nabla u_2\cdot\mathbf{n}_2\|_{L^2(\Gamma)}^2+\rho_4\|u-g\|_{L^2(\partial\Omega)}^2
\end{aligned}$$
where $\rho_2$, $\rho_3$, and $\rho_4$ are penalty parameters.
%$\mathbf{n}_1$ and $\mathbf{n}_2$ are unit outer vectors with respect to $\Omega_1$ and $\Omega_2$.
The solution is represented via an augmented function $u_{\theta}(\mathbf{x},z)$ (with $z=\pm1$ indicating subdomains), approximated by a single NN.
\item[(iii)] \textbf{H-IDRM without LSNN}: We approximate $u$ and $\mathbf{p}$ directly by standard FCNNs through the loss \eqref{eqn:surrogate}. This shows the impact of the LSNN architecture on the overall accuracy.
\end{itemize}
In the DD-PINN and DRM baselines, the flux $\mathbf{p}$ is obtained via the relation $\mathbf{p} = a\nabla u_\theta$.
All baselines use comparable network capacities and training budgets.
Parameters in the NN architecture and training are listed in Table~\ref{table:para}. The Python source code used to generate the results is publicly available at \url{https://github.com/hhjc-web/Hybrid-IDRM-interface}.

\begin{table}[hbt!]
\centering
\begin{threeparttable}
\caption{\label{table:para} Parameters of NNs in our numerical examples.}
\centering
\begin{tabular}[5pt]{c|c|c|c|c}
\toprule
Method & H-IDRM & \makecell{H-IDRM\\(w/o LSNN)} & DD-PINN & DRM \\
\midrule
$u_\theta$ & $(d+1)$-16-32-64-32-16-1 & $d$-16-32-64-32-16-1 & $(d+1)$-16-32-64-32-16-1 & $d$-16-32-16-1 \\
\hline
$\mathbf{p}_\eta$ & $(d+1)$-16-64-16-$d$ & $d$-16-64-16-$d$ & -- & -- \\
\hline
${\rm lr}_u$ & $1\times 10^{-3}$ & $1\times 10^{-3}$ & $5\times 10^{-4}$ & $5\times 10^{-3}$ \\
\hline
${\rm lr}_\mathbf{p}$ & $1\times 10^{-3}$ & $1\times 10^{-3}$ & -- & -- \\
%\hline
%\makecell{Total\\iterations} & 50k & 50k & 50k & 50k \\
\bottomrule
\end{tabular}
\end{threeparttable}
\end{table}

\subsection{Numerical results and discussions}
Now we present several numerical examples to illustrate the performance of the H-IDRM. First we examine a 2D problem with a circular inclusion that lacks a closed-form solution. The notation $\chi_S$ denotes the characteristic function of a set $S$.
\begin{Example}\label{exam1}
Let $\Omega = (-1,1)^2$.
Consider problem~\eqref{eqn:interfaceproblem} with the subdomains defined by a shifted circular interface: $\Omega_1 = \{\boldsymbol{x} \in \Omega : (x_1 - 0.2)^2 + (x_2 - 0.3)^2 > 0.3^2\}$ and $\Omega_2 = \{\boldsymbol{x} \in \Omega : (x_1 - 0.2)^2 + (x_2 - 0.3)^2 < 0.3^2\}$.
The problem data are $a(\boldsymbol{x}) = \chi_{\Omega_1}(\boldsymbol{x})+0.1\chi_{\Omega_2}(\boldsymbol{x})$, $c(\boldsymbol{x}) = 1$, $f(\boldsymbol{x})\equiv 1$, and $g(\boldsymbol{x})=0$.
\end{Example}

The level-set function $\phi$ is taken to be $\phi(\boldsymbol{x}) = (x_1-0.2)^2 + (x_2-0.3)^2 - 0.3^2$. Fig.~\ref{fig:exam1} shows the reference solution, the H-IDRM prediction, and the pointwise absolute error. The reference solution is obtained by the standard linear FEM (with $2{,}000{,}000$ elements).
The H-IDRM attains better accuracy than the DD-PINN and DRM; see the relative $L^2(\Omega)$ errors of $u$ in Table~\ref{table:exam1}. The DD-PINN relies on the PDE residual in the strong form and thus requires sufficient regularity of the solution $u$. However, the exact solution $u$ does not satisfy the requirement, which leads to degraded accuracy. Also, the DRM does not adequately treat the low regularity across the interface $\Gamma$, which again leads to poor performance. %Since the FEM reference does not provide an accurate flux approximation, we don't compute $e_{\mathbf{p}}$.
The results of the H-IDRM indicate that the physical quantity of interest is well predicted by the LSNN using \eqref{eqn:surrogate} associated with the mixed formulation \eqref{eqn:weakform}.

\begin{figure}[hbt!]
\centering\setlength{\tabcolsep}{2pt}
\begin{tabular}{ccc}
\includegraphics[height=3cm]{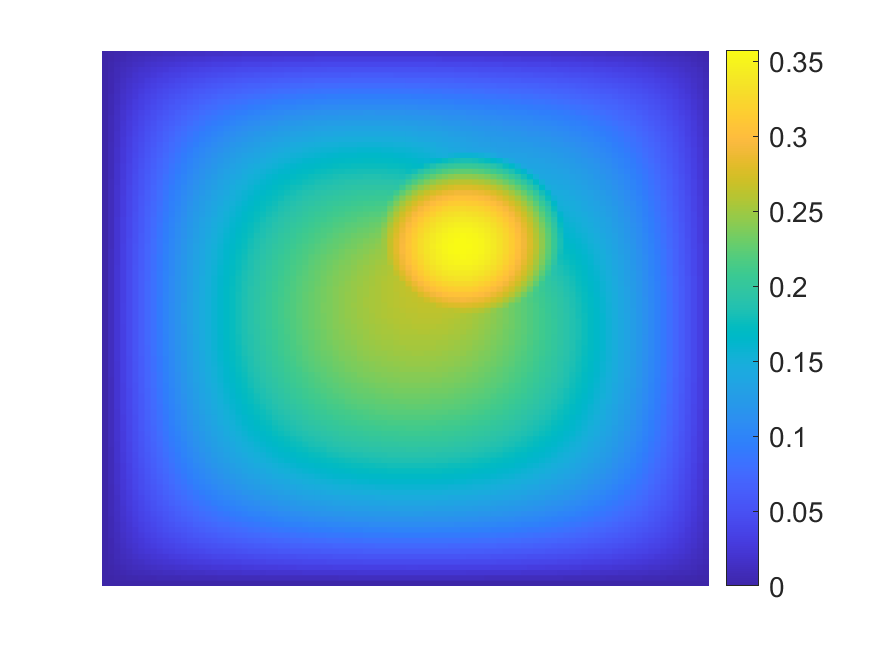} & \includegraphics[height=3cm]{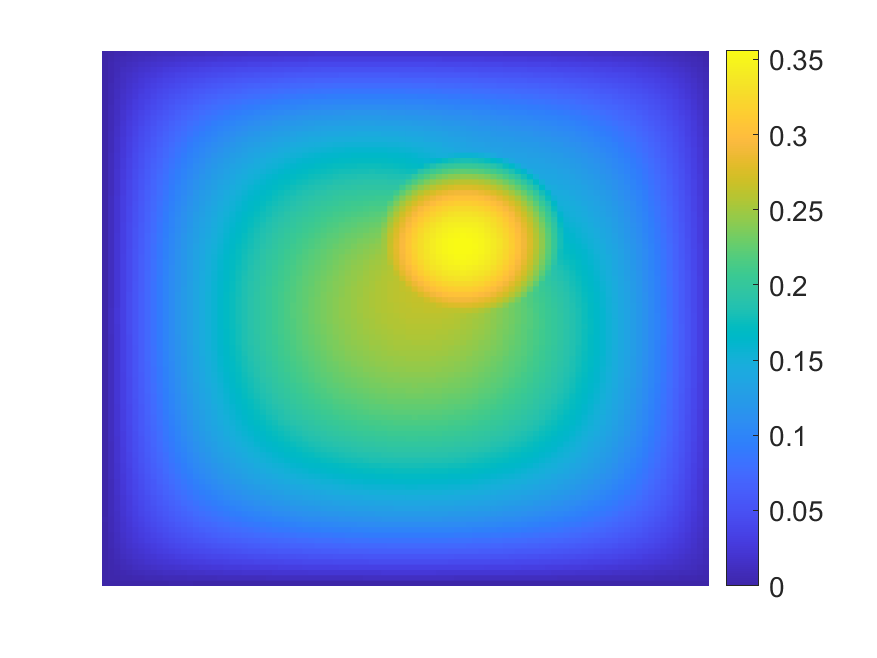} & \includegraphics[height=3cm] {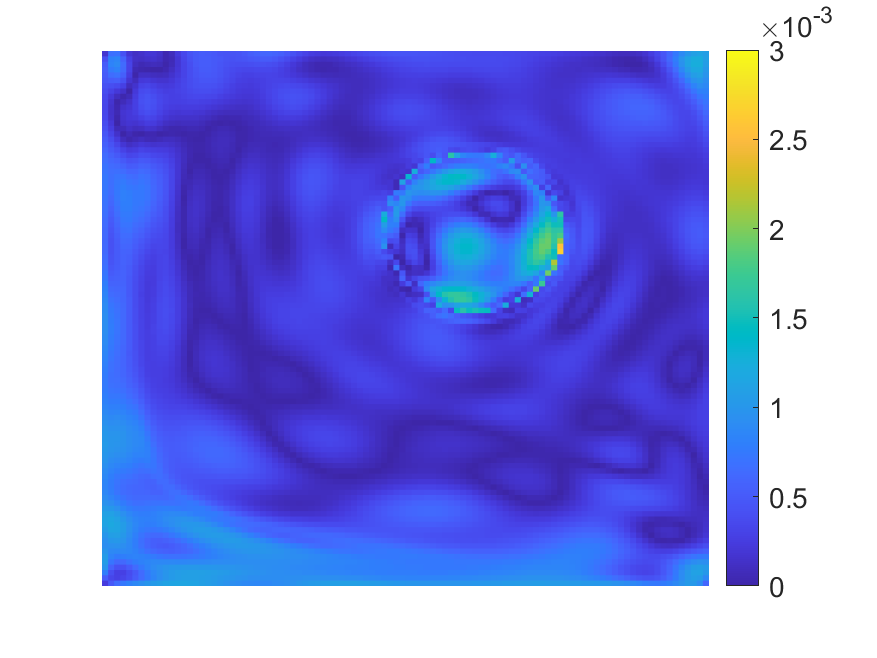}\\
\includegraphics[height=3cm]{true_solution-4} & \includegraphics[height=3cm]{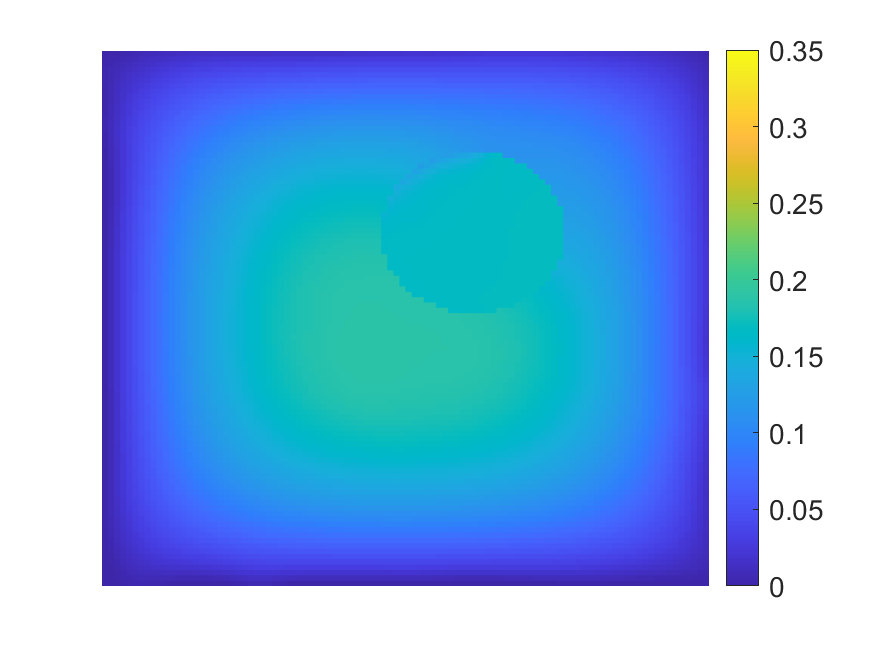} & \includegraphics[height=3cm] {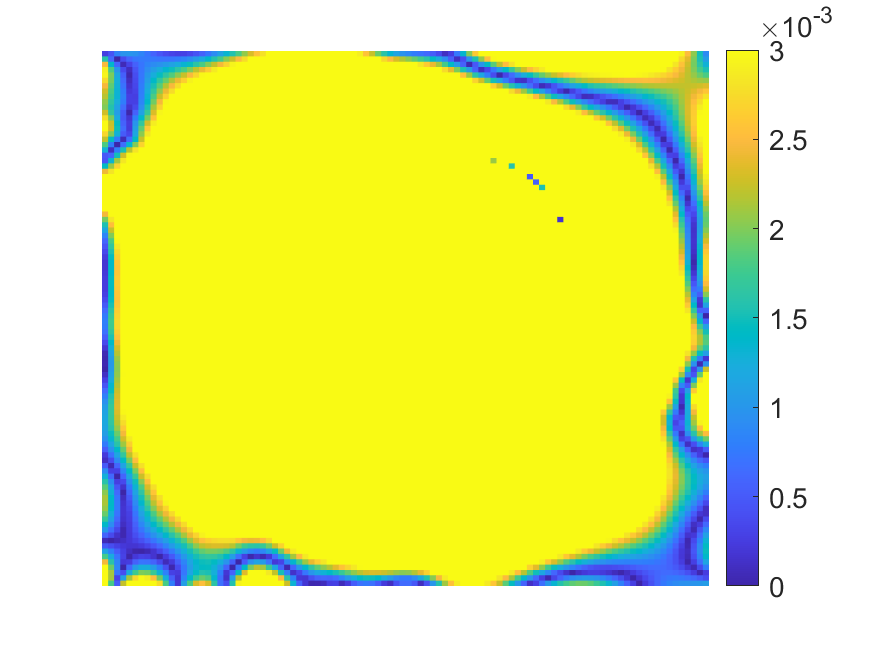}\\
\includegraphics[height=3cm]{true_solution-4} & \includegraphics[height=3cm]{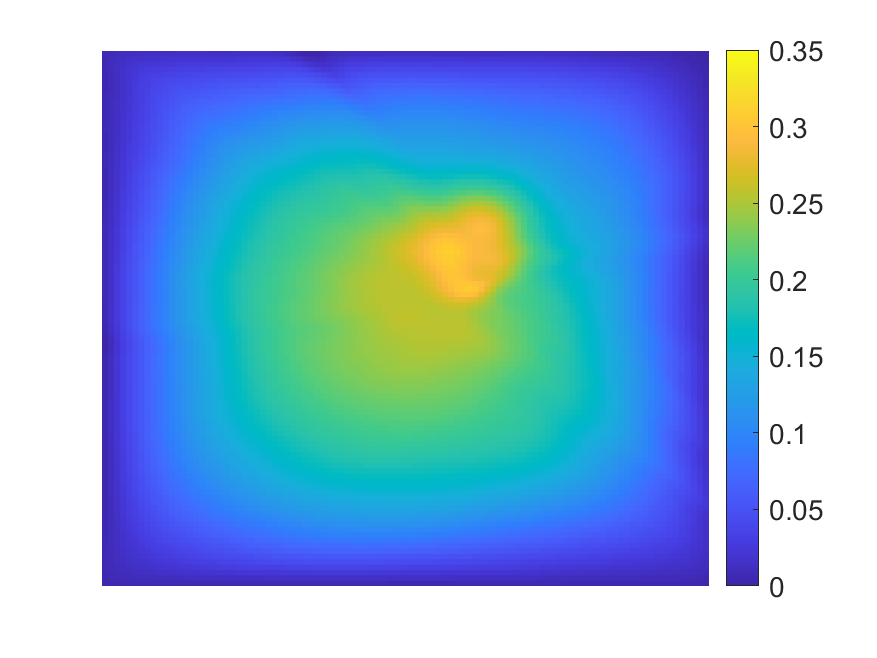} & \includegraphics[height=3cm] {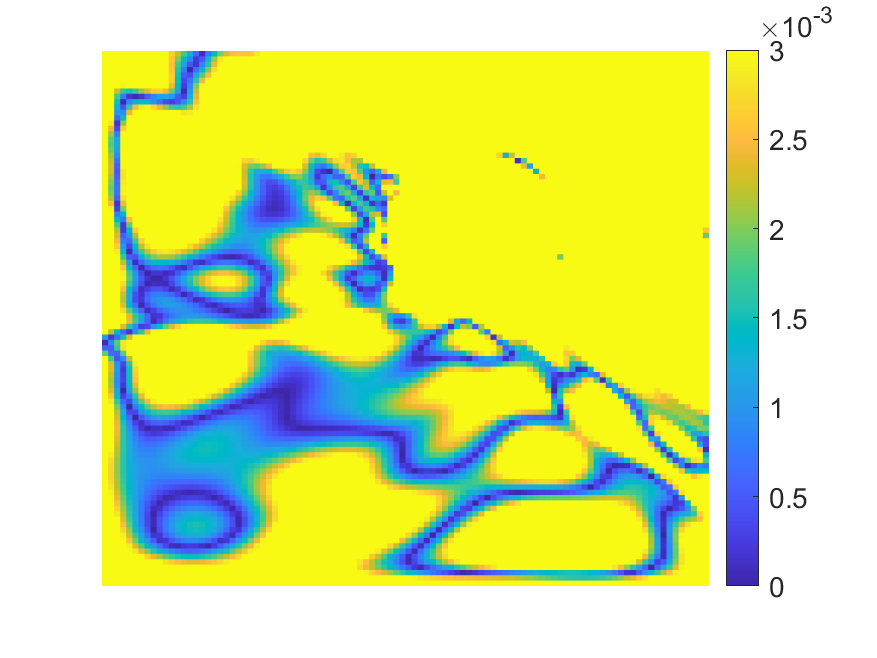}\\
(a) reference & (b) predicted  & (c) error
\end{tabular}
\caption{\label{fig:exam1} The approximations of $u$ for Example~\ref{exam1} by the H-IDRM, DD-PINN, and DRM (from top to bottom).}
\end{figure}

\begin{table}[hbt!]
\centering
\begin{threeparttable}
\caption{\label{table:exam1} The relative errors for Example~\ref{exam1}.}
\centering
\begin{tabular}[5pt]{c|c|c|c}
\toprule
Method & H-IDRM & DD-PINN & DRM \\
\midrule
$e_u$ & $3.06\times10^{-3}$ & $2.77\times10^{-1}$ & $9.71\times10^{-2}$ \\
\bottomrule
\end{tabular}
\end{threeparttable}
\end{table}

Next, we consider a benchmark problem, the checkerboard lattice problem \cite{Sun2025DirichletNeumann}.

\begin{Example}\label{exam2}
    Let $\Omega = (-1,1)^2$.
Consider problem~\eqref{eqn:interfaceproblem} with the subdomains $\Omega_1 = \{\boldsymbol{x} \in \Omega : x_1x_2>0\}$ and $\Omega_2 = \{\boldsymbol{x} \in \Omega : x_1x_2<0\}$.
The problem data are $a(\boldsymbol{x}) = \chi_{\Omega_1}(\boldsymbol{x})+0.1\chi_{\Omega_2}(\boldsymbol{x})$, $c(\boldsymbol{x}) = 1$, $f(\boldsymbol{x})\equiv 1$, and $g(\boldsymbol{x})=0$.
\end{Example}

The interface $\Gamma$ consists of two crossing straight line segments, dividing $\Omega$ into four quadrants in a checkerboard pattern. The crossing at the origin induces low regularity of the solution $u$, which poses challenges to conventional numerical methods. The level-set function $\phi$ is taken to be $\phi(\boldsymbol{x}) = x_1x_2$. We obtain the reference solution using the standard linear FEM (with $2{,}000{,}000$ elements). The results are presented in Table~\ref{table:exam2}, confirming the advantages of the H-IDRM for low-regularity problems. %\bluedel{The excellent agreement in the displacement $u$ indicates that the mixed formulation \eqref{eqn:weakform} correctly captures the underlying physics.}

\begin{figure}[hbt!]
\centering\setlength{\tabcolsep}{2pt}
\begin{tabular}{ccc}
\includegraphics[height=3cm]{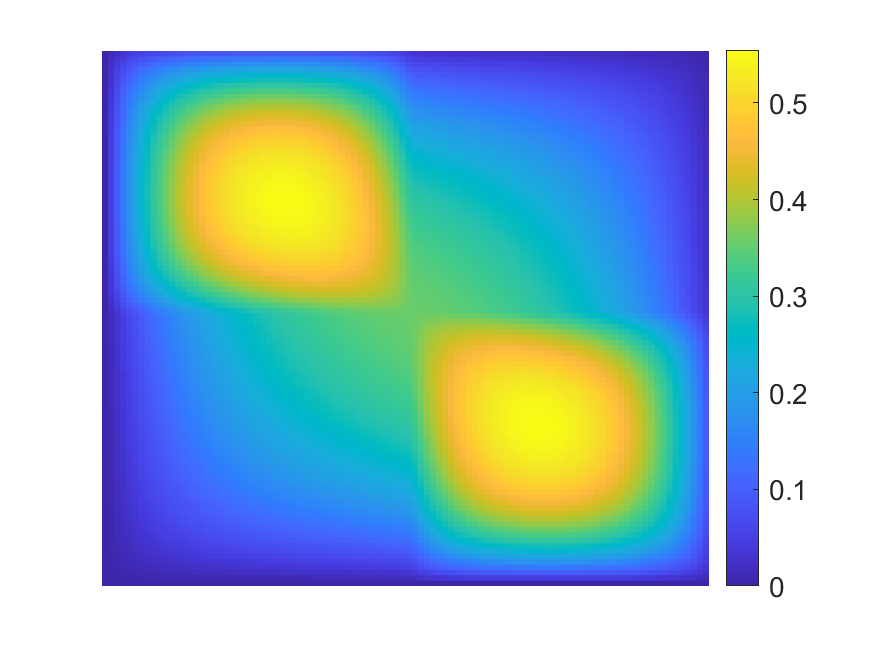} & \includegraphics[height=3cm]{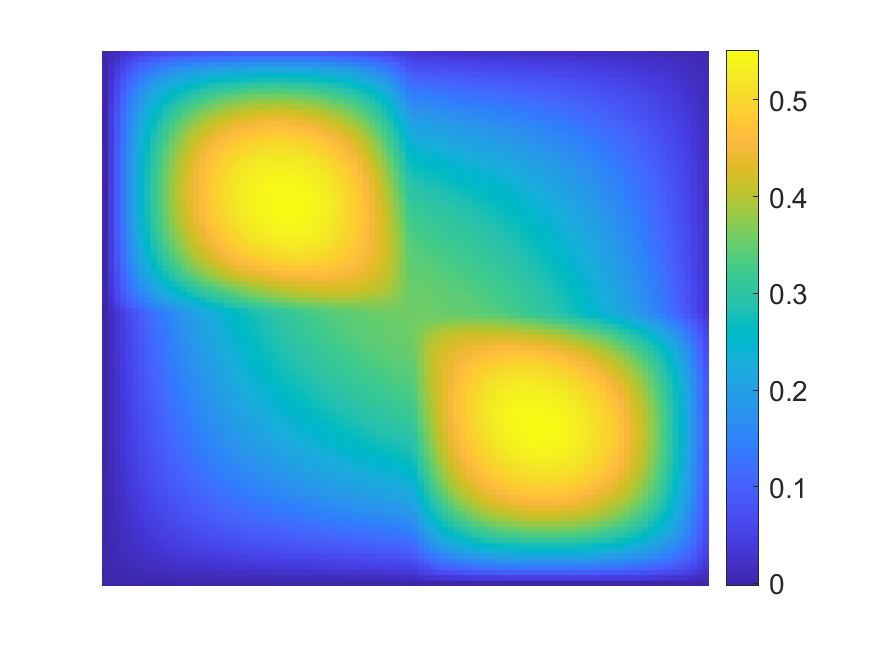} & \includegraphics[height=3cm] {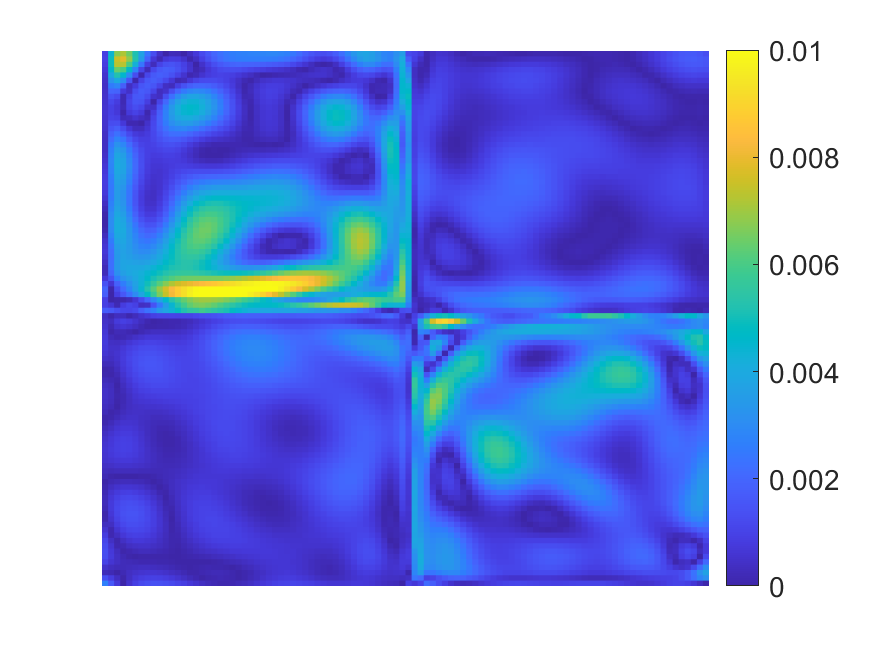}\\
\includegraphics[height=3cm]{true_solution-5} & \includegraphics[height=3cm]{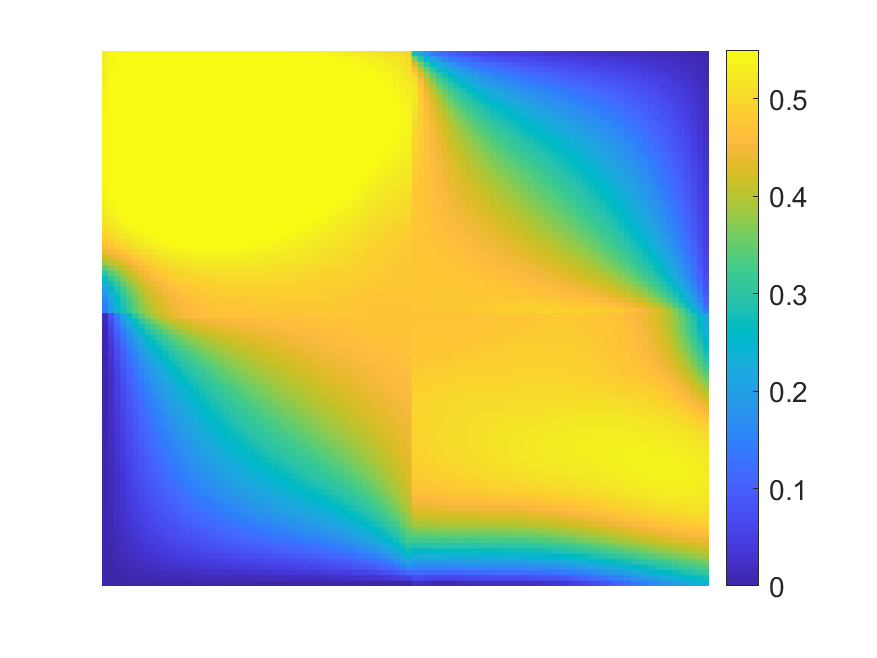} & \includegraphics[height=3cm] {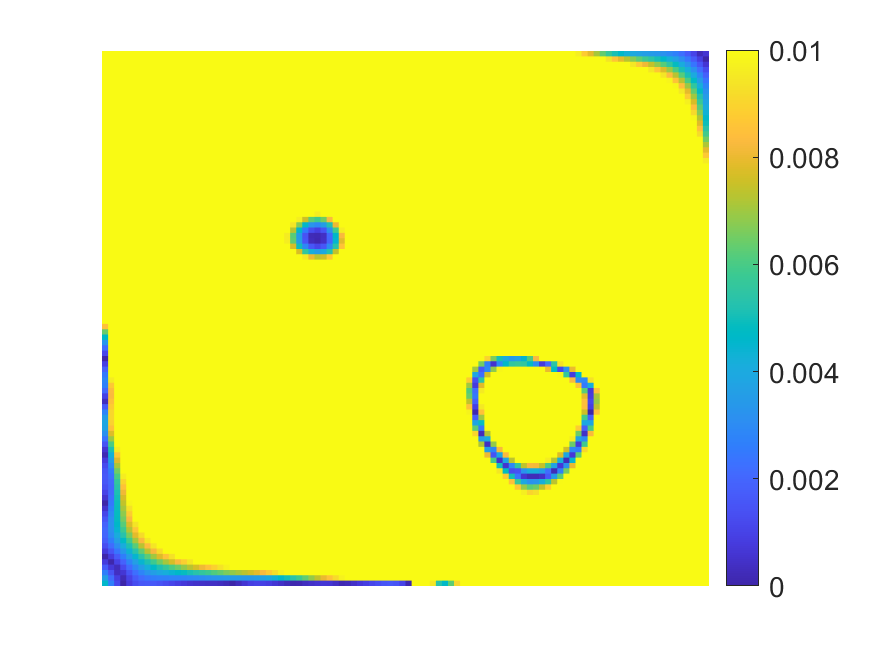}\\
\includegraphics[height=3cm]{true_solution-5} & \includegraphics[height=3cm]{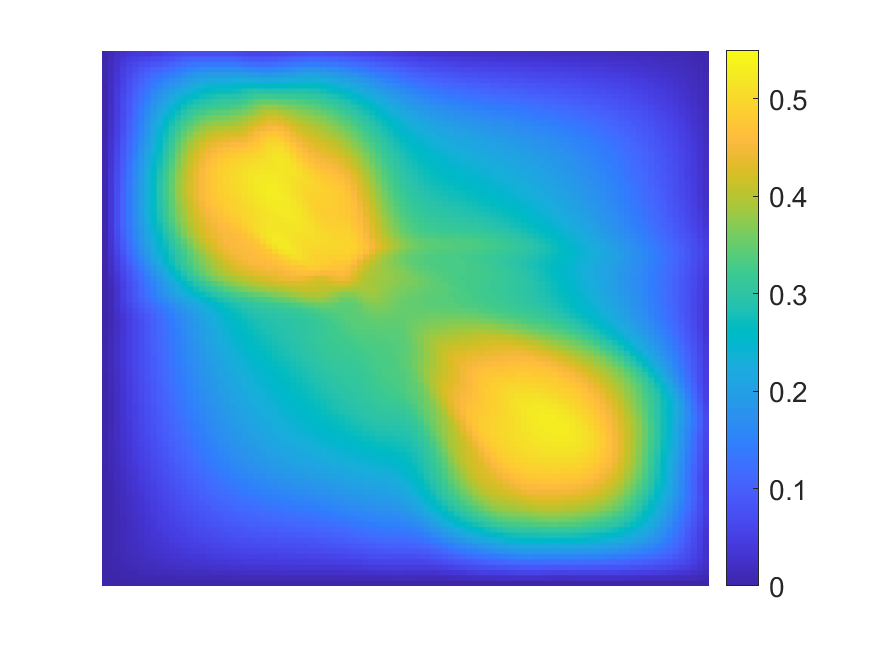} & \includegraphics[height=3cm] {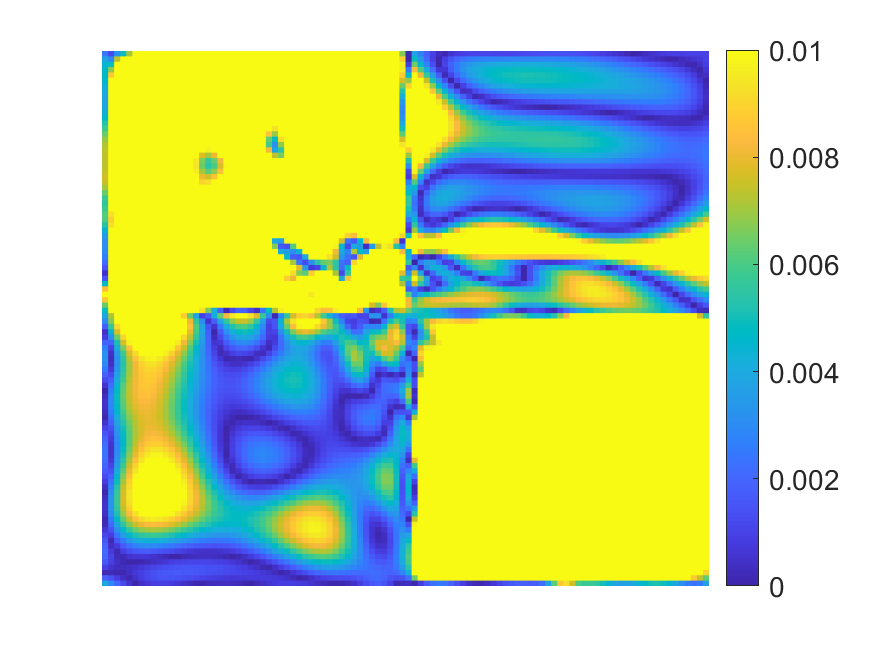}\\
(a) reference & (b) predicted  & (c) error
\end{tabular}
\caption{\label{fig:exam2} The approximations of $u$ for Example~\ref{exam2} by the H-IDRM, DD-PINN, and DRM (from top to bottom).}
\end{figure}

\begin{table}[hbt!]
\centering
\begin{threeparttable}
\caption{\label{table:exam2} The relative errors for Example~\ref{exam2}.}
\centering
\begin{tabular}[5pt]{c|c|c|c}
\toprule
Method & H-IDRM & DD-PINN & DRM \\
\midrule
$e_u$ & $8.09\times10^{-3}$ & $5.91\times10^{-1}$ & $9.86\times10^{-2}$ \\
\bottomrule
\end{tabular}
\end{threeparttable}
\end{table}

The next 3D problem is adapted from~\cite[Example 1]{hu2024directfiniteelementmethod}.

\begin{Example}\label{exam3}
Let $\Omega = \{\boldsymbol{x} \in \mathbb{R}^3 : \|\boldsymbol{x}\|_2< 1\}$.
The subdomains are separated by a concentric spherical interface of radius $r_0 = 0.5$: $\Omega_1 = \{\boldsymbol{x} \in \Omega : \|\boldsymbol{x}\|_2 > r_0\}$ and $\Omega_2 = \{\boldsymbol{x} \in \Omega : \|\boldsymbol{x}\|_2 < r_0\}$.
The diffusion coefficient  $a(\boldsymbol{x})=0.1\chi_{\Omega_1}(\boldsymbol{x})+\chi_{\Omega_2}(\boldsymbol{x})$, and $g(\boldsymbol{x})=0$ on $\partial\Omega$.
The exact solution $u(\boldsymbol{x})$ is given by 
$u(\boldsymbol{x})=[\sin(\|\boldsymbol{x}\|_2^2- r_0^2) + 0.1 r_0^2]\chi_{\Omega_1}(\boldsymbol{x}) + 
0.1\|\boldsymbol{x}\|_2^2 \chi_{\Omega_2}(\boldsymbol{x})$, and the source $f$ is computed from the governing equation.
Consider two cases for the potential $c(\boldsymbol{x})$: {\rm (a)} $c(\boldsymbol{x}) = \max\{\|\boldsymbol{x}\|_2^2 - 0.1, 0\}$, and {\rm (b)} $c(\boldsymbol{x})\equiv0$.
\end{Example}

In case (a), the set $\{\boldsymbol{x} \in \overline{\Omega}_2 : c(\boldsymbol{x}) \geq c_0\}$ has a positive measure for any $c_0 \in (0, 0.15)$.
Thus, the condition on $c$ in Theorem~\ref{thm:elliptic} is satisfied, which ensures the well-posedness of the mixed formulation \eqref{eqn:weakform}. We take the level-set function $\phi(\boldsymbol{x}) = \|\boldsymbol{x}\|_2 - 0.5$, and compare the H-IDRM against three baselines: H-IDRM (w/o LSNN), DD-PINN, and DRM.
The relative errors are reported in Table~\ref{table:exam3a}.
For the DRM, increasing the NN capacity to match that of the H-IDRM leads to severe training instability, which restricts the attainable accuracy. Thus, the DRM results in the table are obtained with a reduced-capacity NN to ensure stable convergence.
This is due to the limited global regularity of the solution $u$: it belongs only to $H^s(\Omega)$, for every $s < \frac{3}{2}$.
Thus, the approximation error of standard FCNNs only decays slowly. This underscores the role of the LSNN architecture in restoring the approximation efficiency.

\begin{table}[hbt!]
\centering
\begin{threeparttable}
\caption{\label{table:exam3a} The relative errors for Example~\ref{exam3}, Case~(a).}
\centering
\begin{tabular}[5pt]{c|c|c|c|c}
\toprule
Method & H-IDRM & H-IDRM (w/o LSNN) & DD-PINN & DRM \\
\midrule
$e_u$ & $2.05\times 10^{-3}$ & $4.02\times 10^{-3}$ & $1.21\times 10^{-2}$ & $9.74\times 10^{-2}$ \\
\hline
$e_{\mathbf{p}}$ & $6.57\times 10^{-3}$ & $1.18\times 10^{-2}$ & $1.92\times 10^{-2}$ & $3.38\times 10^{-1}$ \\
\bottomrule
\end{tabular}
\end{threeparttable}
\end{table}

Fig.~\ref{fig:exam3:u} presents the solution $u$, the predictions by the H-IDRM (with and without LSNN), DD-PINN, and DRM, and their pointwise errors on $x_3=0$, and Fig.~\ref{fig:exam3:p1} shows the first component $a\partial_{x_1}u$ of the flux. The H-IDRM clearly achieves better accuracy than the baselines.
The variant without the LSNN also achieves reasonable accuracy, but its larger approximation error reflects the intrinsic defect of using a smooth NN to deal with a steep gradient. %\bluedel{For the DD-PINN and the DRM, the error $e_u$ for the state $u$ is much larger, and the error in the flux $\mathbf{p}$ is even larger. In contrast, for the hybrid method, although $e_{\mathbf{p}}$ is also slightly larger, it remains at a relatively low level.} 

\begin{figure}[hbt!]
\centering\setlength{\tabcolsep}{2pt}
\begin{tabular}{ccc}
\includegraphics[height=3cm]{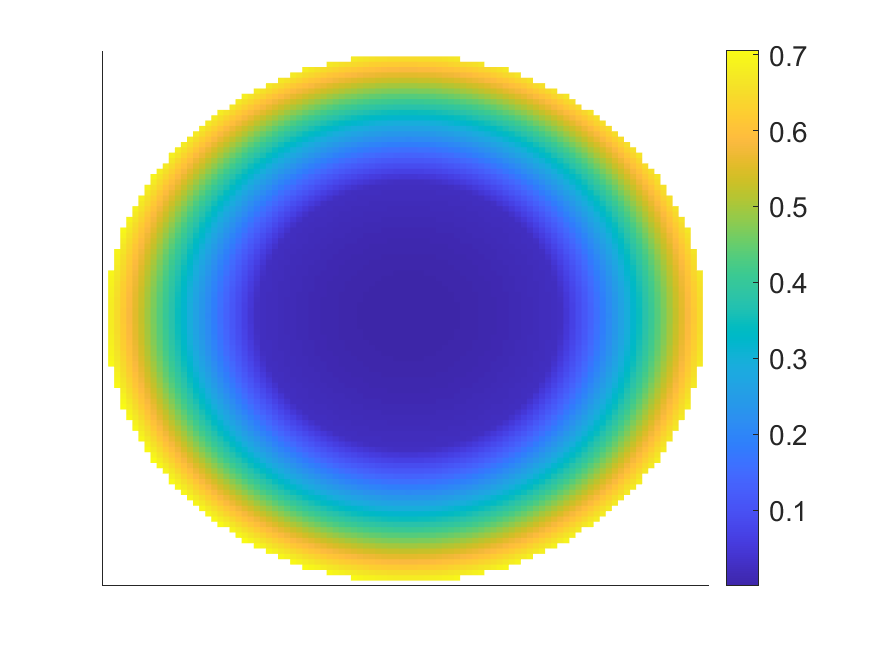} & \includegraphics[height=3cm]{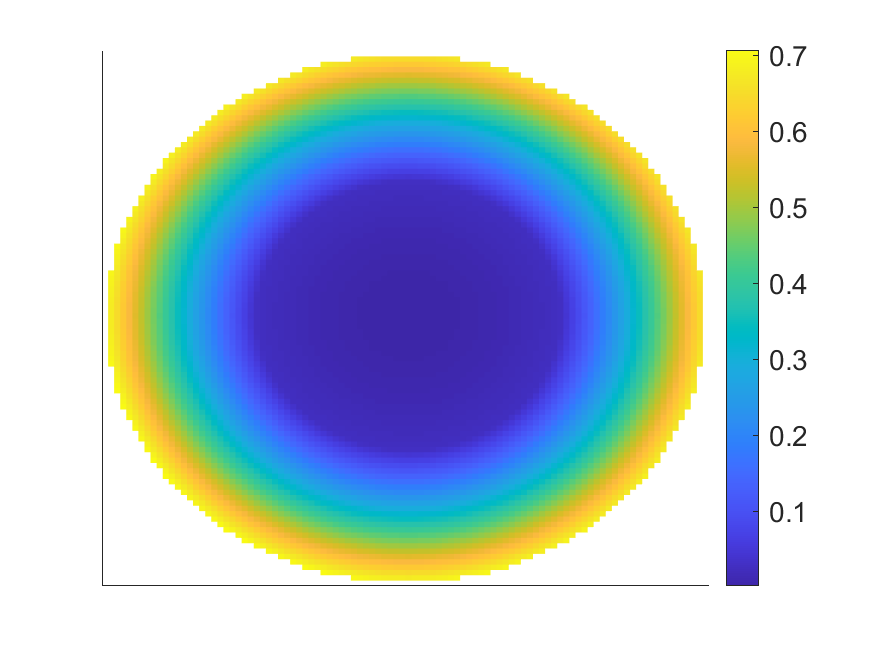} & \includegraphics[height=3cm]{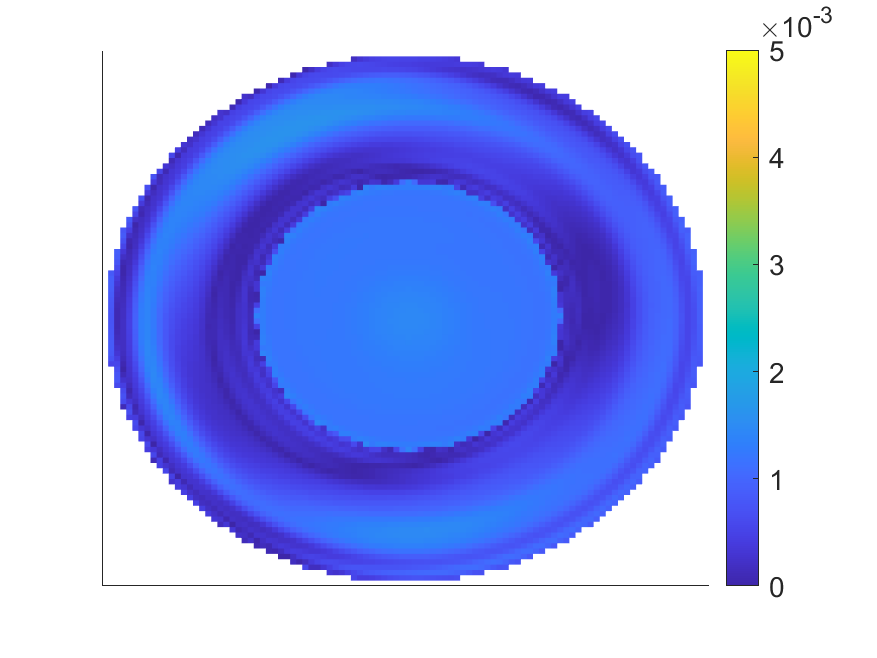}\\
\includegraphics[height=3cm]{true_solution-1} & \includegraphics[height=3cm]{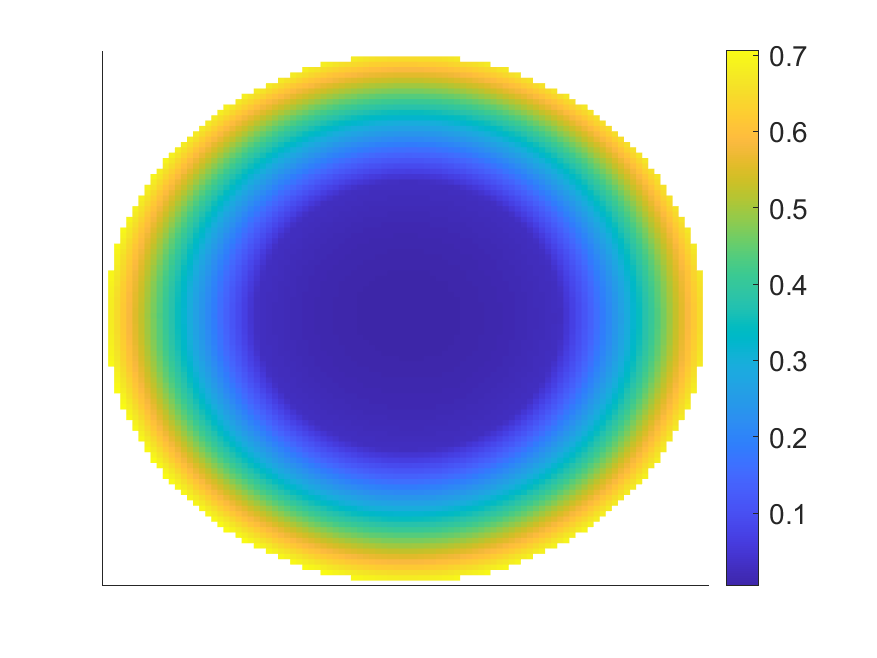} & \includegraphics[height=3cm]{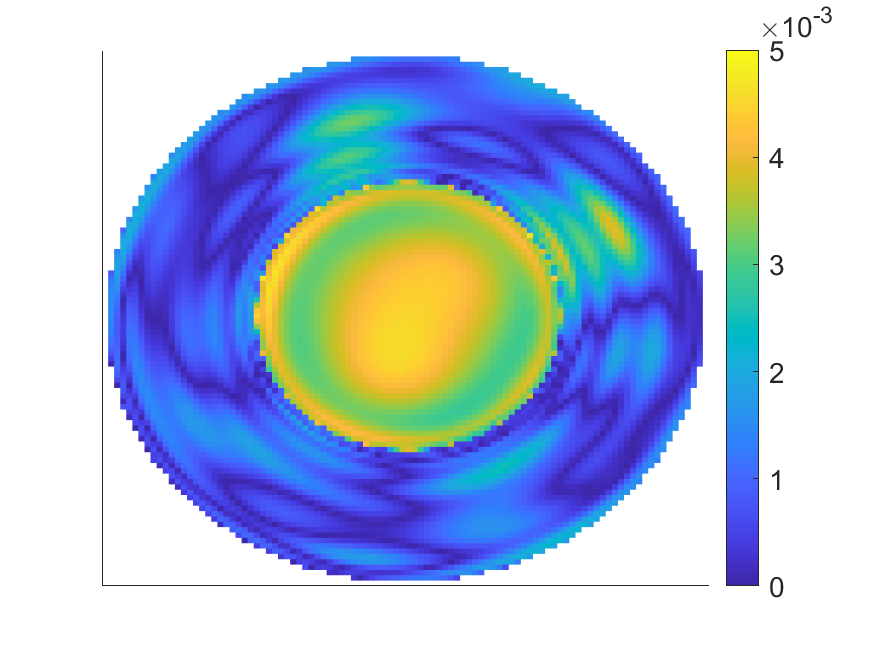}\\
\includegraphics[height=3cm]{true_solution-1} & \includegraphics[height=3cm]{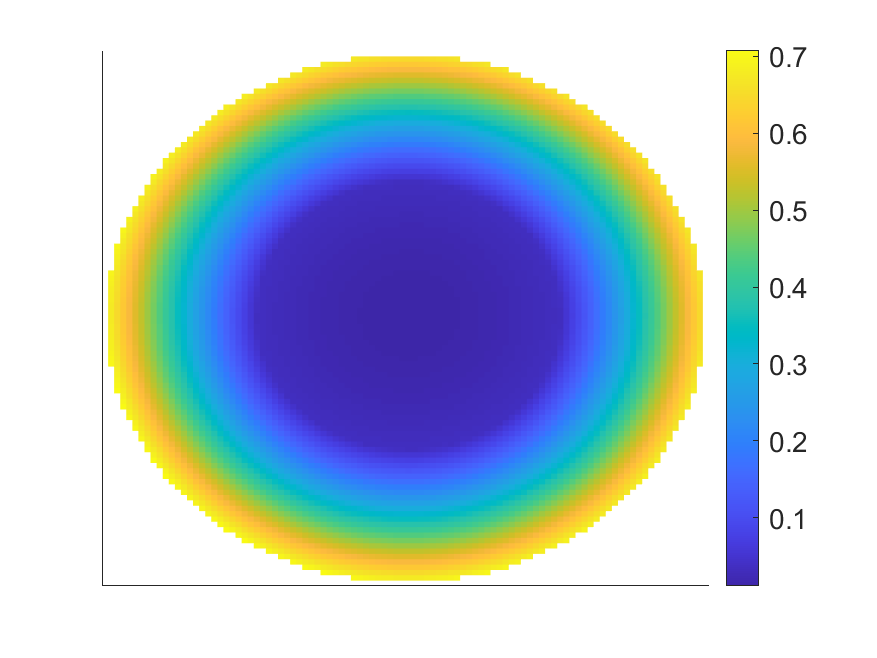} & \includegraphics[height=3cm]{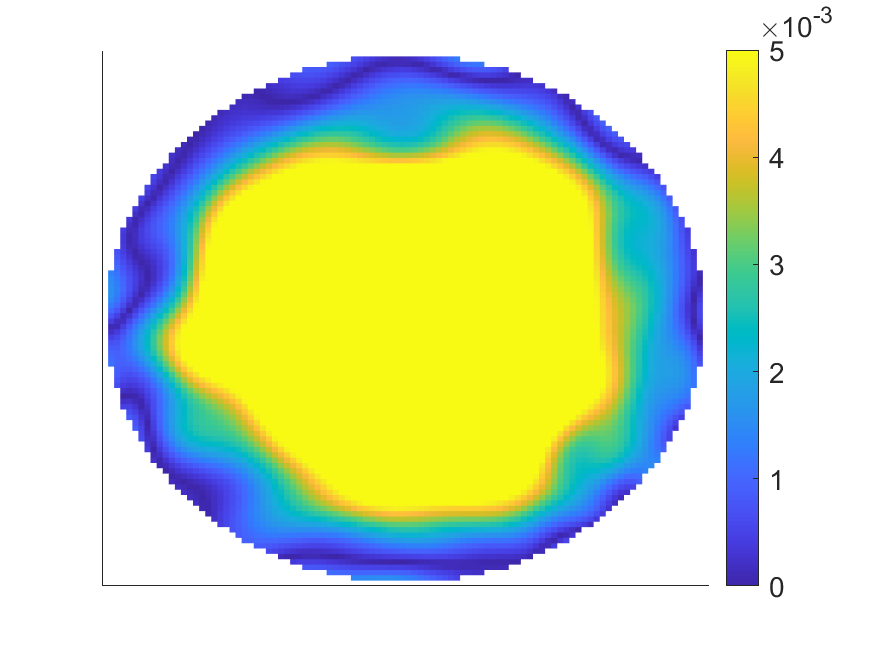}\\
\includegraphics[height=3cm]{true_solution-1} & \includegraphics[height=3cm]{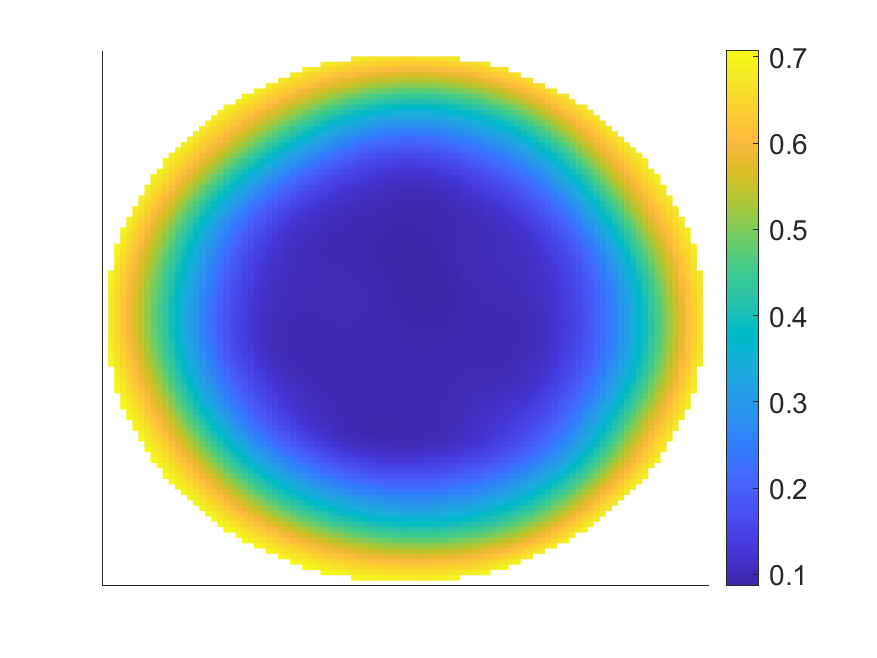} & \includegraphics[height=3cm]{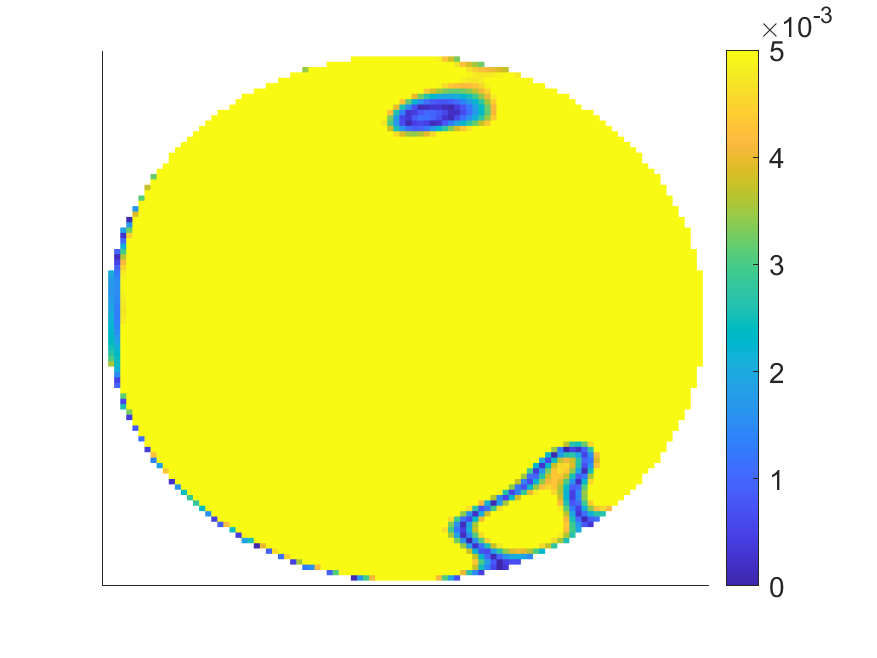}\\
(a) exact & (b) predicted  & (c) error
\end{tabular}
\caption{\label{fig:exam3:u} The approximations of $u$ for Example~\ref{exam3}, Case~(a) (slice at $x_3 = 0$) by the H-IDRM, H-IDRM (w/o LSNN), DD-PINN, and DRM (from top to bottom).}
\end{figure}

\begin{figure}[hbt!]
\centering\setlength{\tabcolsep}{2pt}
\begin{tabular}{ccc}
\includegraphics[height=3cm]{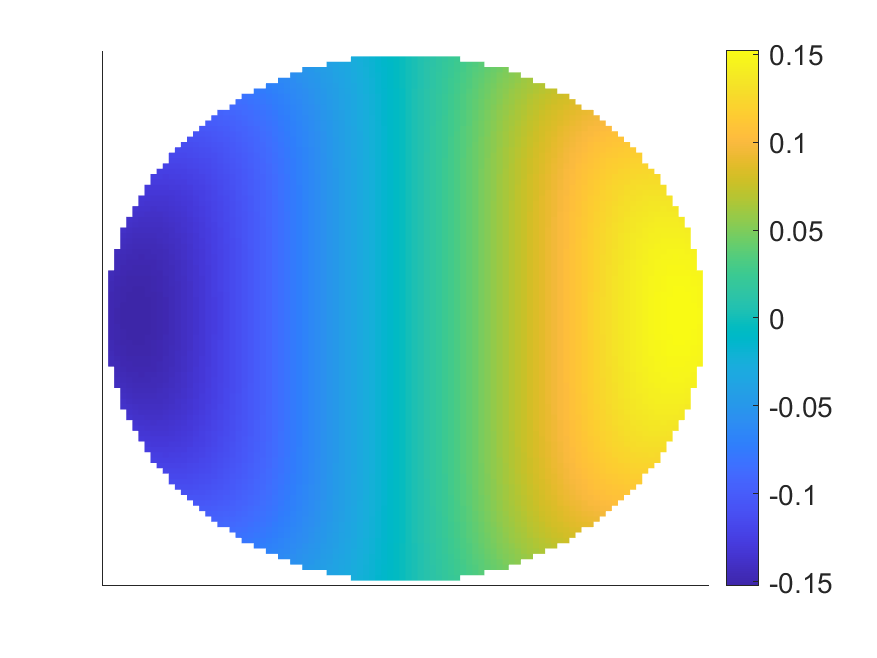} & \includegraphics[height=3cm]{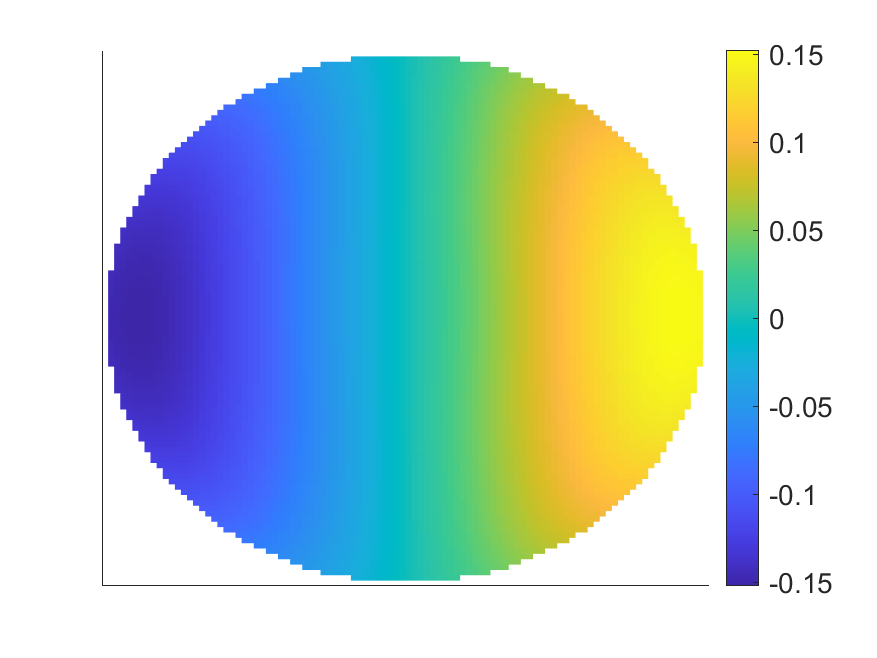} & \includegraphics[height=3cm] {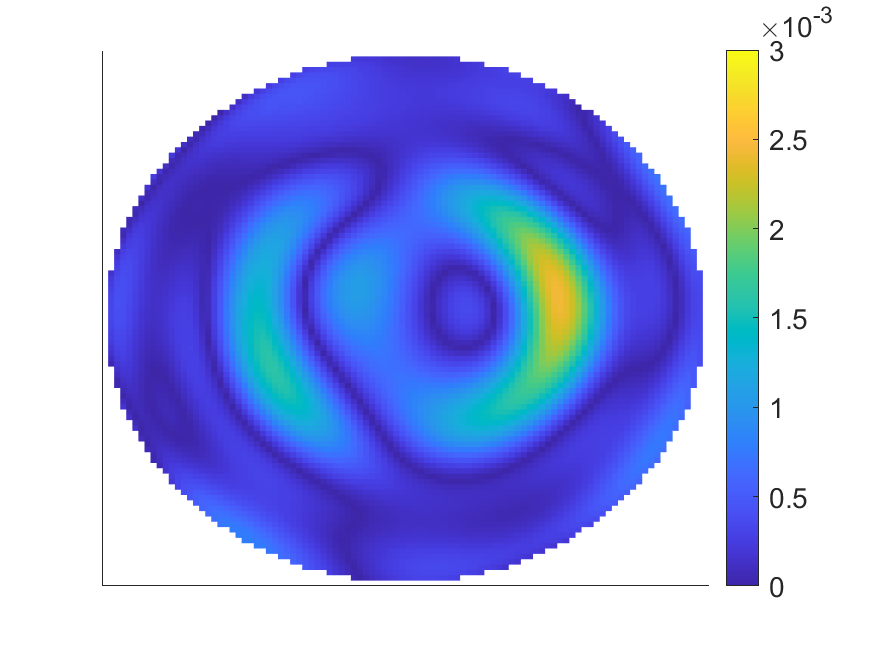}\\
\includegraphics[height=3cm]{true_gradsolution1-1} & \includegraphics[height=3cm]{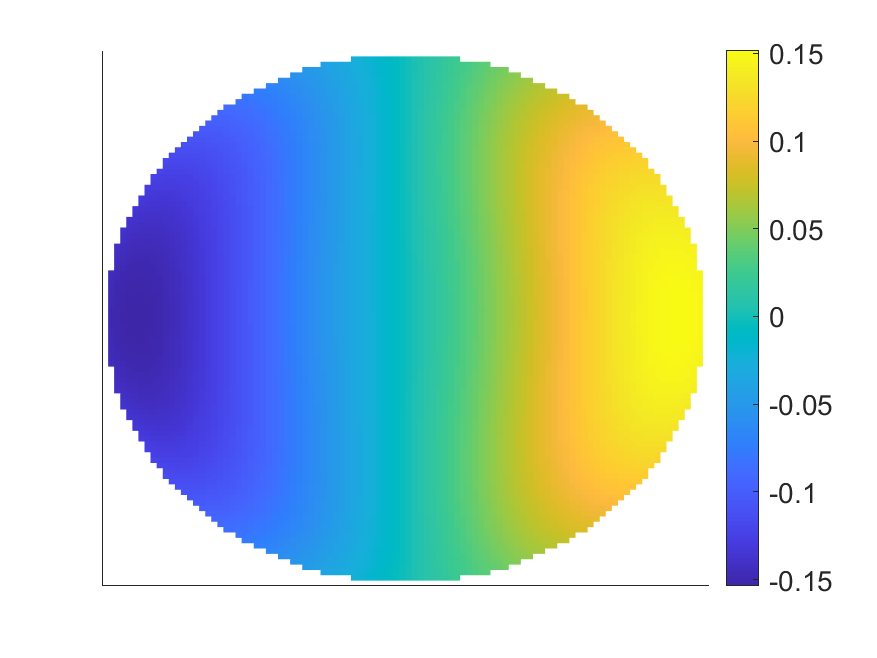} & \includegraphics[height=3cm]  {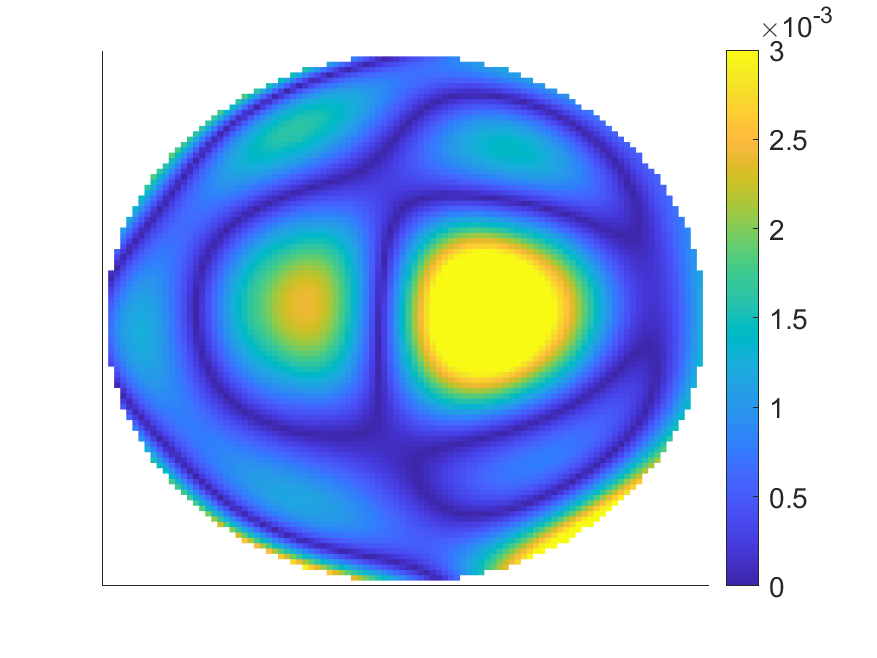}\\
\includegraphics[height=3cm]{true_gradsolution1-1} & \includegraphics[height=3cm]{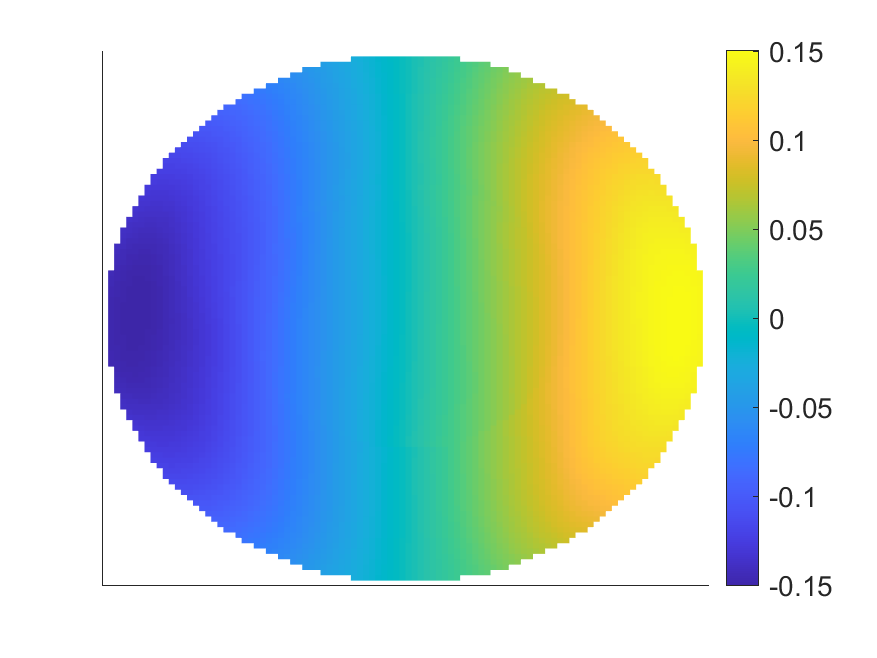} & \includegraphics[height=3cm]  {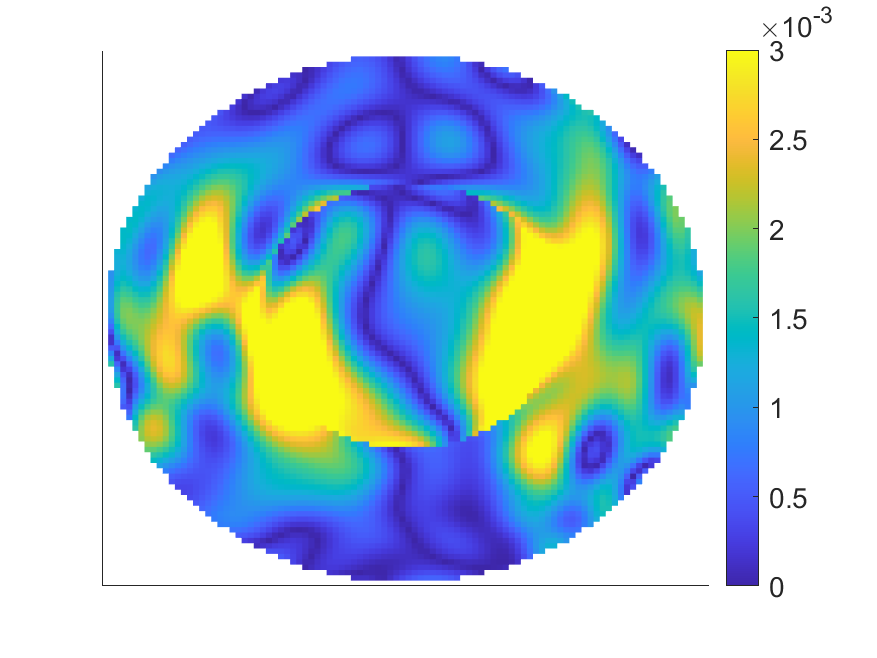}\\
\includegraphics[height=3cm]{true_gradsolution1-1} & \includegraphics[height=3cm]{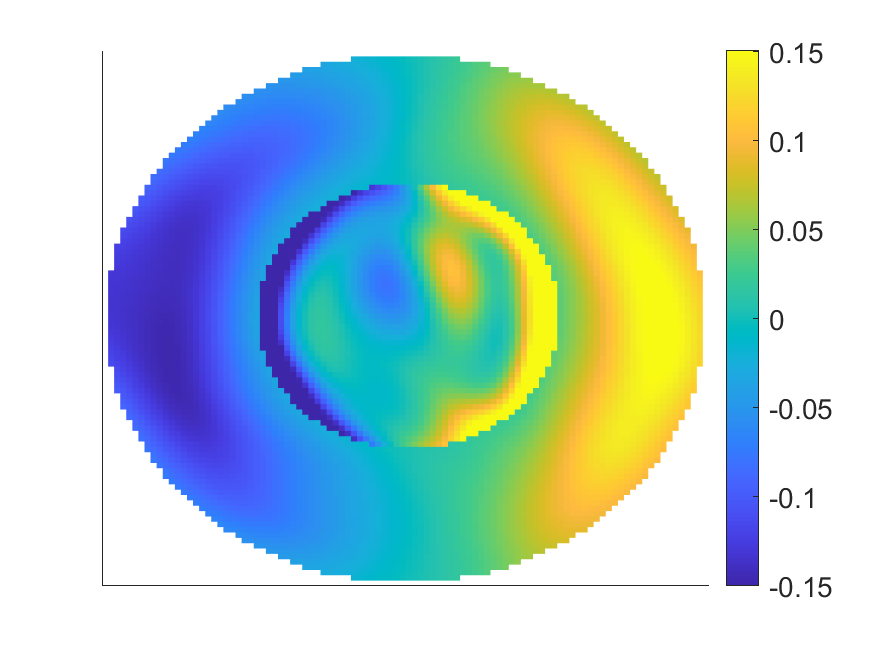} & \includegraphics[height=3cm]  {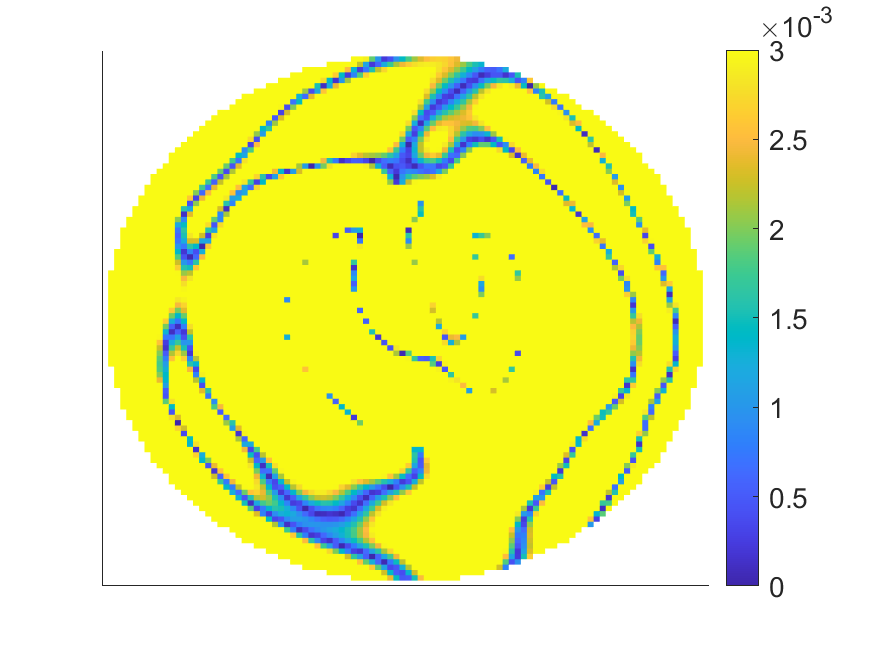}\\
(a) exact & (b) predicted  & (c) error
\end{tabular}
\caption{\label{fig:exam3:p1} The approximations of $p_1 = a\partial_{x_1}u$ for Example~\ref{exam3}, Case~(a) (slice at $x_3 = 0$) by the H-IDRM, H-IDRM (w/o LSNN), DD-PINN, and DRM (from top to bottom).}
\end{figure}

Fig.~\ref{fig:exam3:loss} presents the convergence trajectory of the loss and the relative errors during the entire training process. The monotonic decay of the loss $L_k$ in Fig.~\ref{fig:exam3:loss}(a) confirms the training stability imparted by the convex surrogate.
The errors of the primal variable $u$ and the flux $\mathbf{p}$ eventually reach low levels, cf. Fig.~\ref{fig:exam3:loss}(b)-(c), which shows the robustness of the proposed training scheme.

\begin{figure}[hbt!]
\centering\setlength{\tabcolsep}{2pt}
\begin{tabular}{ccc}
\includegraphics[height=4cm] {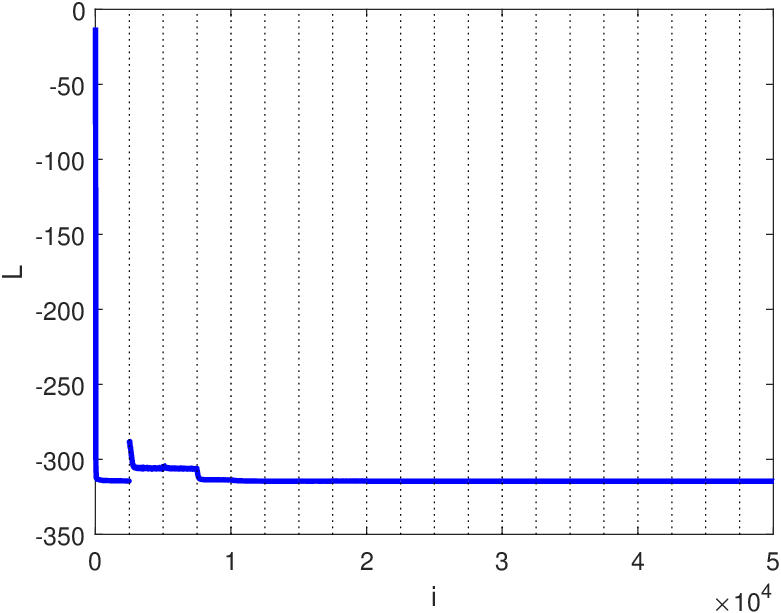}&\includegraphics[height=4cm] {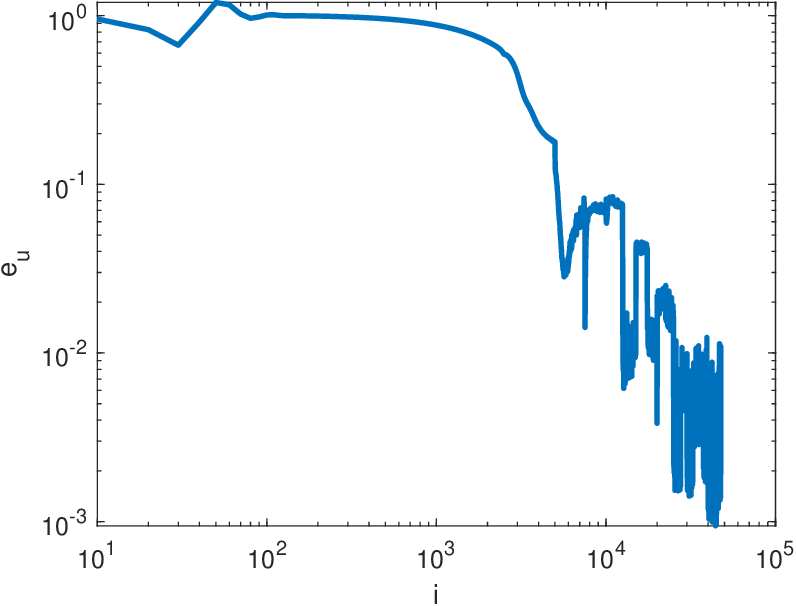}&\includegraphics[height=4cm] {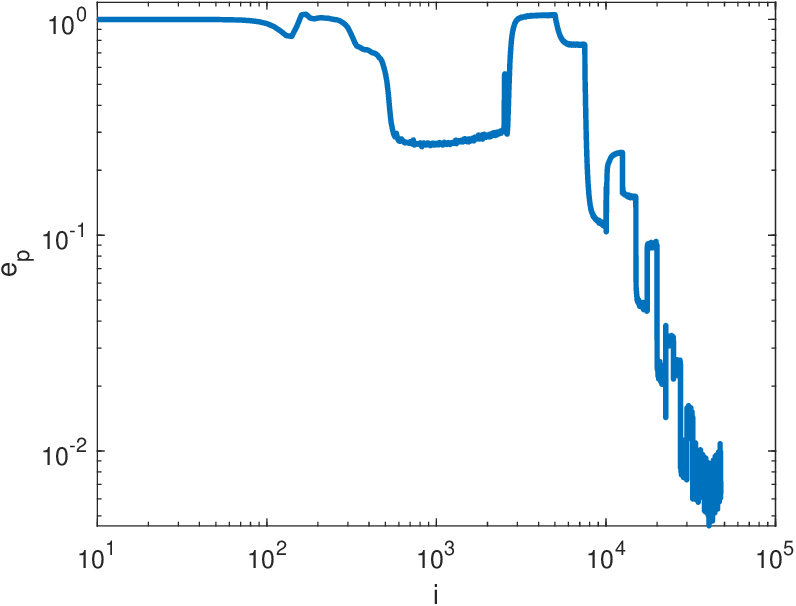}\\
(a) $i$ vs $L$ & (b) $i$ vs $e_u$  & (c) $i$ vs $e_{\mathbf{p}}$
\end{tabular}
\caption{\label{fig:exam3:loss} The convergence history of $L$, $e_u$, and $e_{\mathbf{p}}$ versus the iteration $i$ for Example~\ref{exam3}, Case (a).}
\end{figure}

In the degenerate case (b), since the condition on $c$ in Theorem~\ref{thm:elliptic} does not hold, the mixed formulation \eqref{eqn:weakform} lacks the desired coercivity.
Following Remark~\ref{remark:plusepsilon}, we resort to a regularized formulation to restore the well-posedness. Fig.~\ref{fig:exam3b} shows the exact solution, the prediction, and the pointwise errors for $u$ and the flux $\mathbf{p}$, with the relative errors $e_u = 1.44 \times 10^{-3}$ and $e_{\mathbf{p}} = 8.56 \times 10^{-3}$.
It is observed that the pointwise error concentrates in the interior region $\Omega_2$. The high accuracy of the approximation confirms that the regularized H-IDRM effectively handles degenerate coefficients.

\begin{figure}[hbt!]
\centering\setlength{\tabcolsep}{2pt}
\begin{tabular}{ccc}
\includegraphics[height=3cm]{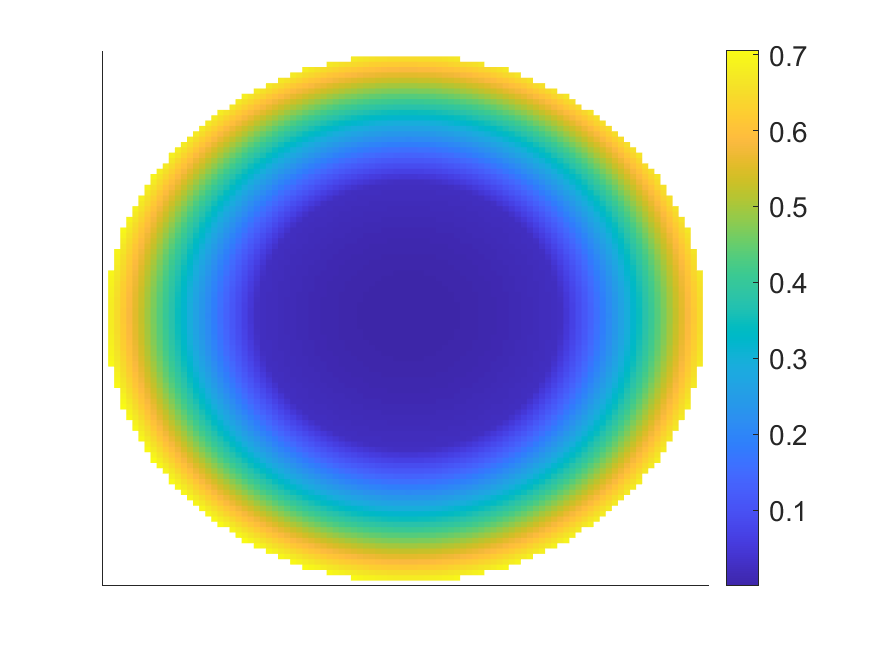} & \includegraphics[height=3cm]{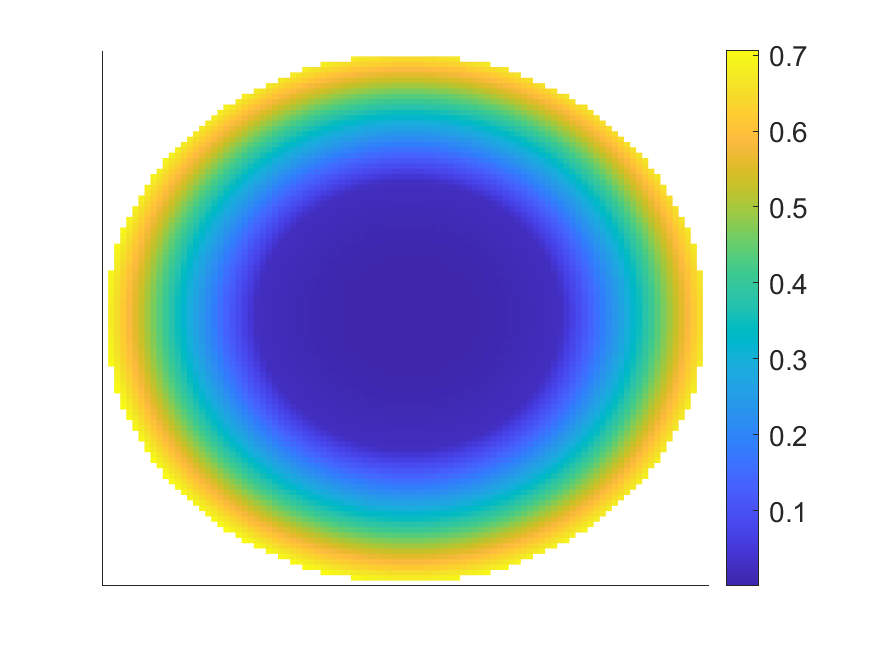} & \includegraphics[height=3cm] {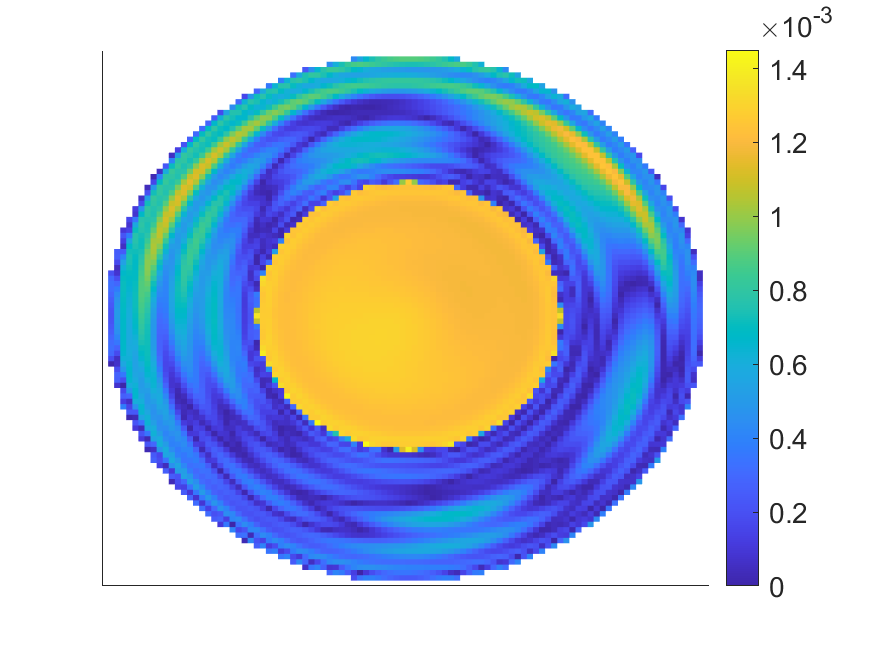}\\
\includegraphics[height=3cm]{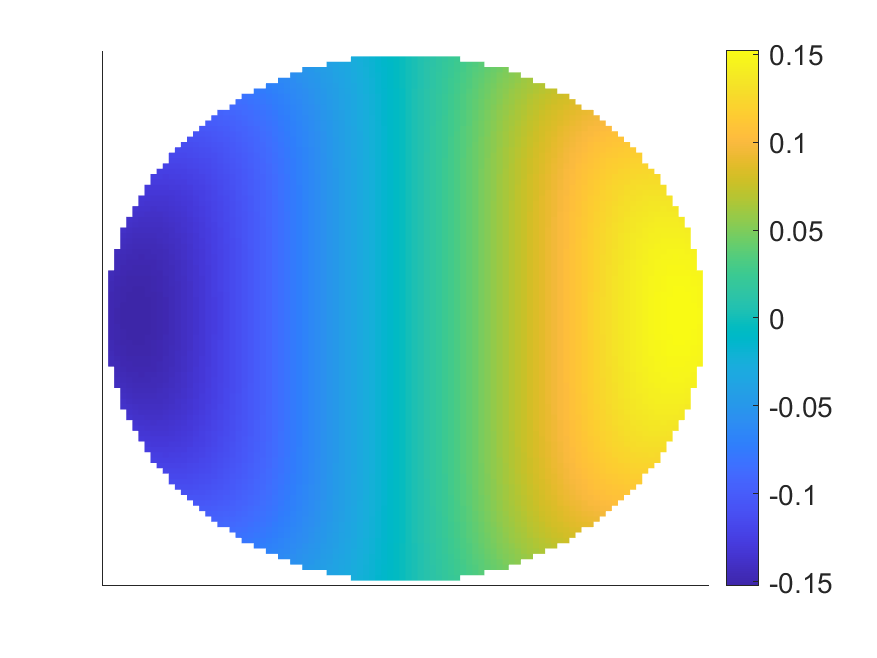} & \includegraphics[height=3cm]{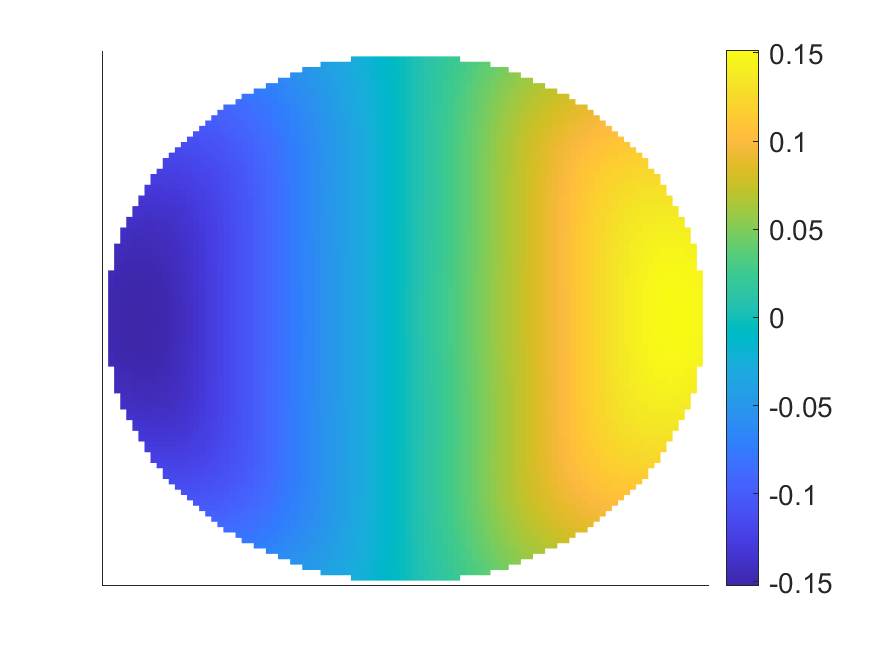} & \includegraphics[height=3cm]  {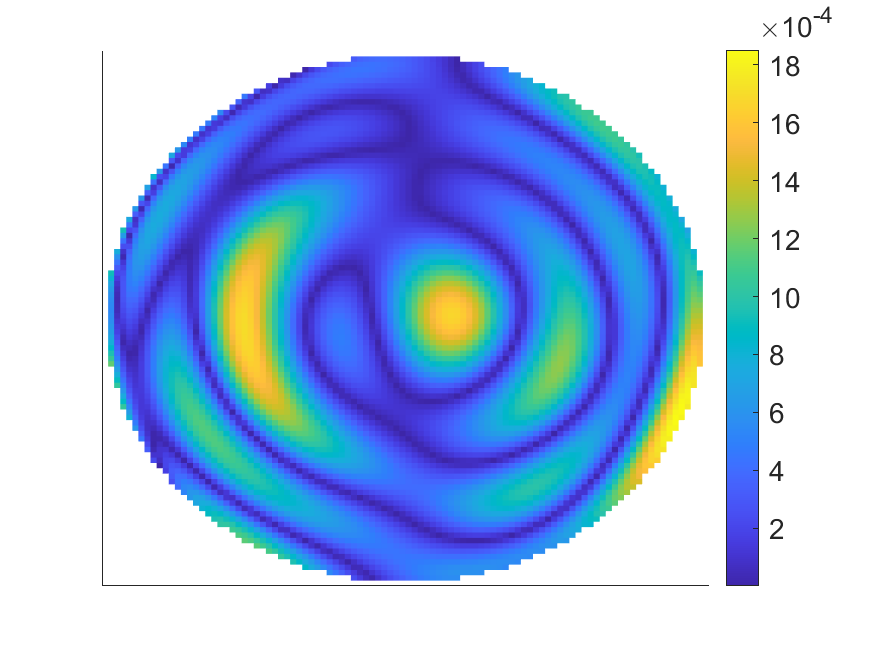}\\
\includegraphics[height=3cm]{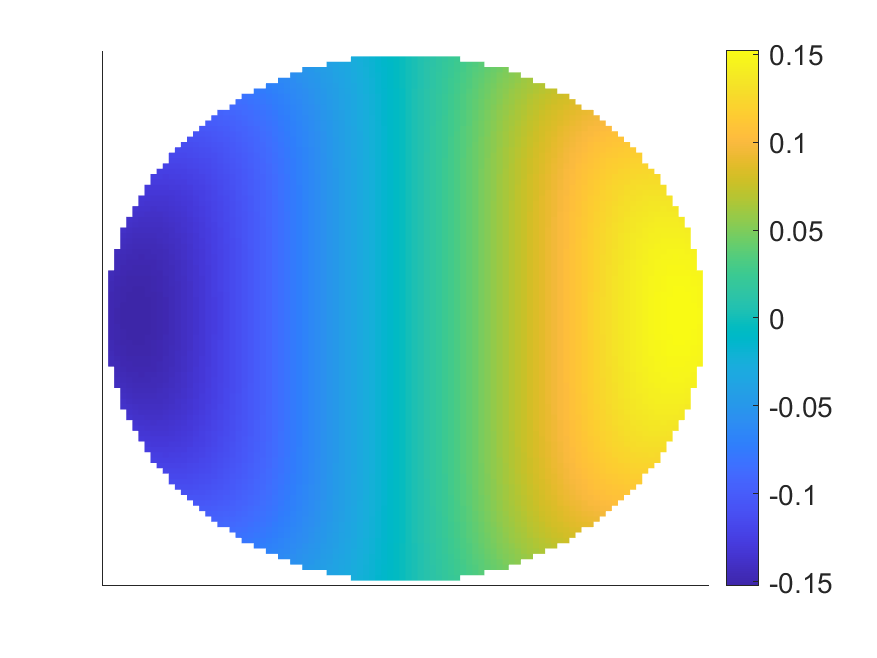} & \includegraphics[height=3cm]{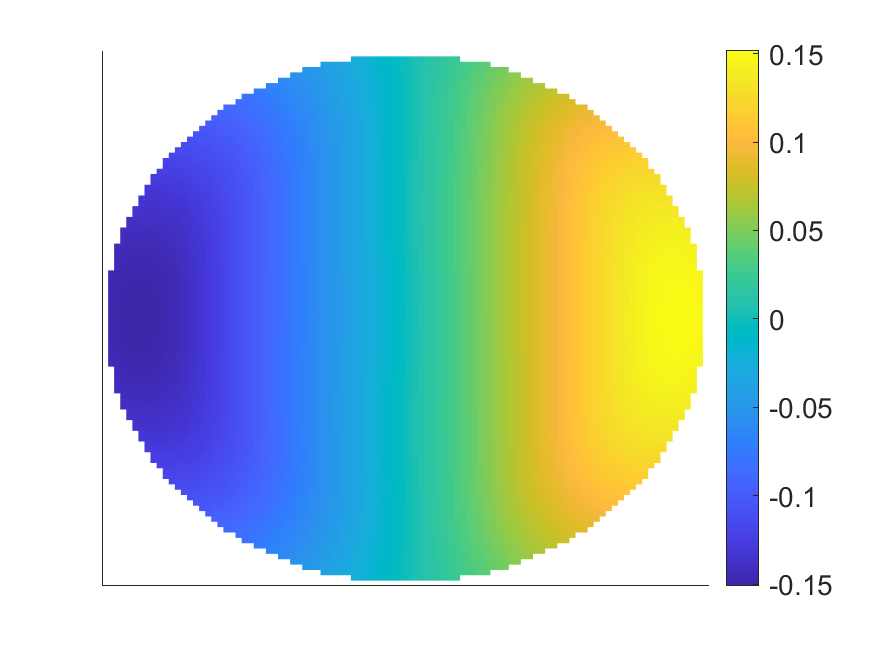} & \includegraphics[height=3cm]  {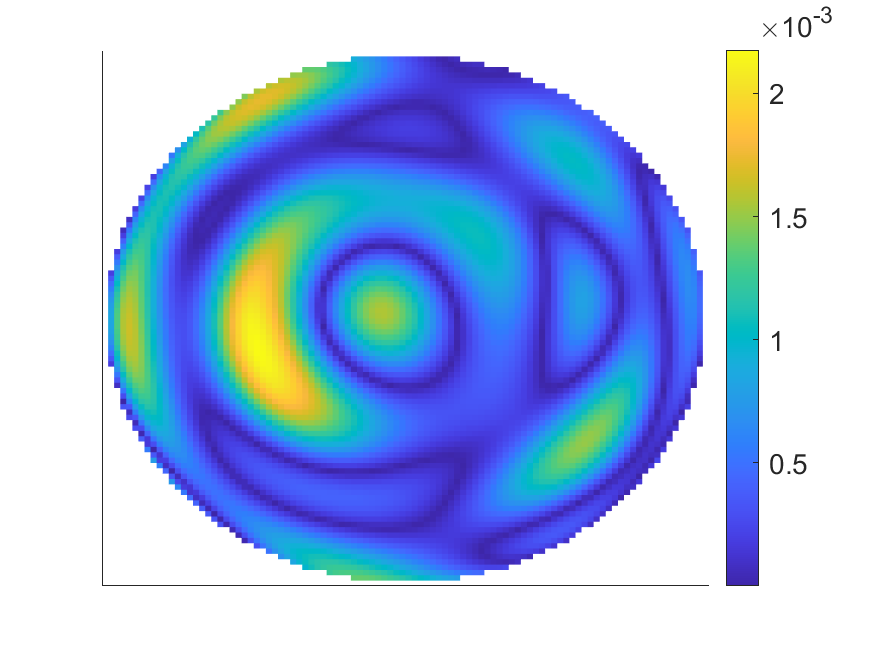}\\
\raisebox{1.8ex}{\hspace{1.0em}\includegraphics[height=2.7cm]{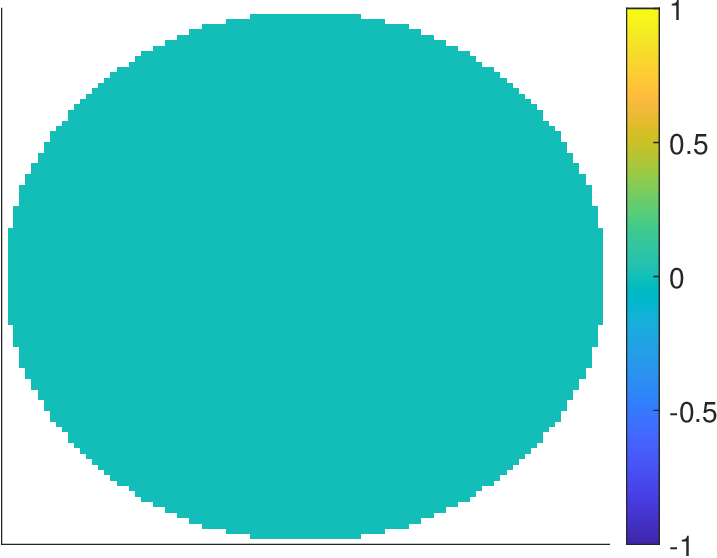}} & \includegraphics[height=3cm]{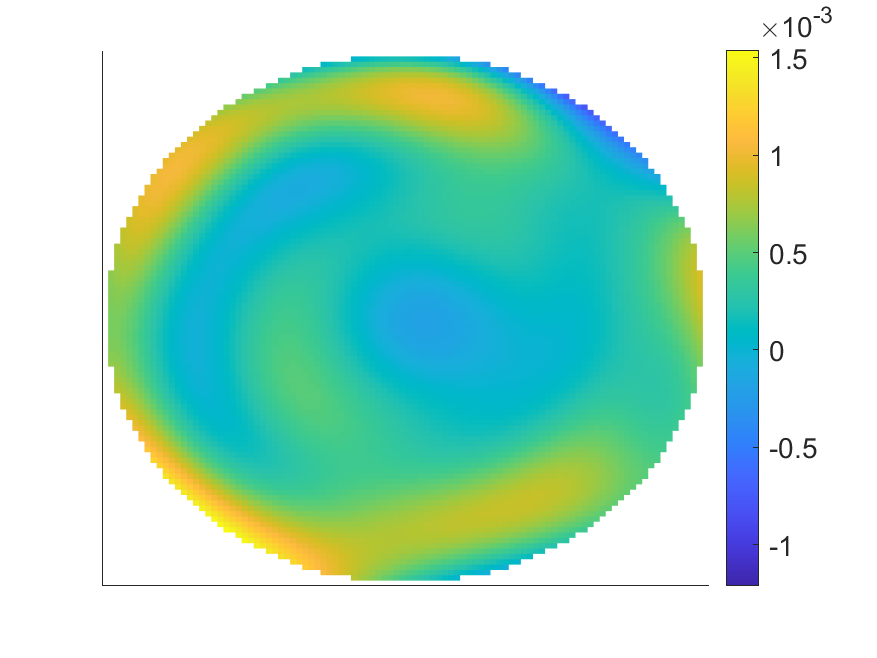} & \includegraphics[height=3cm]  {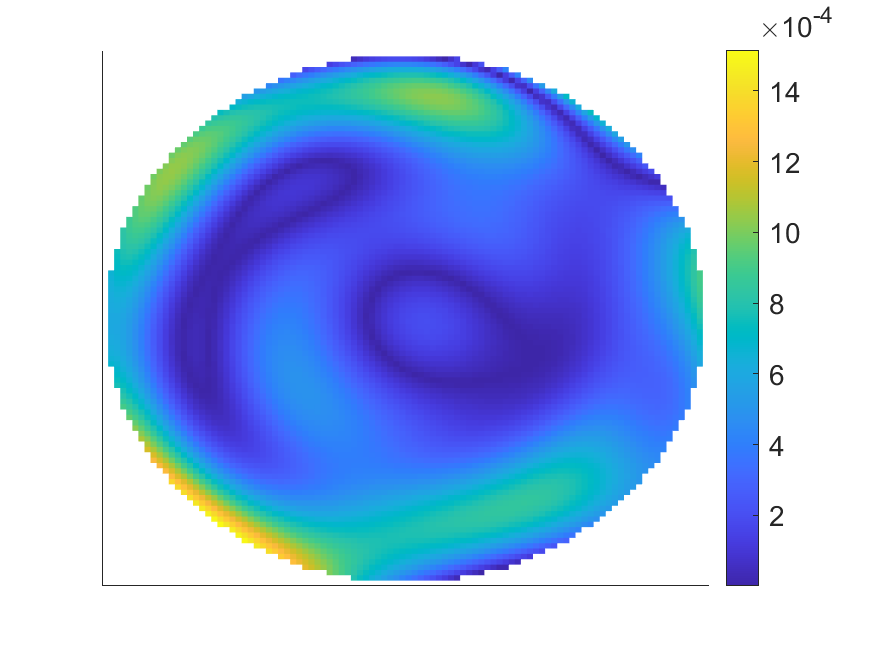}\\
(a) exact & (b) predicted  & (c) error
\end{tabular}
\caption{\label{fig:exam3b} The approximations of $u$ and $\mathbf{p}=(p_1,p_2,p_3)^\top$ (from top to bottom) for Example~\ref{exam3}, Case~(b) (slice at $x_3 = 0$).}
\end{figure}

Next we test the H-IDRM on an example with an implicitly defined interface.

\begin{Example}\label{exam4}
Let $\Omega = \{\boldsymbol{x} \in \mathbb{R}^3 : \|\boldsymbol{x}\|_2 < 1\}$.
The subdomains are defined by the level-set function $\phi(\boldsymbol{x}) = \sin(2x_1^2) + \sin(x_2^2) + \sin(x_3^2) - 0.5$: $\Omega_1 = \{\boldsymbol{x} \in \Omega : \phi(\boldsymbol{x}) > 0\}$ and $\Omega_2 = \{\boldsymbol{x} \in \Omega : \phi(\boldsymbol{x}) < 0\}$.
The problem data are $a(\boldsymbol{x}) = \chi_{\Omega_1}(\boldsymbol{x})+0.1\chi_{\Omega_2}(\boldsymbol{x})$, and $g=0$ on $\partial\Omega$. The solution $u$ is given by 
$u(\boldsymbol{x}) =
(\chi_{\Omega_1}(\boldsymbol{x})+10\chi_{\Omega_2}(\boldsymbol{x}))[\sin(2x_1^2) + \sin(x_2^2) + \sin(x_3^2) - 0.5]$.
Consider two choices of $c(\boldsymbol{x})$: {\rm(a)} $c(\boldsymbol{x}) = \max\{\|\boldsymbol{x}\|_2^2 - 0.1, 0\}$ and {\rm(b)} $c (\boldsymbol{x})\equiv0$.
\end{Example}

\begin{figure}[hbt!]
\centering\setlength{\tabcolsep}{2pt}
\begin{tabular}{cc}
\includegraphics[height=5cm]{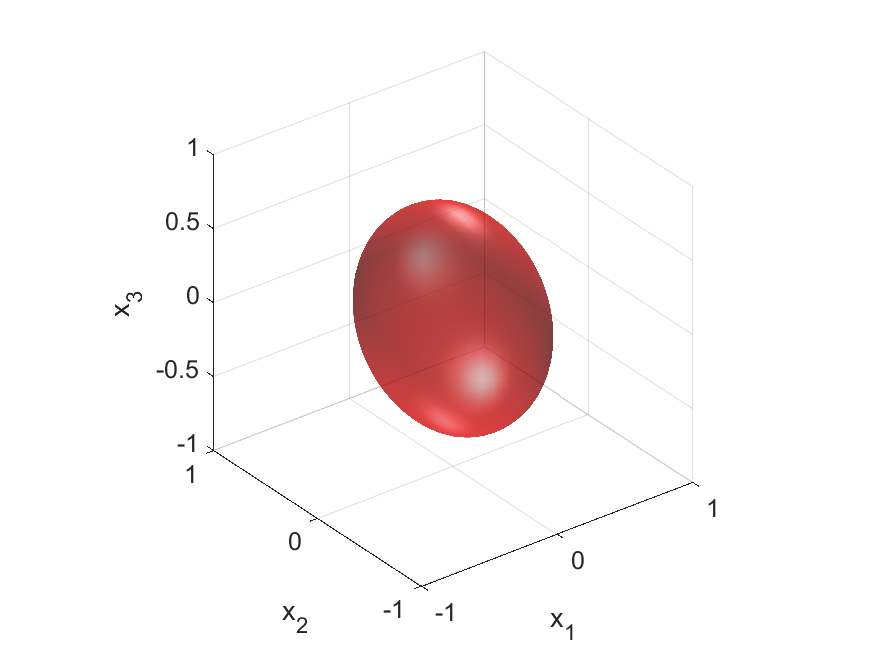} & \includegraphics[height=5cm]{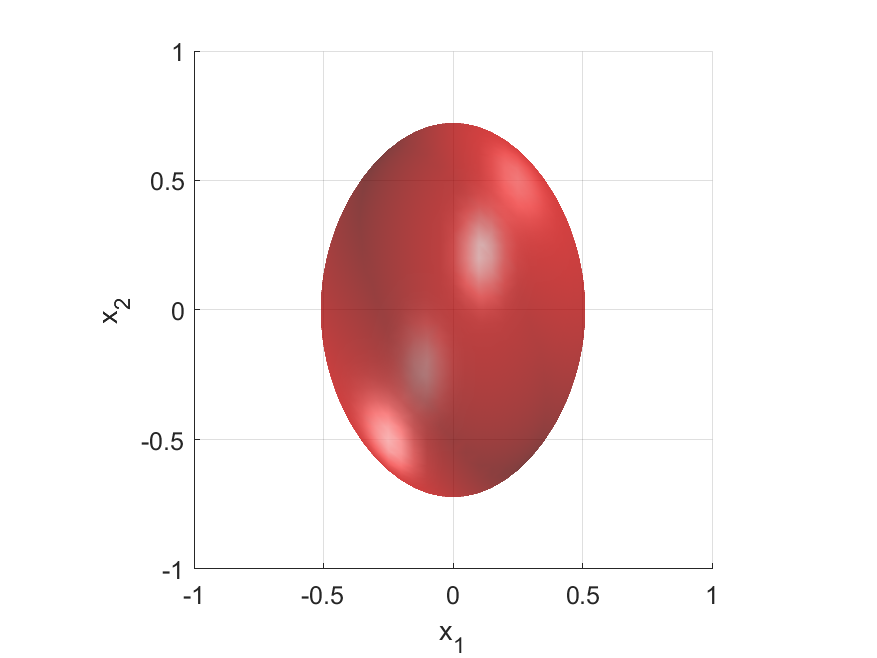}\\
\end{tabular}
\caption{\label{fig:exam4:inter} Interface of Example~\ref{exam4}.}
\end{figure}

The interface $\Gamma$ is shown in Fig.~\ref{fig:exam4:inter}, and we take the level-set function $\phi(\boldsymbol{x}) = \sin(2x_1^2)+\sin(x_2^2)+\sin(x_3^2)-0.5$ in the LSNN. The condition on $c$ in Theorem~\ref{thm:elliptic} is satisfied in case (a). We note that direct uniform sampling on $\Gamma$ is infeasible since there is no tractable parametrization of the interface $\Gamma$. It is inconvenient to use the DD-PINN due to the presence of an $L^2(\Gamma)$ penalty term in the loss. In sharp contrast, the H-IDRM can handle complex geometries without any need for quadrature over the interface $\Gamma$.
Thus, we only compare the H-IDRM with H-IDRM without LSNN and DRM.
The relative errors are presented in Table~\ref{table:exam4a}.
Since the DRM employs a single smooth NN over the domain $\Omega$, one cannot expect accurate approximations of  $u$, as indicated by the errors in Table~\ref{table:exam4a} and the plots in Figs.~\ref{fig:exam4:u} and~\ref{fig:exam4:p1}.
In contrast, the H-IDRM can accurately resolve the solution $u$ and flux $\mathbf{p}$, with relative errors $e_u = 6.57 \times 10^{-3}$ and $e_{\mathbf{p}} = 8.61 \times 10^{-3}$.
Leaving out the LSNN degrades the accuracy by a factor of two to three. This illustrates the importance of encoding the geometry of the interface in the neural solver. 

\begin{table}[hbt!]
\centering
\begin{threeparttable}
\caption{\label{table:exam4a} Relative errors for Example~\ref{exam4}, Case~(a).}
\centering
\begin{tabular}[5pt]{c|c|c|c}
\toprule
Method & H-IDRM & H-IDRM (w/o LSNN) & DRM \\
\midrule
$e_u$ & $6.57\times 10^{-3}$ & $8.11\times 10^{-3}$ & $1.22\times 10^{-1}$ \\
\hline
$e_{\mathbf{p}}$ & $8.61\times 10^{-3}$ & $2.08\times 10^{-2}$ &  $3.02\times 10^{-1}$ \\
\bottomrule
\end{tabular}
\end{threeparttable}
\end{table}

\begin{figure}[hbt!]
\centering\setlength{\tabcolsep}{2pt}
\begin{tabular}{ccc}
\includegraphics[height=3cm]{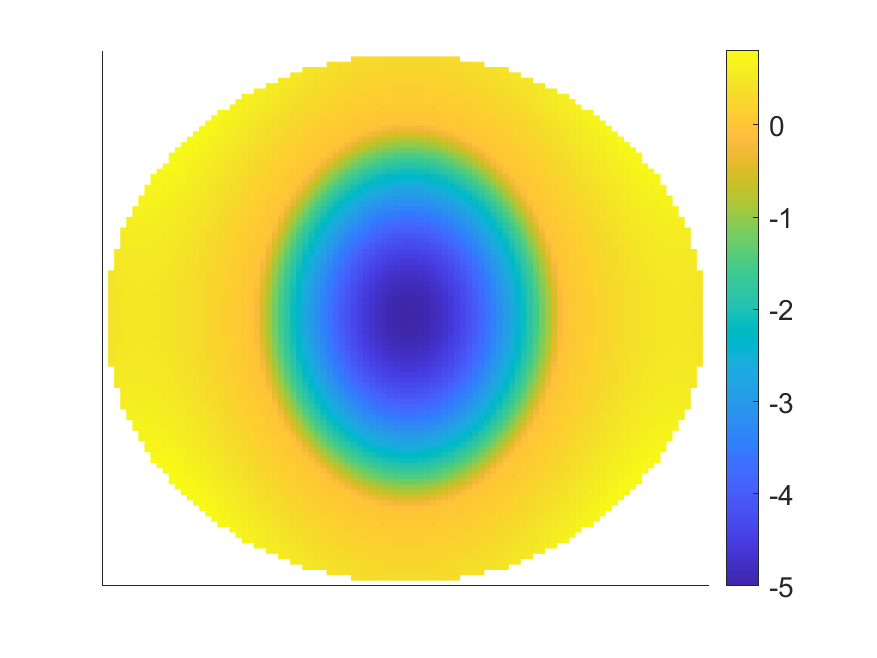} & \includegraphics[height=3cm]{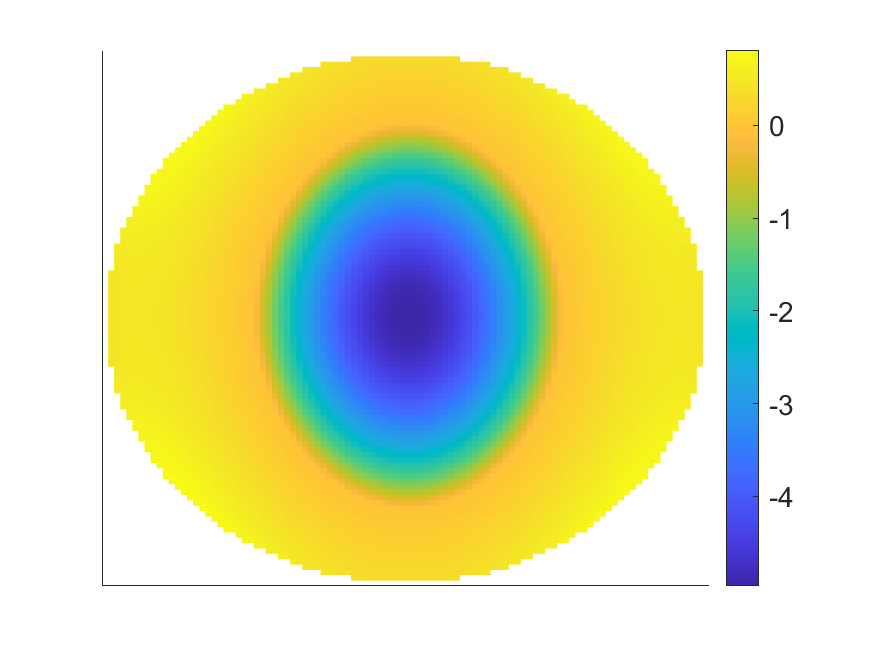} & \includegraphics[height=3cm]{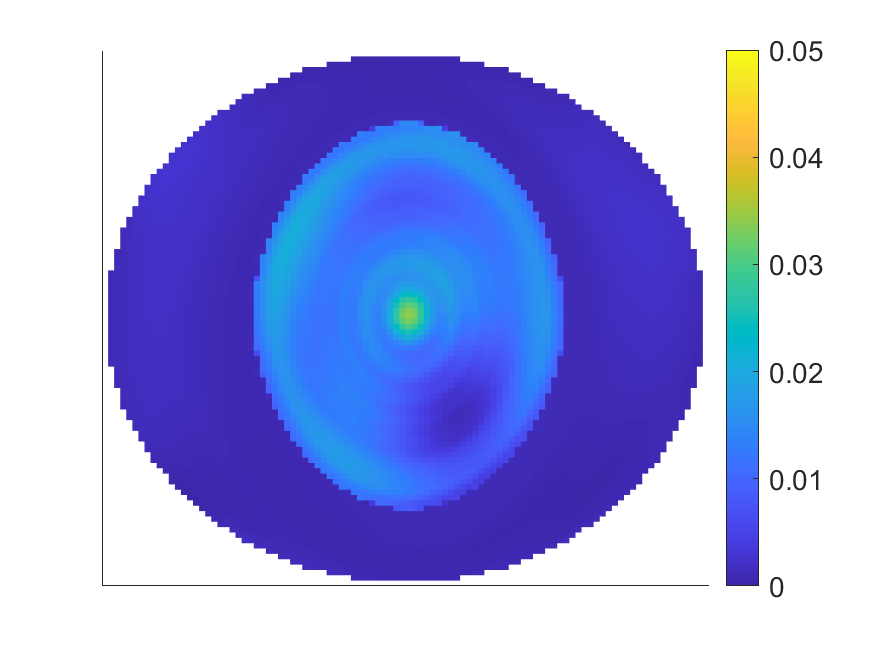}\\
\includegraphics[height=3cm]{true_solution-2} & \includegraphics[height=3cm]{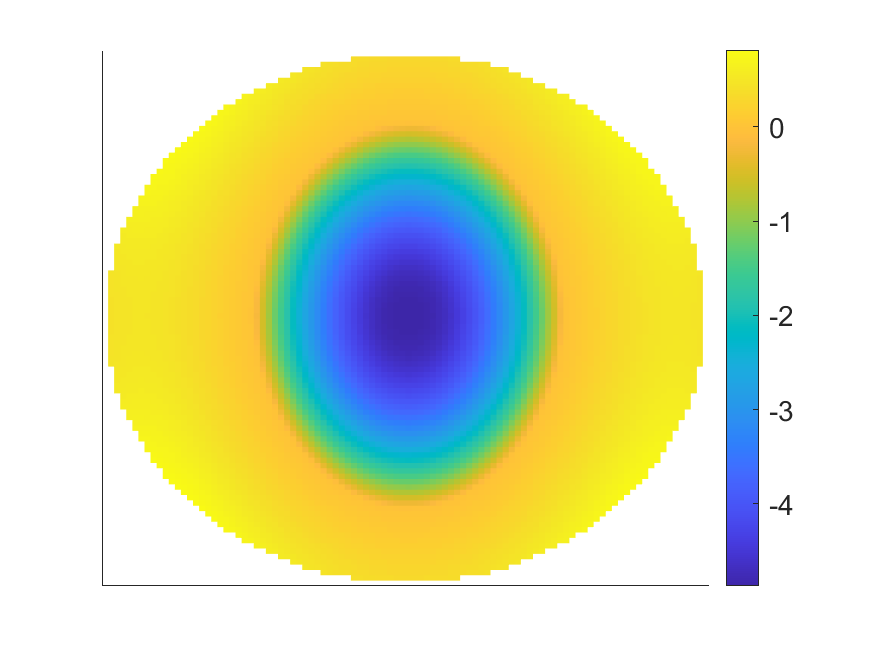} & \includegraphics[height=3cm]{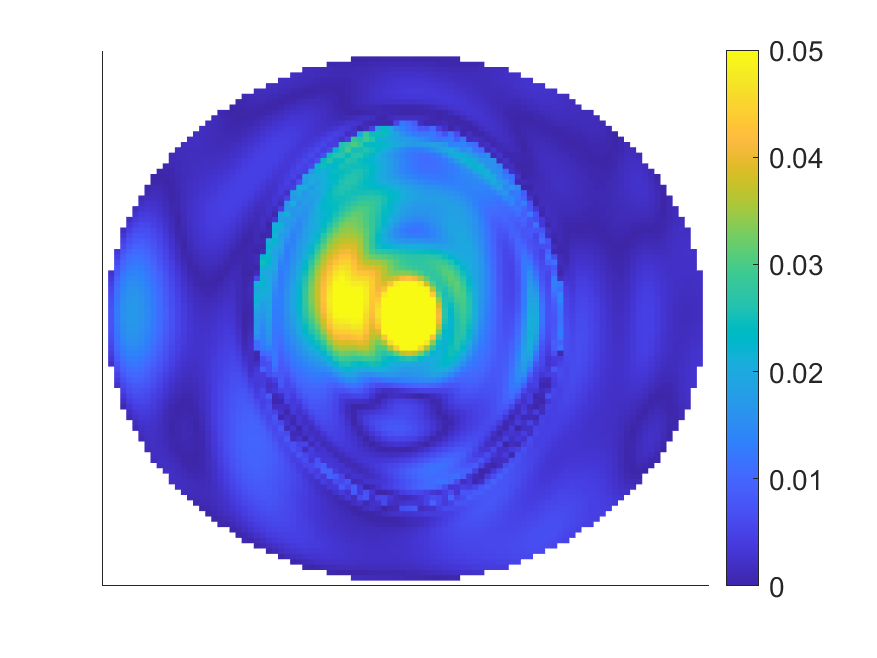}\\
\includegraphics[height=3cm]{true_solution-2} & \includegraphics[height=3cm]{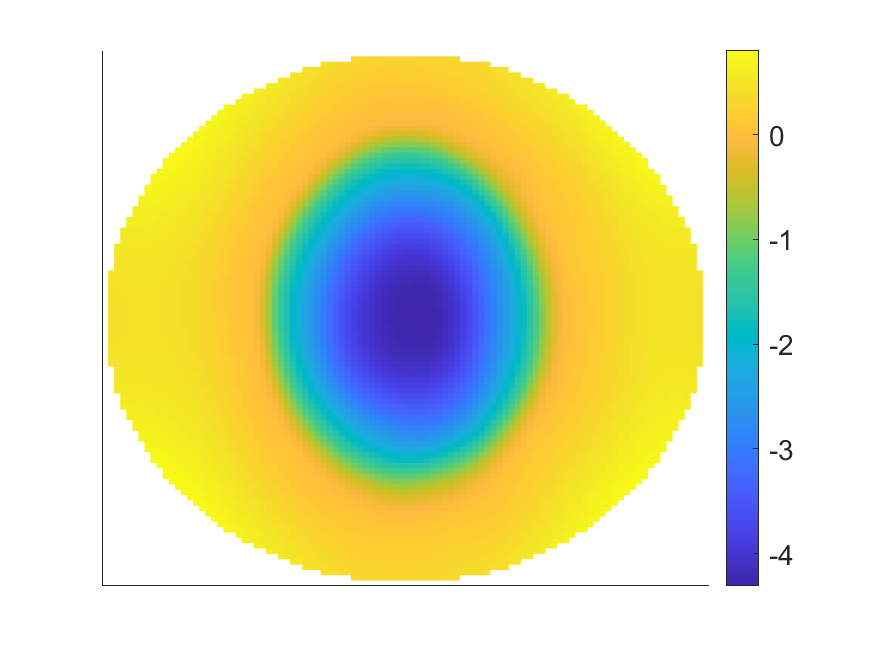} & \includegraphics[height=3cm]{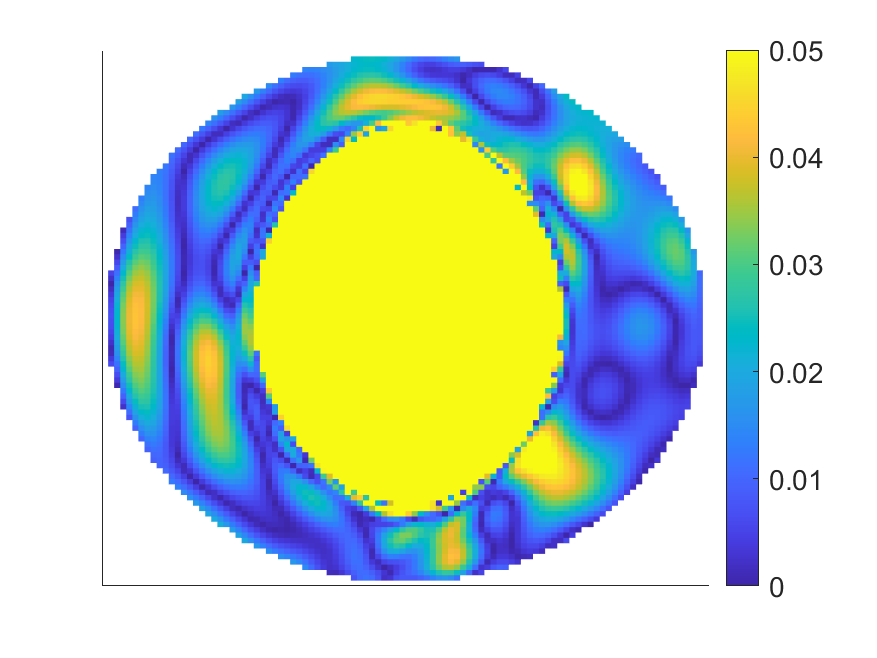}\\
(a) exact & (b) predicted  & (c) error
\end{tabular}
\caption{\label{fig:exam4:u} The approximations of $u$ for Example~\ref{exam4}, Case~(a) (slice at $x_3 = 0$) by the H-IDRM, H-IDRM (w/o LSNN), and DRM (from top to bottom).}
\end{figure}

\begin{figure}[hbt!]
\centering\setlength{\tabcolsep}{2pt}
\begin{tabular}{ccc}
\includegraphics[height=3cm]{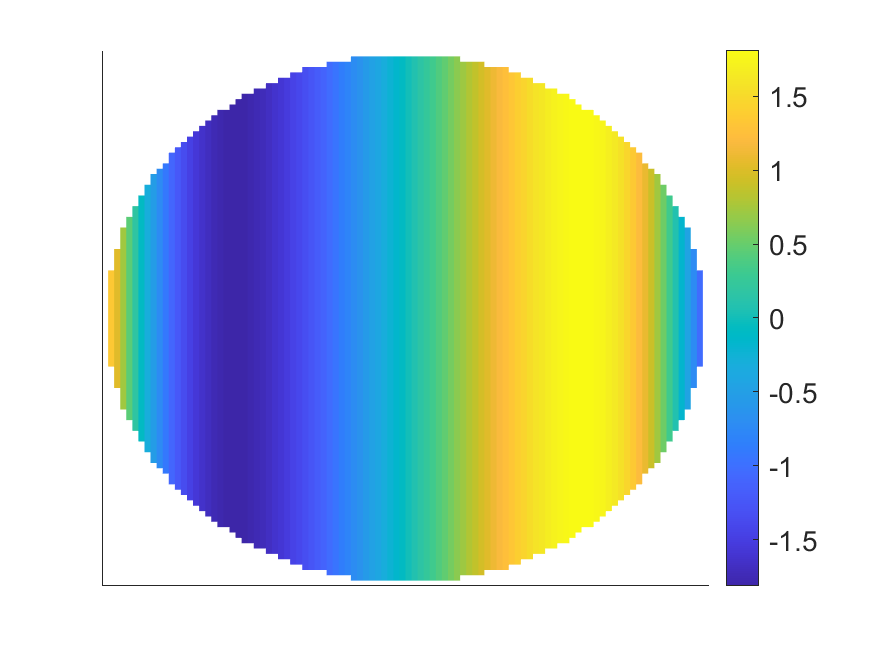} & \includegraphics[height=3cm]{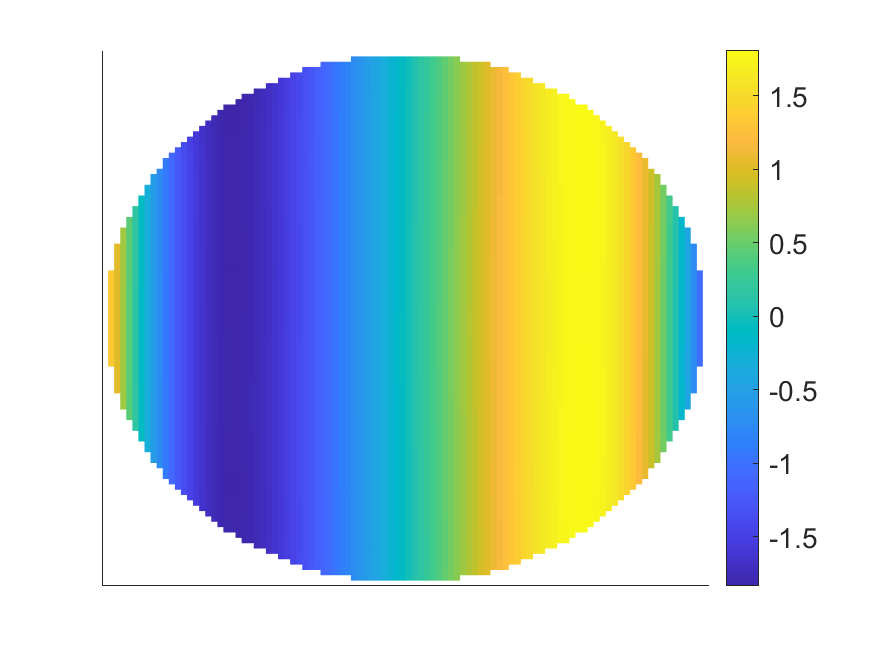} & \includegraphics[height=3cm] {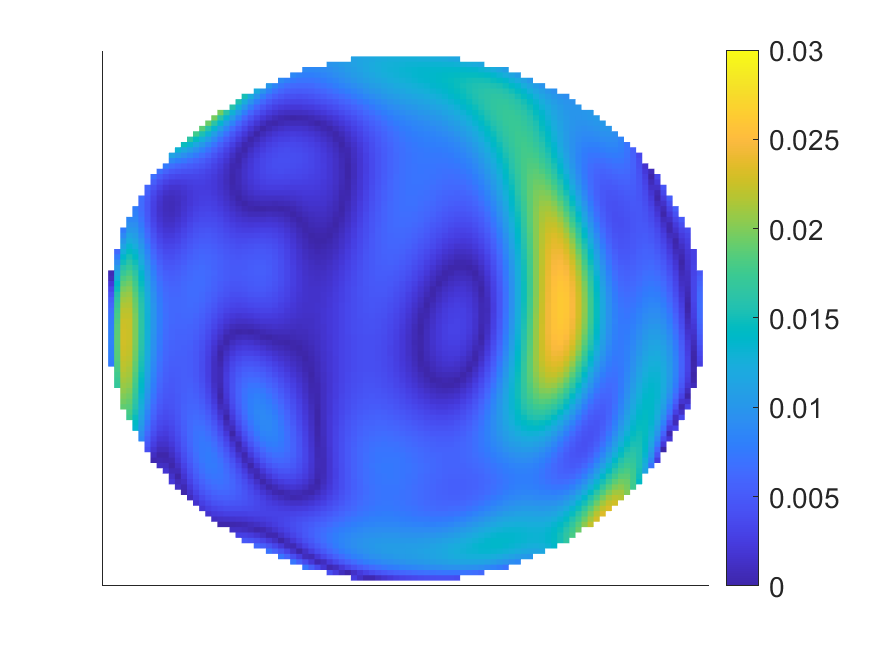}\\
\includegraphics[height=3cm]{true_gradsolution1-2} & \includegraphics[height=3cm]{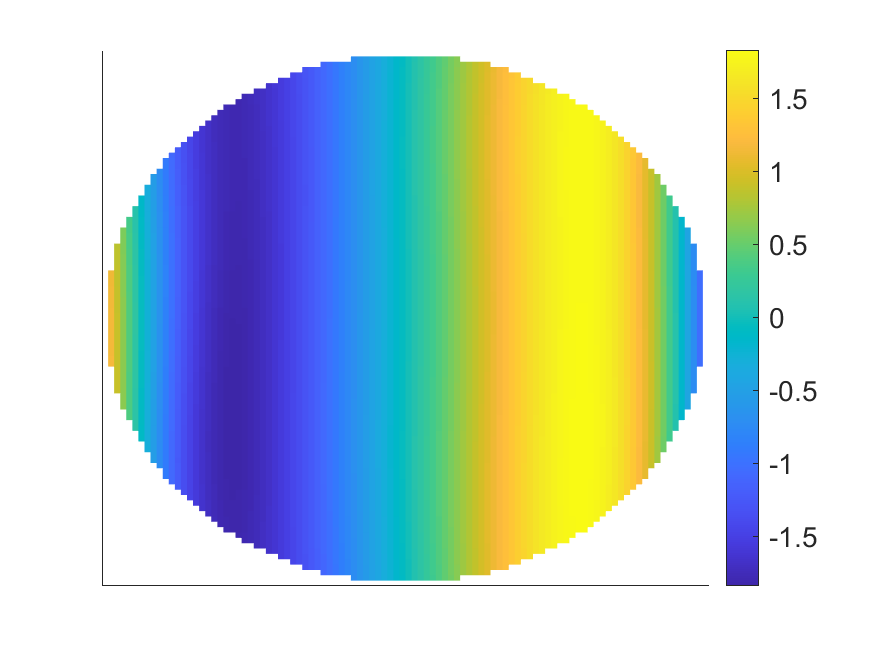} & \includegraphics[height=3cm]  {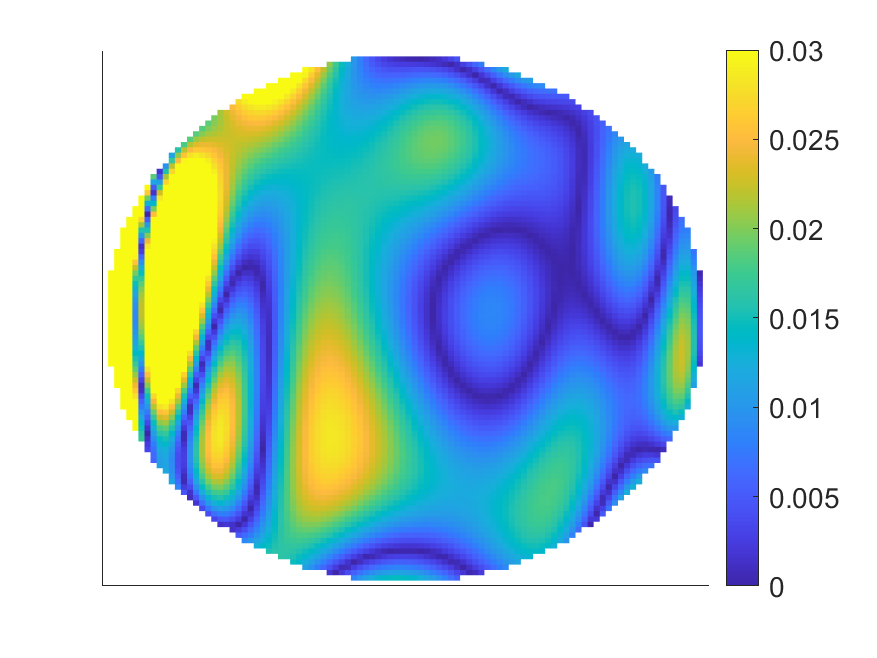}\\
\includegraphics[height=3cm]{true_gradsolution1-2} & \includegraphics[height=3cm]{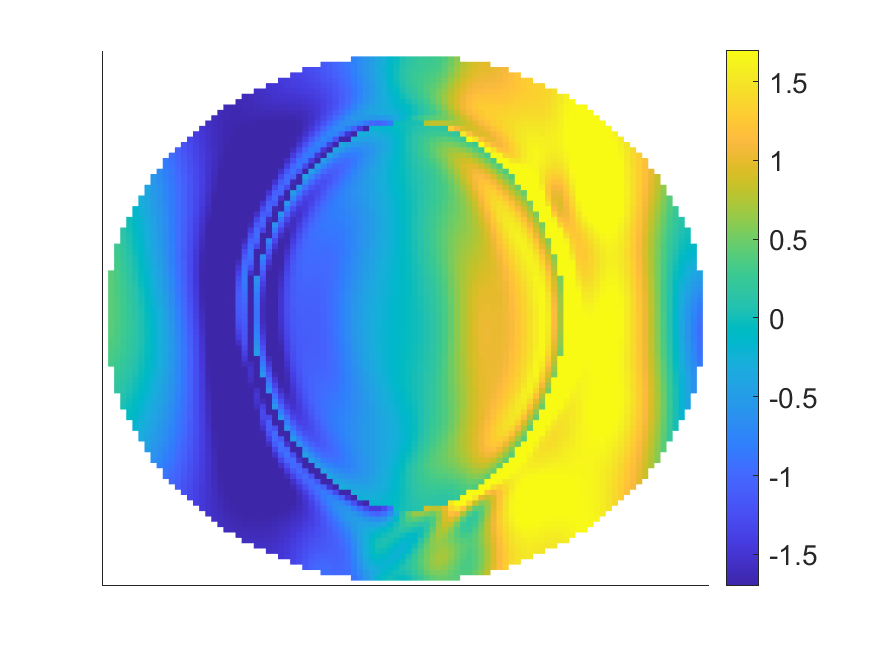} & \includegraphics[height=3cm]  {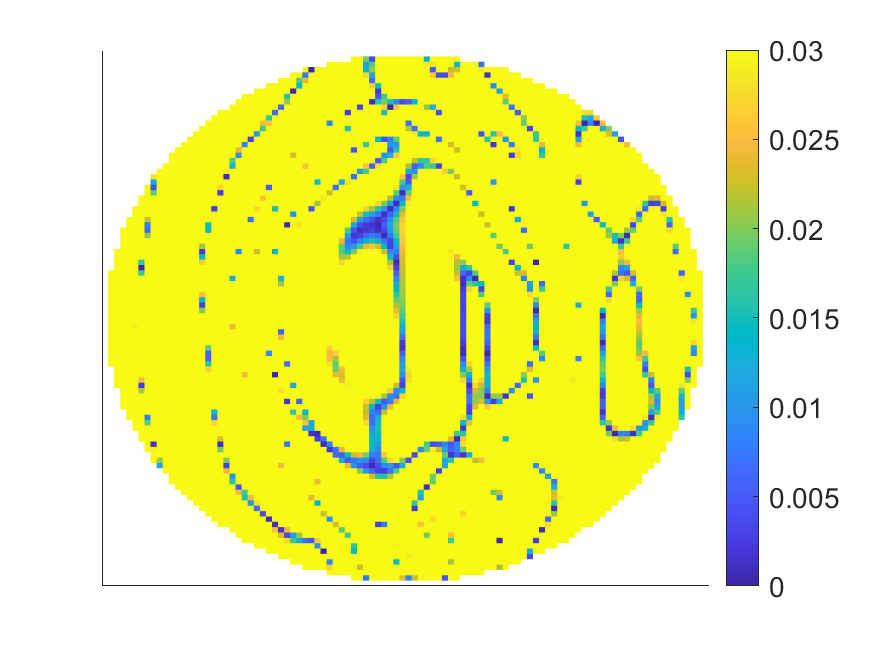}\\
(a) exact & (b) predicted  & (c) error
\end{tabular}
\caption{\label{fig:exam4:p1} The approximations of $ a\partial_{x_1}u$ for Example~\ref{exam4}, Case~(a) (slice at $x_3 = 0$), by the H-IDRM, H-IDRM (w/o LSNN), and DRM (from top to bottom).}
\end{figure}

In Case (b), the problem is degenerate and we add a regularization with $\epsilon = 10^{-3}$. Fig.~\ref{fig:exam4b} presents the exact solution, predictions, and pointwise absolute errors for the solution $u$ and the flux $\mathbf{p}$, with the relative errors $e_u = 6.87 \times 10^{-3}$ and $e_{\mathbf{p}} = 9.78 \times 10^{-3}$.
Despite the geometric complexity, the H-IDRM achieves high accuracy, which confirms the robustness of the regularized formulation.

\begin{figure}[hbt!]
\centering\setlength{\tabcolsep}{2pt}
\begin{tabular}{ccc}
\includegraphics[height=3cm]{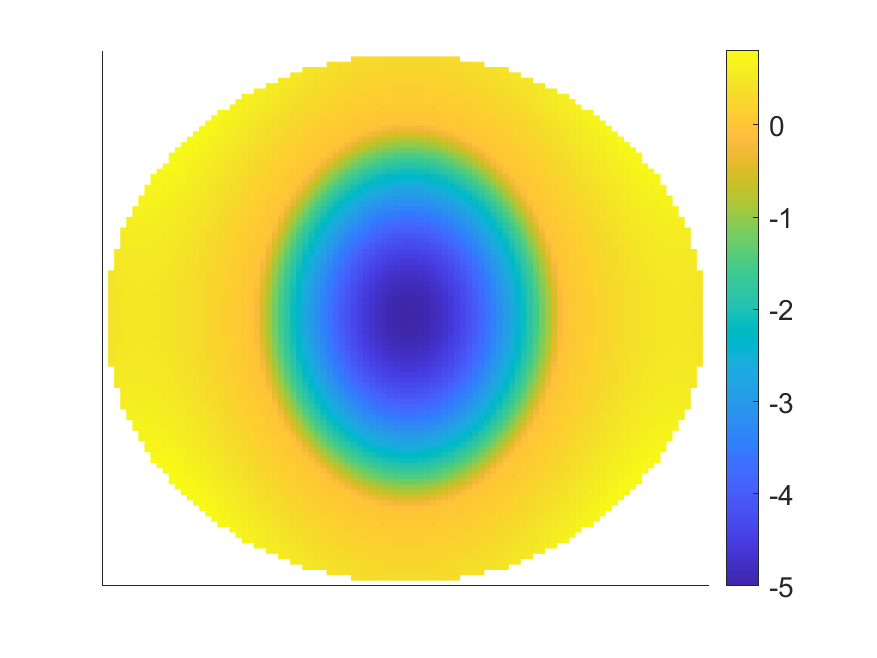} & \includegraphics[height=3cm]{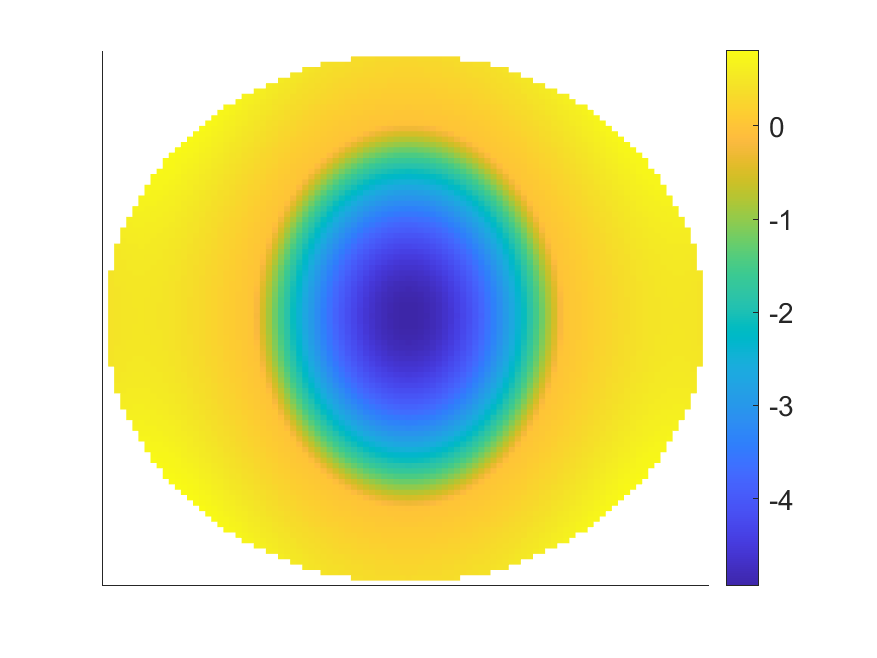} & \includegraphics[height=3cm] {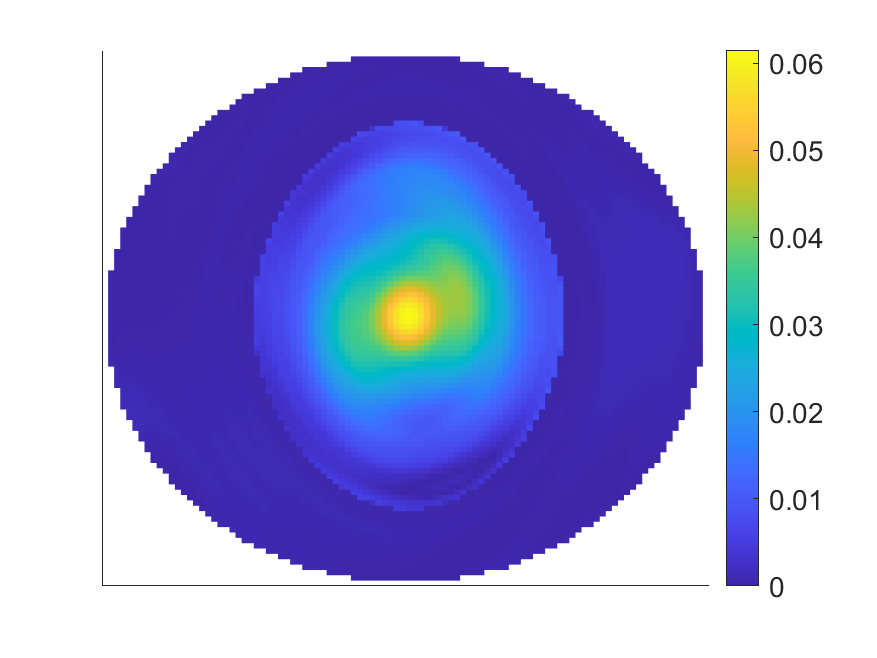}\\
\includegraphics[height=3cm]{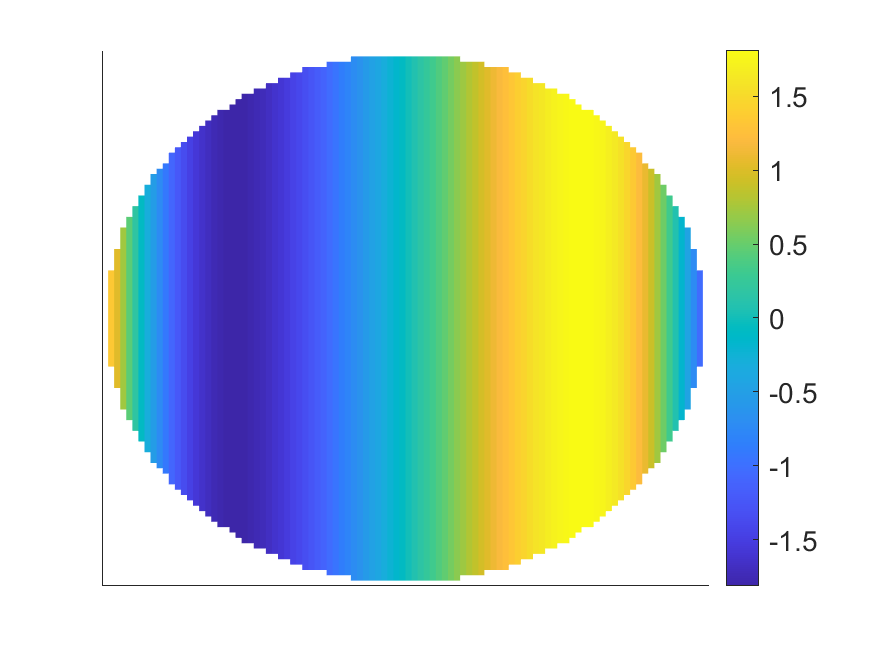} & \includegraphics[height=3cm]{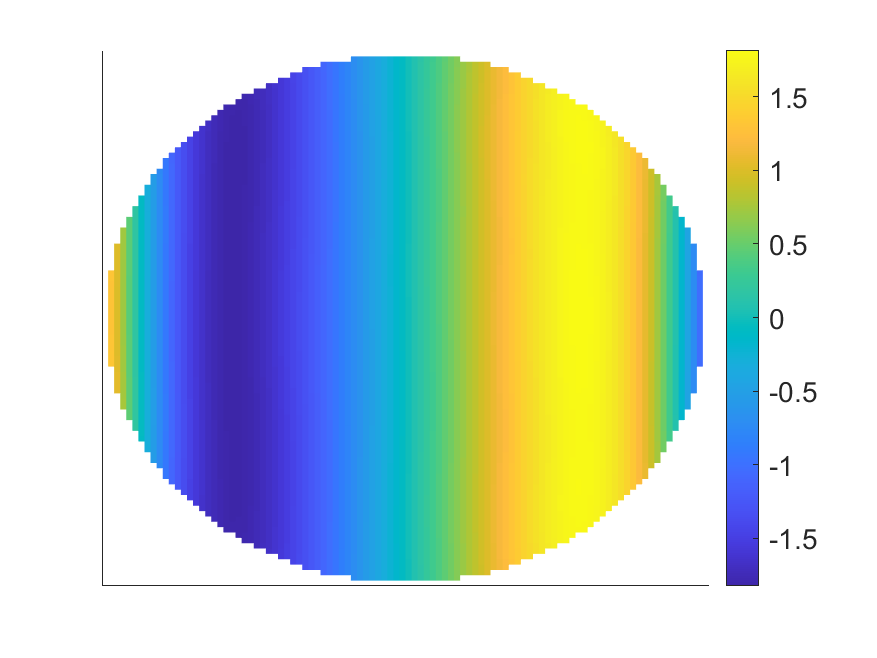} & \includegraphics[height=3cm]  {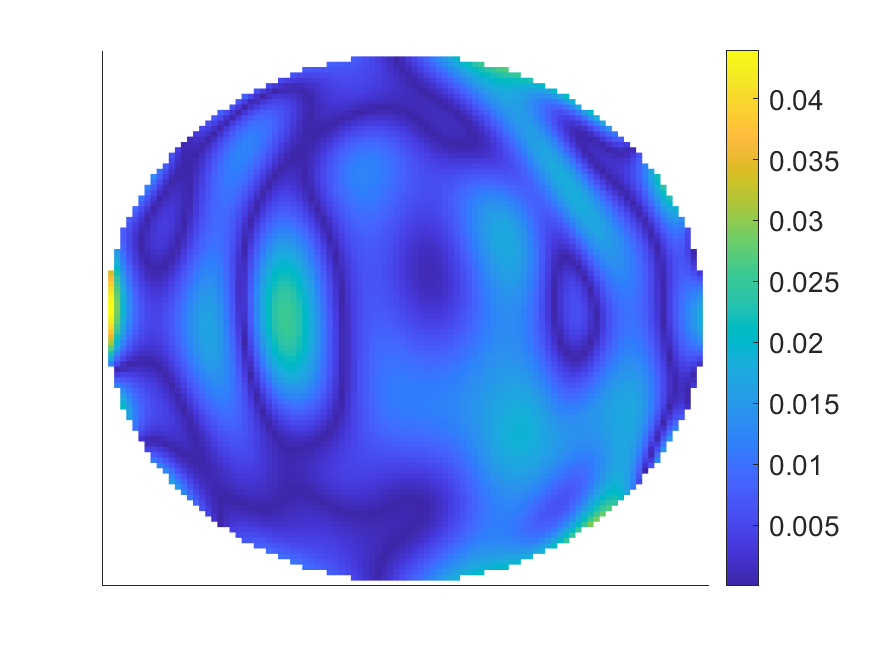}\\
\includegraphics[height=3cm]{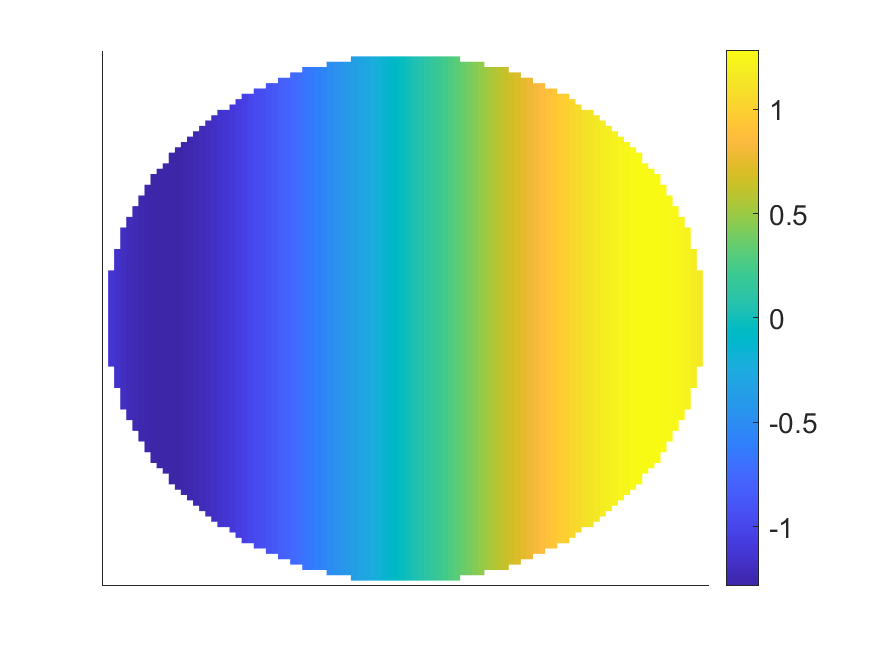} & \includegraphics[height=3cm]{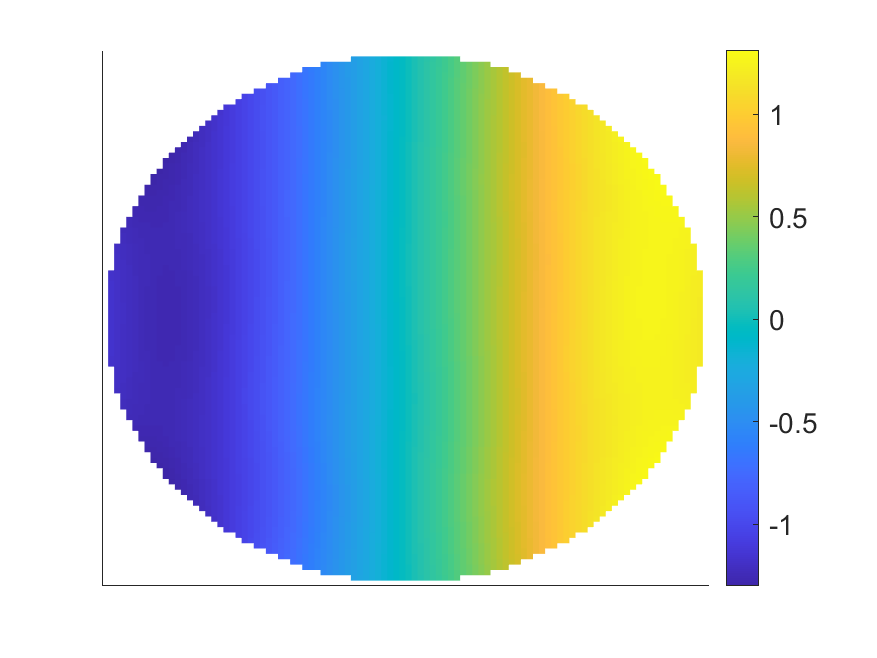} & \includegraphics[height=3cm]  {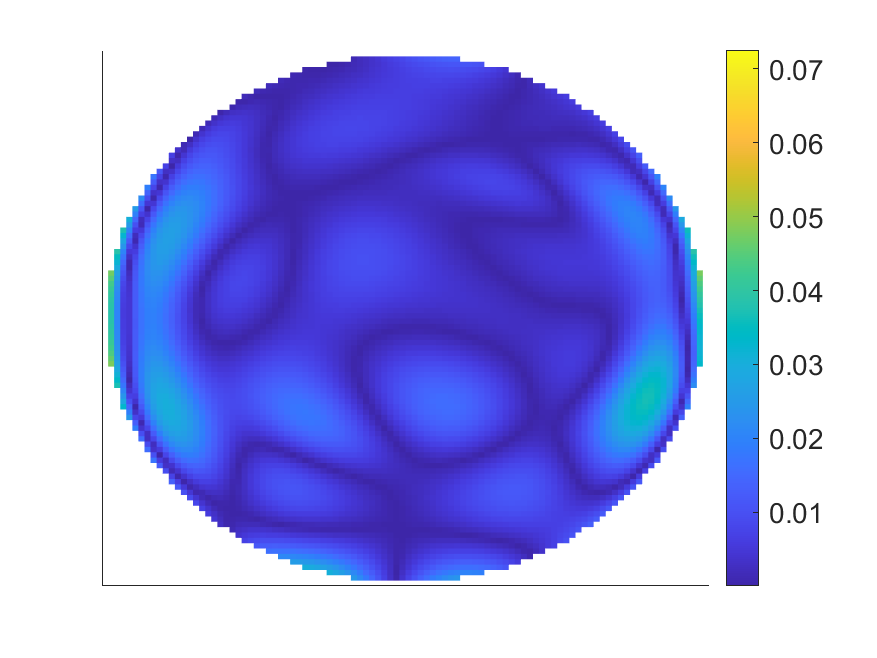}\\
\raisebox{1.8ex}{\hspace{1.0em}\includegraphics[height=2.6cm]{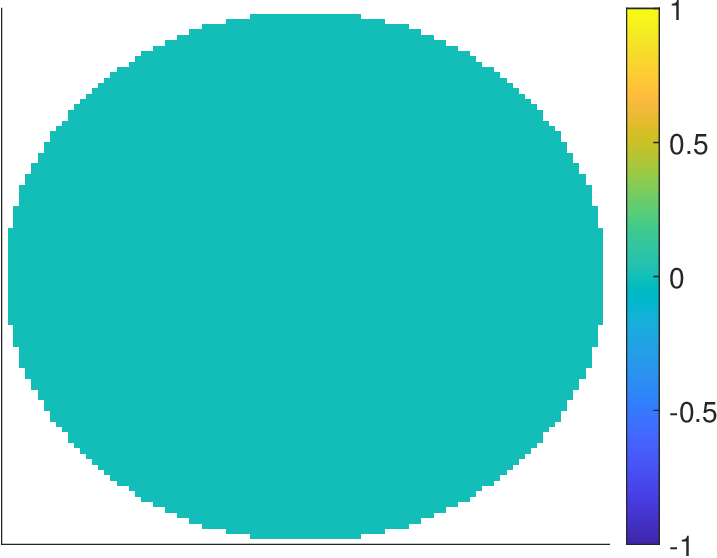}} & \includegraphics[height=3cm]{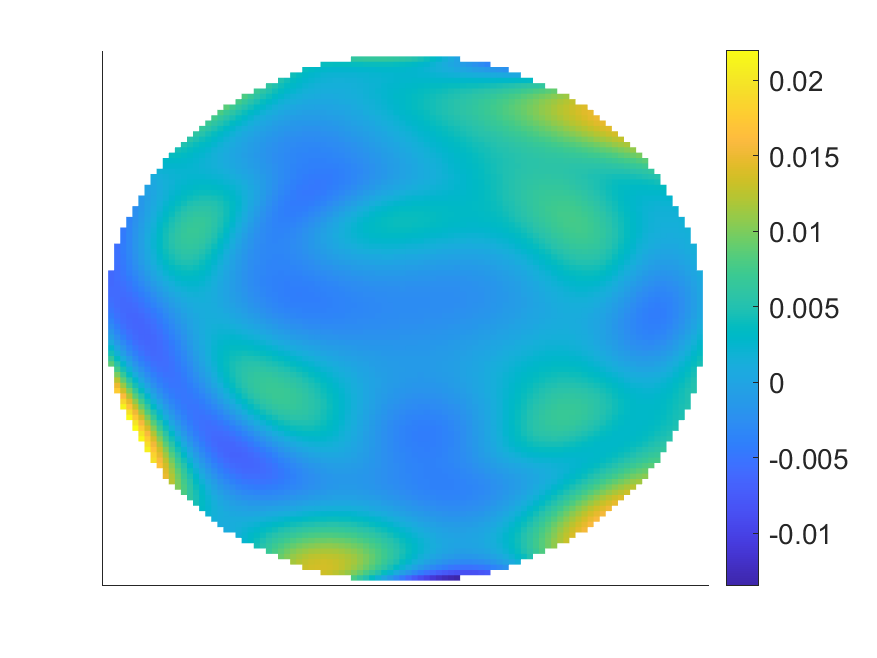} & \includegraphics[height=3cm]  {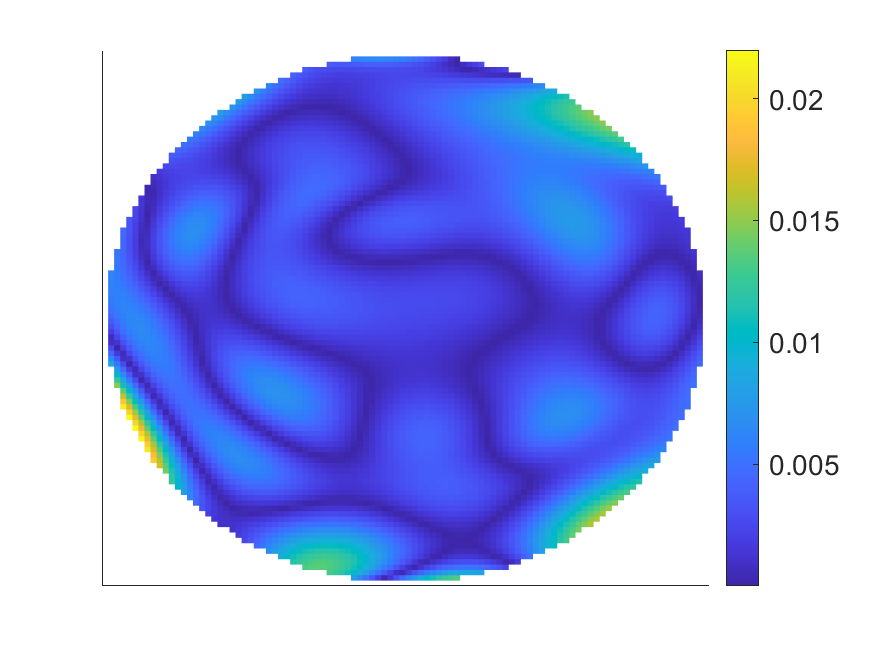}\\
(a) exact & (b) predicted  & (c) error
\end{tabular}
\caption{\label{fig:exam4b} The approximations of $u$ and $\mathbf{p}=(p_1, p_2, p_3)^\top$ (from top to bottom) for Example~\ref{exam4}, Case~(b) (slice at $x_3 = 0$).}
\end{figure}

The last example is concerned with nonzero jumps in both the solution and the conormal derivative. The 5D setting is adapted from~\cite[Example 5]{doi:10.1137/24M1632309}.

\begin{Example}\label{exam5}
Let $\Omega = \{\boldsymbol{x} \in \mathbb{R}^5 : \|\boldsymbol{x}\|_2 < 1\}$ be the unit ball.
The subdomains are separated by a concentric spherical interface of radius $r_0 = 0.5$: $\Omega_1 = \{\boldsymbol{x} \in \Omega : \|\boldsymbol{x}\| > r_0\}$ and $\Omega_2 = \{\boldsymbol{x} \in \Omega : \|\boldsymbol{x}\| < r_0\}$.
The problem data are $a(\boldsymbol{x}) = \chi_{\Omega_1}(\boldsymbol{x})+100\chi_{\Omega_2}(\boldsymbol{x})$ and $c(\boldsymbol{x}) = 1$.
The exact solution $u(\boldsymbol{x})$ is given by
$u(\boldsymbol{x})=\frac{1}{5}\|\boldsymbol{x}\|_2^2 \chi_{\Omega_1}(\boldsymbol{x}) + \frac15 \sum_{i=1}^5 x_i \chi_{\Omega_2}(\boldsymbol{x})$.
\end{Example}

To treat the nonzero jump conditions, we employ a two-stage procedure.
First, we construct a FCNN $v_\zeta$ that matches the prescribed jumps on the interface $\Gamma$ by minimizing the loss
\begin{equation*}
L_{\text{jump}}(v) = \|v - [u]\|_{L^2(\Gamma)}^2 + \|a\partial_{\mathbf{n}}v - [a\nabla u \cdot \mathbf{n}]\|_{L^2(\Gamma)}^2.
\end{equation*}
The NN $v_\zeta$ is trained for $10{,}000$ epochs using the Adam optimizer with a learning rate $10^{-3}$.
Then we extend $v_\zeta$ by $\tilde{u}(\boldsymbol{x}) = 
v_\zeta(\boldsymbol{x})\chi_{\Omega_2}(\boldsymbol{x})$.
By construction, $\tilde{u}$ absorbs the interface discontinuities. The correction $\hat{u} = u - \tilde{u}$ satisfies a modified elliptic problem with zero jump conditions:
\begin{equation*}
\left\{
\begin{aligned}
-\nabla \cdot (a\nabla \hat{u}) + c\hat{u} = f + \nabla \cdot (a\nabla \tilde{u}) - c\tilde{u}, &\quad \text{in } \Omega, \\
[\hat{u}] = 0, \quad [a\nabla \hat{u} \cdot \mathbf{n}] = 0, &\quad \text{on } \Gamma, \\
\hat{u} = -\tilde{u}, &\quad \text{on } \partial\Omega.
\end{aligned}
\right.
\end{equation*}
Then we employ the H-IDRM to approximate $\hat{u}$ and the associated flux $\hat{\mathbf{p}}$, using the LSNN architectures of $6$-$16$-$32$-$64$-$32$-$16$-$1$ for $\hat{u}$ and $6$-$16$-$32$-$16$-$5$ for $\hat{\mathbf{p}}$, and the level-set function $\phi(\boldsymbol{x}) = \|\boldsymbol{x}\|_2 - 0.5$. 
The overall approximations are given by $u_\theta = \hat{u}_\theta + \tilde{u}$ and $\mathbf{p}_\eta = \hat{\mathbf{p}}_\eta + a\nabla \tilde{u}$.

Fig.~\ref{fig:exam5} presents the solution $u$, the H-IDRM prediction, the pointwise error, and the flux component $a\partial_{x_1}u$ on slices $x_3 = x_4 = x_5 = 0$.
The H-IDRM achieves relative errors  $e_u = 2.71 \times 10^{-3}$ and $e_{\mathbf{p}} = 4.84 \times 10^{-3}$.
The accuracy is comparable to that of the 3D examples discussed earlier, confirming that the two-stage procedure effectively handles inhomogeneous interface conditions without compromising performance and that the H-IDRM scales well to higher-dimensional settings.

\begin{figure}[hbt!]
\centering\setlength{\tabcolsep}{2pt}
\begin{tabular}{ccc}
\includegraphics[height=3cm]{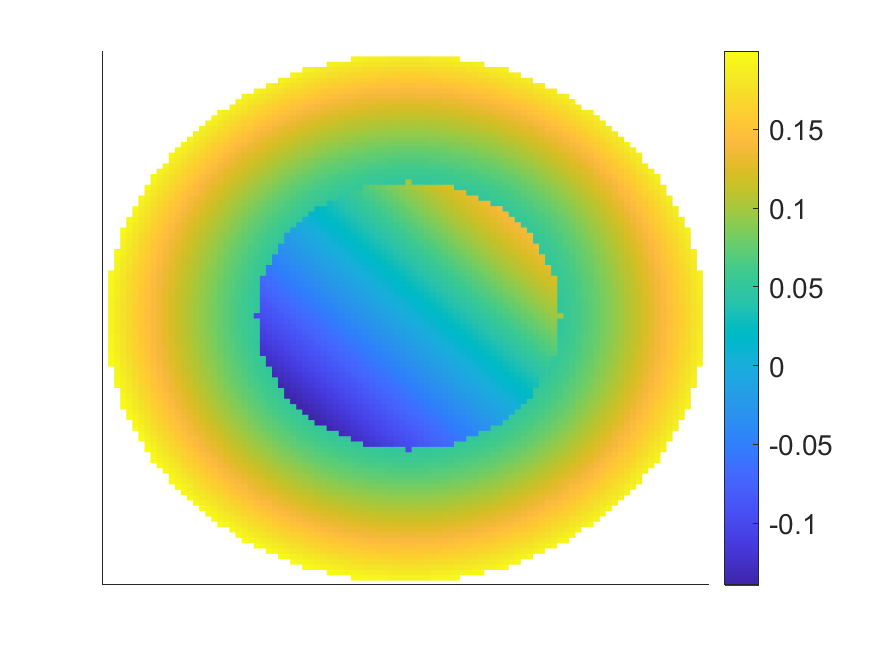} & \includegraphics[height=3cm]{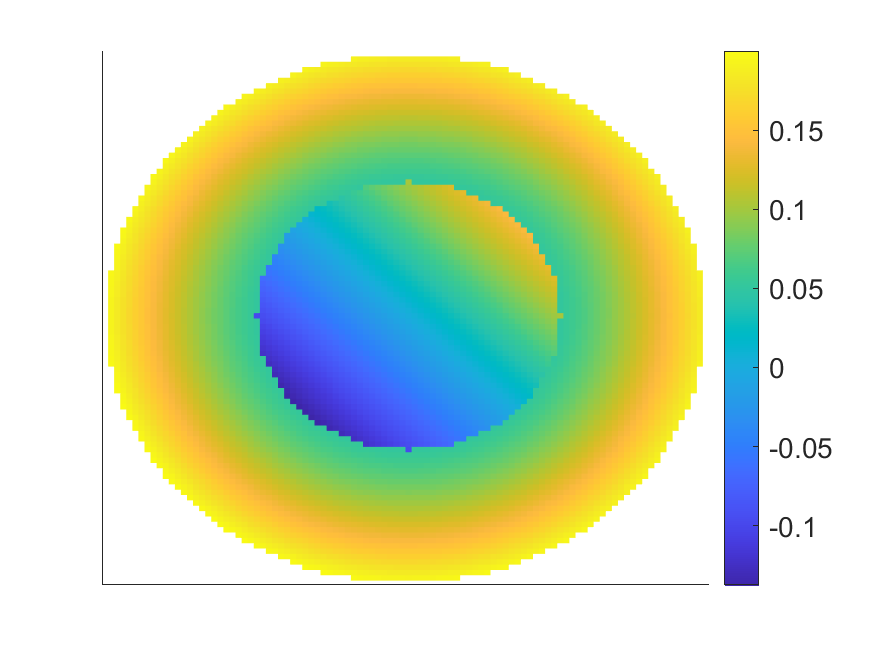} & \includegraphics[height=3cm] {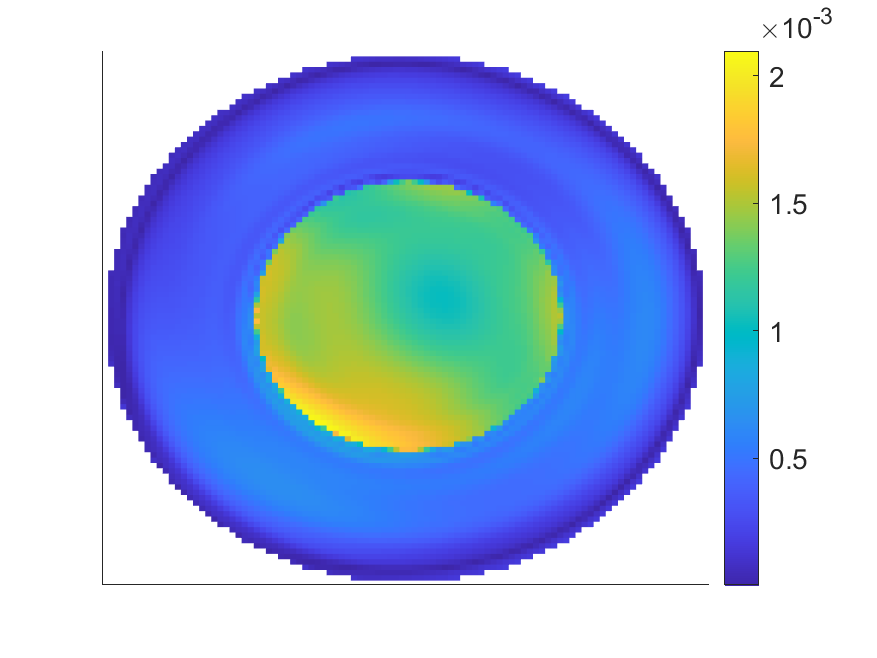}\\
\includegraphics[height=3cm]{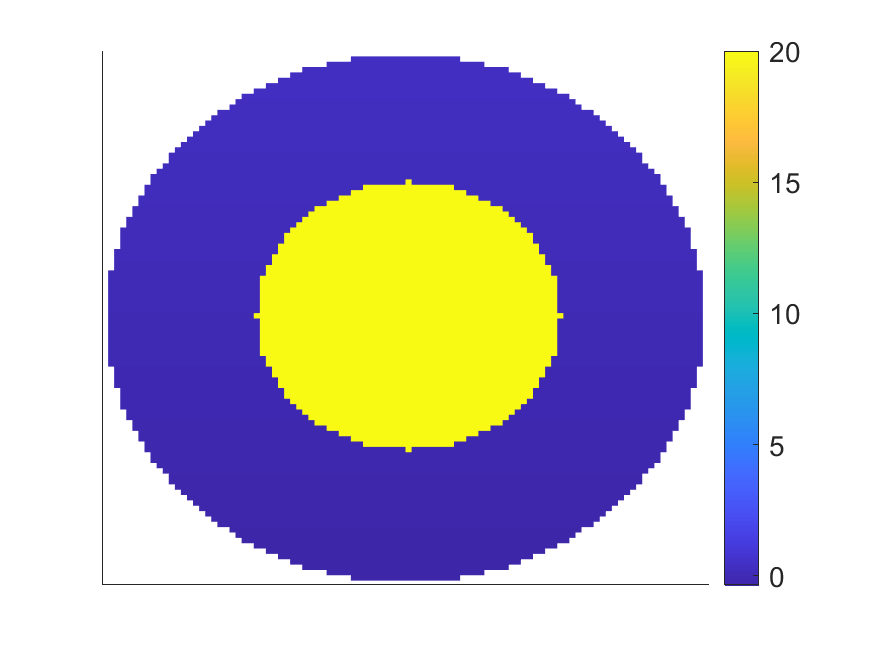} & \includegraphics[height=3cm]{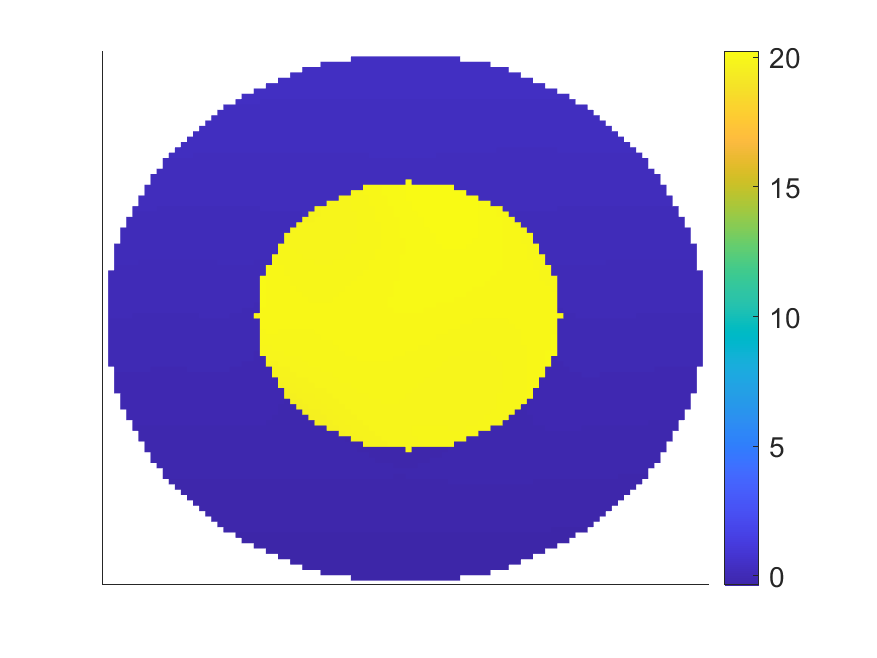} & \includegraphics[height=3cm]  {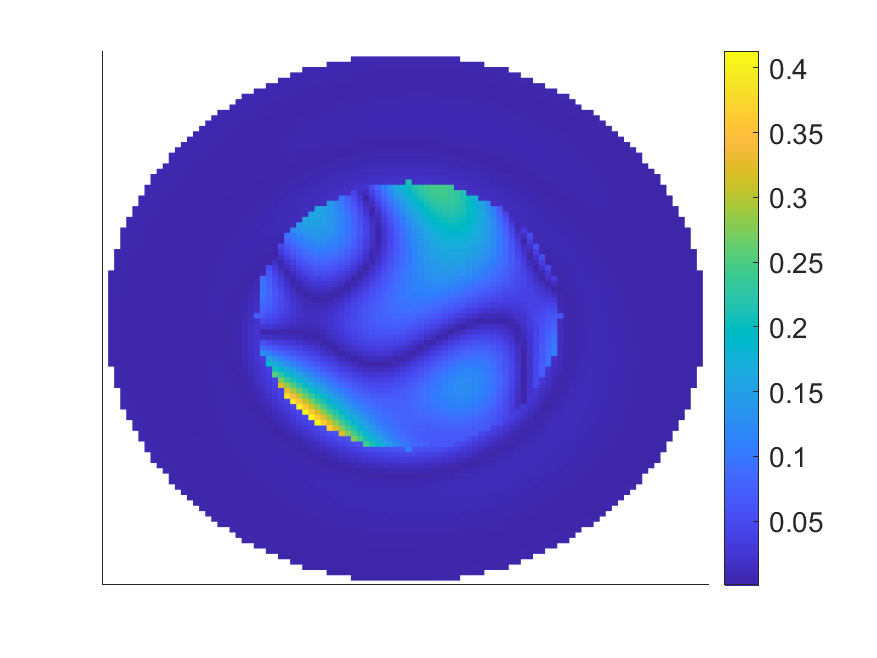}\\
(a) exact & (b) predicted  & (c) error
\end{tabular}
\caption{\label{fig:exam5} The approximations of $u$ (top) and $a\partial_{x_1}u$ (bottom) for Example~\ref{exam5} (slices at $x_3 = x_4 = x_5 = 0$).}
\end{figure}

\section{Conclusions}\label{SEC:CON}

In this work, we have developed a hybrid iterative deep Ritz method (H-IDRM) for a class of elliptic interface problems. This method approximates the solution by a customized level-set neural network (LSNN) within an iterative optimization scheme based on a mixed formulation.
The mixed formulation recasts the interface problem as a sequence of convex minimization subproblems and avoids explicit interface sampling by converting the $H^{\frac{1}{2}}(\Gamma)$ and $H^{-\frac{1}{2}}(\Gamma)$ interface coupling into volumetric integrals.
The LSNN encodes the interface geometry through a smoothed Heaviside transformation of the level-set function.
Moreover, we prove the well-posedness of the mixed formulation, provide linear convergence of the iterative scheme, and derive approximation rates for the LSNN architecture.
Numerical experiments show that the H-IDRM consistently outperforms existing neural PDE solvers, particularly on problems with intricate interface geometries or lower global regularity.

\appendix

\section{Proof of Theorem \ref{thm:levelsetappro}}\label{append}
The proof consists of four steps. In Step~1, we construct an auxiliary function from the extended subdomain solutions and approximate it by an NN. In Step~2, we relate the approximation to the LSNN output through the level-set transformation. In Step~3, we bound the error by separating the domain into regions near and away from the interface. Last, in Step~4, we obtain the final rate.

\noindent {\it Step 1. Construction and analysis of $v$.} By the Sobolev extension theorem~\cite[Theorem 5.19]{Adams:2003}, for $k=1,2$, there exist extension operators
$E_k:H^{\frac{3}{2}+r}(\Omega_k)\to H^{\frac{3}{2}+r}(\Omega)$
such that $E_ku_k|_{\Omega_k}=u_k$ a.e. in $\Omega_k$.
Since $u^*\in H^1(\Omega)$, by a mollification argument \cite{Adams:2003}, for any $\delta>0$ the extensions can be chosen so that
\begin{equation}\label{eqn:mollifier}
\|E_ku_k-u_{3-k}\|_{L^2(\Omega)}
\le C\delta\|u^*\|_{H^1(\Omega)}
\qquad\text{and}\qquad
\|E_ku_k\|_{H^{\frac{3}{2}+r}(\Omega)}
\le C\delta^{-r-\frac{1}{2}}\|u^*\|_{H^1(\Omega)},
\end{equation}
where the constant $C$ depends only on $d$, $r$, $\Omega$, and the extension operators. Define $v:\Omega\times[0,1]\to\mathbb R$ by
\[
v(x_1,\dots,x_d,x_{d+1})
=(1-x_{d+1})E_1u_1(x_1,\dots,x_d)
+x_{d+1}E_2u_2(x_1,\dots,x_d).
\]
For each fixed \(x_{d+1} \in [0, 1]\), \(v(\cdot, x_{d+1})\) is a linear combination of \(E_1 u_1\) and \(E_2 u_2\). Thus \(v(\cdot, x_{d+1}) \in H^{\frac{3}{2}+r}(\Omega)\), with $\sup_{x_{d+1} \in [0, 1]} \|v(\cdot, x_{d+1})\|_{H^{\frac{3}{2}+r}(\Omega)} < \infty$. This implies 
\[
v \in L^\infty(0, 1; H^{\frac{3}{2}+r}(\Omega)) \subset L^2(0, 1; H^{\frac{3}{2}+r}(\Omega)).
\]
Moreover, since \(v\) is affine in \(x_{d+1}\), we have
\begin{equation}\label{eqn:dv}
\partial_{x_{d+1}} v = E_2 u_2 - E_1 u_1 \in H^{\frac{3}{2}+r}(\Omega) \subset L^2(\Omega),
\end{equation}
and all higher derivatives vanish. Hence, \(v \in C^\infty(0, 1; L^2(\Omega))\), and in particular
$v \in H^{\frac{3}{2}+r}(0, 1; L^2(\Omega))$. Thus,
$v \in L^2(0, 1; H^{\frac{3}{2}+r}(\Omega)) \cap H^{\frac{3}{2}+r}(0, 1; L^2(\Omega))$. By the anisotropic Sobolev embedding theorem~\cite[Chapter 1, Theorem 2.3]{lions1972non}, $v\in H^{\frac{3}{2}+r}(\Omega\times(0,1))$. Moreover, by the definition of the Sobolev–Slobodeckij norm \cite{DiNezza2012},
\begin{equation}\label{eqn:vnorm}
\|v\|_{H^{\frac{3}{2}+r}(\Omega\times(0,1))}
\le C\left(\|E_1u_1\|_{H^{\frac{3}{2}+r}(\Omega)}
+\|E_2u_2\|_{H^{\frac{3}{2}+r}(\Omega)}\right)
\le C\delta^{-r-\frac{1}{2}}\|u^*\|_{H^1(\Omega)}.
\end{equation}
For any $\nu>0$, by the approximation capability of $\tanh$ NNs in Sobolev spaces~\cite[Proposition 4.8]{GuhringRaslan2021}, there exists a NN
$v_\eta\in\mathcal N(D,W,W^{\frac{9}{2}+\frac{2r+2}{d+1}})$ with $D=C\log(d+2+r)$, such that
\begin{equation}\label{eqn:nnapprox}
\|v-v_\eta\|_{H^{\frac{3}{2}}(\Omega\times(0,1))}
\le C W^{-\frac{r-\nu}{d+1}}
\|v\|_{H^{\frac{3}{2}+r}(\Omega\times(0,1))}
\le C W^{-\frac{r-\nu}{d+1}}
\delta^{-r-\frac{1}{2}}\|u^*\|_{H^1(\Omega)},
\end{equation}
where $C$ depends on $d$, $r$, and $\nu$. Note that $r\ge1$ is required in \cite[Proposition 4.8]{GuhringRaslan2021}. However, by a mollification argument, \eqref{eqn:nnapprox} remains valid for all $r>0$; see Remark~\ref{remark:r<1}.

\noindent {\it Step 2. Using $v$ to bridge $u_{\theta,\alpha}$ and $u^*$.} Set
\[
u_{\theta,\alpha}(\boldsymbol{x})=v(\boldsymbol{x},t(\boldsymbol{x})),
\quad \mbox{with }
t(\boldsymbol{x})=({1+e^{\alpha\phi(\boldsymbol{x})}})^{-1}.
\]
By the triangle inequality,
\begin{align}\label{eqn:err-decom}
\|u_{\theta,\alpha}-u^*\|_V
\le \|u_{\theta,\alpha}-v(\cdot,t(\cdot))\|_V
+\|v(\cdot,t(\cdot))-u^*\|_V.
\end{align}
Observe that
\[
(1+|\nabla t(\boldsymbol{x})|^2)^{-1/2}\le 1
\quad\text{and}\quad
|\nabla t(\boldsymbol{x})|
=\frac{|\alpha e^{\alpha\phi(\boldsymbol{x})}\nabla\phi(\boldsymbol{x})|}{(1+e^{\alpha\phi(\boldsymbol{x})})^2}
\le \|\phi\|_{C^1(\Omega)}\alpha,
\]
Thus, applying the trace theorem \cite{Adams:2003} on the graph
$\Gamma_t=\{(\boldsymbol x,t(\boldsymbol x)):\boldsymbol x\in\Omega\}$, we obtain
\[
\begin{aligned}
&\|u_{\theta,\alpha}-v(\cdot,t(\cdot))\|_V^2
=\|v_\eta(\cdot,t(\cdot))-v(\cdot,t(\cdot))\|_V^2\\
\le& C\Big(
\|v_\eta-v\|_{L^2(\Gamma_t)}^2
+\sum_{i=1}^d\|\partial_{x_i}v_\eta-\partial_{x_i}v\|_{L^2(\Gamma_t)}^2
+\|\phi\|_{C^1(\Omega)}\alpha\|\partial_{x_{d+1}}(v_\eta-v)\|_{L^2(\Gamma_t)}^2
\Big)\\
\le& C\alpha\|v_\eta-v\|_{H^{\frac32}(\Omega\times(0,1))}^2,
\end{aligned}
\]
where $C$ depends on $d$, $\Omega$, and $\|\phi\|_{C^1(\Omega)}$. Then, using \eqref{eqn:nnapprox}, we get
\begin{equation}\label{eqn:est0}
\|u_{\theta,\alpha}-v(\cdot,t(\cdot))\|_V
\le C \alpha^{\frac12} W^{-\frac{r-\nu}{d+1}}\delta^{-r-\frac12},
\end{equation}
where $C$ depends on $d$, $r$, $\Omega$, $\|\phi\|_{C^1(\Omega)}$, and $\nu$.

\noindent {\it Step 3. Bound on $\|v(\cdot,t(\cdot))-u^*\|_V$.}
Fix $\epsilon>0$ small and decompose
\[
A_1=\{\boldsymbol x\in\Omega:\phi(\boldsymbol x)\ge \epsilon\},\quad
A_2=\{\boldsymbol x\in\Omega:\phi(\boldsymbol x)\le -\epsilon\},\quad
B=(\Omega_1\cup\Omega_2)\backslash (A_1\cup A_2).
\]
By the triangle inequality,
\[
\|v(\cdot,t(\cdot))-u^*\|_V
\le {\rm I}+{\rm II}+{\rm III},
\]
with
${\rm I}=\|v(\cdot,t(\cdot))-u^*\|_{H^1(A_1)}$, ${\rm II}=\|v(\cdot,t(\cdot))-u^*\|_{H^1(A_2)}$, and $
{\rm III}=\|v(\cdot,t(\cdot))-u^*\|_{H^1(B)}$.
On $A_1$, we have $u^*=E_1u_1$ and
$v(\cdot,t(\cdot))-E_1u_1=t(E_2u_2-E_1u_1)$.
Moreover,
\[
t\le (1+e^{\alpha\epsilon})^{-1}
\quad\text{and}\quad
|\partial_{x_i}t|\le \|\phi\|_{C^1(\Omega)}\alpha(1+e^{\alpha\epsilon})^{-1}.
\]
Thus, by \eqref{eqn:mollifier},
\begin{equation}\label{eqn:estI}
\begin{aligned}
{\rm I}&=\|v(\cdot,t(\cdot))-E_1u_1\|_{H^1(A_1)}\\
&\le \|t\|_{W^{1,\infty}(A_1)}
\left(\|E_1u_1\|_{H^1(A_1)}+\|E_2u_2\|_{H^1(A_1)}\right)\\
&\le C\alpha(1+e^{\alpha\epsilon})^{-1}\delta^{-r-\frac12},
\end{aligned}
\end{equation}
where $C$ depends on $d$, $r$, $\Omega$, $\|\phi\|_{C^1(\Omega)}$, and $\|u^*\|_{H^1(\Omega)}$.
Similarly, on $A_2$,
\[
\|1-t\|_{W^{1,\infty}(A_2)}
\le \|\phi\|_{C^1(\Omega)}\alpha(1+e^{\alpha\epsilon})^{-1},
\]
and hence
\begin{equation}\label{eqn:estII}
{\rm II}
\le C\alpha(1+e^{\alpha\epsilon})^{-1}\delta^{-r-\frac12}.
\end{equation}
It remains to estimate ${\rm III}$. 
By the triangle inequality, we have
\begin{equation}\label{ineq:decompIII}
    {\rm III}=\|v(\cdot,t(\cdot))-u^*\|_{H^1(B)}\leq \|v(\cdot,t(\cdot))\|_{H^1(B)}+\|u^*\|_{H^1(B)}
\end{equation}
Since $0\le t\le1$ and $|\nabla t|\le \|\phi\|_{C^1(\Omega)}\alpha$, we have
$$\begin{aligned}
&\|v(\cdot,t(\cdot))\|_{H^1(B)}=\|(1-t)E_1u_1+tE_2u_2\|_{H^1(B)}\\
\leq&\|E_1u_1\|_{H^1(B)}+\|t\|_{W^{1,\infty}(\Omega)}\|E_2u_2-E_1u_1\|_{L^2(B)}+\|t\|_{L^{\infty}(\Omega)}\|E_2u_2-E_1u_1\|_{H^1(B)}\\
\leq&C(\|E_1u_1\|_{H^1(B)}+\|E_2u_2\|_{H^1(B)}+\alpha\|E_2u_2-E_1u_1\|_{L^2(B)}).
\end{aligned}$$
where $C$ depends on $\|\phi\|_{C^1(\Omega)}$.
Since
$$H^{\frac{3}{2}+r}(B)\hookrightarrow \begin{cases}
W^{1,\frac{2d}{d-1-2r}}(B), &\text{if}\quad r<\frac{d-1}{2},\\
W^{1,p}(B), \forall p>1, &\text{if}\quad r=\frac{d-1}{2},\\
W^{1,\infty}(B), &\text{if}\quad r>\frac{d-1}{2}.\\
\end{cases}$$
If $r<\frac{d-1}{2}$, by H\"older's inequality, there holds for $k=1,2$,
\begin{equation*}
\begin{aligned}
\|E_ku_k\|_{H^1(B)}\leq\operatorname{meas}(B)^{\frac{2r+1}{2d}}\|E_ku_k\|_{W^{1,\frac{2d}{d-1-2r}}(B)}= C\epsilon^{\frac{2r+1}{2d}}\|E_ku_k\|_{W^{1,\frac{2d}{d-1-2r}}(B)}\leq C\epsilon^{\frac{2r+1}{2d}}.
\end{aligned}
\end{equation*}
If $r>\frac{d-1}{2}$, we have $\|E_ku_k\|_{H^1(B)}\leq \|E_ku_k\|_{W^{1,\infty}(B)}\operatorname{meas}(B)^{\frac{1}{2}}=C\epsilon^{\frac{1}{2}}$ for $k=1,2$. For the borderline case $r = \frac{d-1}{2}$, one may select any $r' < r$ arbitrarily close to $r$, apply the estimate for $r < \frac{d-1}{2}$, and then let $r' \to r$ to obtain $\|E_ku_k\|_{H^1(B)} \le C \epsilon^{\frac{1}{2}}$. %, which coincides with the bound obtained for $r > \frac{d-1}{2}$. 
Further, we have
$$\|E_2u_2-E_1u_1\|_{L^2(B)}\le\|E_2u_2-u_1\|_{L^2(B\cap\Omega_1)}+\|u_2-E_1u_1\|_{L^2(B\cap\Omega_2)}\leq C\delta\|u^*\|_{H^1(\Omega)}.$$
Since $u^*\in H^{\frac{3}{2}+r}(B)$, we obtain
$$\|u^*\|_{H^1(B)}\leq C\epsilon^{\min\{\frac{2r+1}{2d},\frac{1}{2}\}}.$$
Now, combining the three estimates and choosing $\gamma=\min\{\frac{2r+1}{2d},\frac{1}{2}\}$  give
\begin{equation}\label{eqn:estIII}
{\rm III}\leq C\left(\epsilon^\gamma+\delta\right).
\end{equation}

\noindent {\it Step 4. Balancing the parameters.} Last, we choose the parameters so that the error terms in \eqref{eqn:est0}, \eqref{eqn:estI}, \eqref{eqn:estII}, and \eqref{eqn:estIII} are of the same order by setting $
\epsilon=\alpha^{-\frac12}$, $
\delta=\alpha^{-\frac{\gamma}{2}}$, and $\alpha=W^{\frac{2(r-\nu)}{(d+1)(1+\gamma r+1.5\gamma)}}$.
With these choices, substituting \eqref{eqn:est0}, \eqref{eqn:estI}, \eqref{eqn:estII}, and \eqref{eqn:estIII} into the decomposition \eqref{eqn:err-decom} yields the desired approximation rate.

\begin{Remark}\label{remark:r<1}
The approximation result~\cite[Proposition~4.8]{GuhringRaslan2021} is stated for $r\ge 1$.
For $0 < r < 1$, the bound~\eqref{eqn:nnapprox} can be derived via a mollification argument. Specifically,
let $v_h$ be a smoothed version of $v$ in $H^{\frac{5}{2}}(\Omega\times(0,1))$ satisfying
\begin{align*}
\|v_h\|_{H^{\frac{5}{2}}(\Omega\times(0,1))} &\le ch^{r-1}\|v\|_{H^{\frac{3}{2}+r}(\Omega\times(0,1))},\\
\|v-v_h\|_{H^{\frac{3}{2}}(\Omega\times(0,1))} &\le ch^{r}\|v\|_{H^{\frac{3}{2}+r}(\Omega\times(0,1))},
\end{align*}
with $c$ independent of $h$ and $v$.
Using the integer-order approximation to $v_h$ and balancing the terms by choosing $h = W^{-\frac{1}{d+1}}$ yield
\begin{equation*}
\|v-v_\eta\|_{H^{\frac{3}{2}}(\Omega\times(0,1))}\leq \|v-v_h\|_{H^{\frac{3}{2}}(\Omega\times(0,1))} + \|v_h-v_{\eta}\|_{H^{\frac{3}{2}}(\Omega\times(0,1))} \leq CW^{-\frac{r-\nu}{d+1}}\|v\|_{H^{\frac{3}{2}+r}(\Omega\times(0,1))},
\end{equation*}
which extends~\eqref{eqn:nnapprox} to the full range $r>0$.
\end{Remark}

\bibliographystyle{abbrv}
\bibliography{ref}

\begin{thebibliography}{10}

\bibitem{Adams:2003}
R.~A. Adams and J.~J.~F. Fournier.
\newblock {\em Sobolev {S}paces}.
\newblock Elsevier, Amsterdam, 2nd edition, 2003.

\bibitem{AdjeridGuo:2023}
S.~Adjerid, I.~Babuška, R.~Guo, and T.~Lin.
\newblock An enriched immersed finite element method for interface problems
  with nonhomogeneous jump conditions.
\newblock {\em Comput. Methods Appl. Mech. Eng.}, 404:115770, 2023.

\bibitem{ARNOLD1990281}
D.~N. Arnold.
\newblock Mixed finite element methods for elliptic problems.
\newblock {\em Comput. Methods Appl. Mech. Eng.}, 82(1):281--300, 1990.

\bibitem{Bartlett2003RademacherAG}
P.~L. Bartlett and S.~Mendelson.
\newblock Rademacher and {G}aussian complexities: Risk bounds and structural
  results.
\newblock {\em J. Mach. Learn. Res.}, 3:463--482, 2002.

\bibitem{MixedandHybridFiniteElementMethod}
F.~Brezzi and M.~Fortin.
\newblock {\em {Mixed and Hybrid Finite Element Method}}.
\newblock Springer, Berlin, 1991.

\bibitem{CaiYeZhang:2011}
Z.~Cai, X.~Ye, and S.~Zhang.
\newblock Discontinuous {G}alerkin finite element methods for interface
  problems: a priori and a posteriori error estimations.
\newblock {\em SIAM J. Numer. Anal.}, 49(5):1761--1787, 2011.

\bibitem{2015Stabilized}
L.~Cattaneo, L.~Formaggia, G.~F. Iori, A.~Scotti, and P.~Zunino.
\newblock Stabilized extended finite elements for the approximation of saddle
  point problems with unfitted interfaces.
\newblock {\em Calcolo}, 52(2):123--152, 2015.

\bibitem{CHAO19932021}
C.~Chao and R.~Chang.
\newblock Steady-state heat conduction problem of the interface crack between
  dissimilar anisotropic media.
\newblock {\em Int. J. Heat Mass Transfer}, 36(8):2021--2026, 1993.

\bibitem{ChenZou:1998}
Z.~Chen and J.~Zou.
\newblock Finite element methods and their convergence for elliptic and
  parabolic interface problems.
\newblock {\em Numer. Math.}, 79:175--202, 1998.

\bibitem{DiNezza2012}
E.~Di~Nezza, G.~Palatucci, and E.~Valdinoci.
\newblock Hitchhiker's guide to the fractional {S}obolev spaces.
\newblock {\em Bull. Sci. Math.}, 136(5):521--573, 2012.

\bibitem{yu2018deep}
W.~E and B.~Yu.
\newblock The deep {R}itz method: a deep learning-based numerical algorithm for
  solving variational problems.
\newblock {\em Commun. Math. Stat.}, 6:1--12, 2018.

\bibitem{FanTan:2025}
H.~Fan and Z.~Tan.
\newblock Novel and general discontinuity-removing {PINN}s for elliptic
  interface problems.
\newblock {\em J. Comput. Phys.}, 529:113861, 28 pp., 2025.

\bibitem{GiraultRaviart:1986}
V.~Girault and R.~P. A.
\newblock {\em {F}inite {E}lement {M}ethods for {N}avier-{S}tokes {E}quations}.
\newblock Springer-verlag, Berlin, 1986.

\bibitem{GuhringRaslan2021}
I.~Guhring and M.~Raslan.
\newblock Approximation rates for neural networks with encodable weights in
  smoothness spaces.
\newblock {\em Neural Networks}, 134:107--130, 2021.

\bibitem{hu2024directfiniteelementmethod}
J.~Hu and L.~Ma.
\newblock A direct finite element method for elliptic interface problems.
\newblock {\em J. Sci. Comput.}, 106(2):48, 2026.

\bibitem{HU2025113791}
T.~Hu, B.~Jin, and F.~Wang.
\newblock An iterative deep {R}itz method for monotone elliptic problems.
\newblock {\em J. Comput. Phys.}, 527:113791, 2025.

\bibitem{10.1093/imanum/draf129}
T.~Hu, B.~Jin, and Z.~Zhou.
\newblock Point source identification using singularity-enriched neural
  networks.
\newblock {\em IMA J. Numer. Anal.}, page draf129, 2026.

\bibitem{HU2022111576}
W.~F. Hu, T.~S. Lin, and M.~C. Lai.
\newblock A discontinuity capturing shallow neural network for elliptic
  interface problems.
\newblock {\em J. Comput. Phys.}, 469:111576, 2022.

\bibitem{HuangZou:2002}
J.~Huang and J.~Zou.
\newblock Some new a priori estimates for second-order elliptic and parabolic
  interface problems.
\newblock {\em J. Differential Equations}, 184(2):570--586, 2002.

\bibitem{HuangZou:2007}
J.~Huang and J.~Zou.
\newblock Uniform a priori estimates for elliptic and static {M}axwell
  interface problems.
\newblock {\em Discrete Contin. Dyn. Syst. Ser. B}, 7(1):145--170, 2007.

\bibitem{Jiao2022Rate}
Y.~Jiao, Y.~Lai, D.~Li, X.~Lu, F.~Wang, Y.~Wang, and J.~Z. Yang.
\newblock A rate of convergence of physics informed neural networks for the
  linear second order elliptic {PDE}s.
\newblock {\em Commun. Comput. Phys.}, 31(4):1272--1295, 2022.

\bibitem{JMLR:v26:24-1258}
Y.~Jiao, R.~Li, P.~Wu, J.~Z. Yang, and P.~Zhang.
\newblock {DRM} revisited: A complete error analysis.
\newblock {\em J. Mach. Learn. Res.}, 26(115):1--76, 2025.

\bibitem{Kellog:1971}
R.~B. Kellogg.
\newblock Singularities in interface problems.
\newblock In {\em Numerical {S}olution of {P}artial {D}ifferential {E}quations,
  {II} ({SYNSPADE} 1970)}, pages 351--400. Academic Press, New York-London,
  1971.

\bibitem{Kellog:1972}
R.~B. Kellogg.
\newblock Higher order singularities for interface problems.
\newblock In {\em The mathematical foundations of the finite element method
  with applications to partial differential equations}, pages 589--602.
  Academic Press, New York-London, 1972.

\bibitem{KingmaBa:2015}
D.~P. Kingma and J.~Ba.
\newblock Adam: A method for stochastic optimization.
\newblock In {\em 3rd International Conference for Learning Representations},
  San Diego, 2015.

\bibitem{LeVeque1997Immersed}
R.~J. LeVeque and Z.~Li.
\newblock Immersed interface methods for {S}tokes flow with elastic boundaries
  or surface tension.
\newblock {\em SIAM J. Sci. Comput.}, 18(3):709--735, 1997.

\bibitem{LI2025113847}
Y.~Li and F.~Wang.
\newblock Local randomized neural networks with finite difference methods for
  interface problems.
\newblock {\em J. Comput. Phys.}, 529:113847, 2025.

\bibitem{Li2006}
Z.~Li and K.~Ito.
\newblock {\em {The Immersed Interface Method: Numerical Solutions of PDEs
  Involving Interfaces and Irregular Domains}}.
\newblock SIAM, Philadelphia, PA, 2006.

\bibitem{LiLinWu:2003}
Z.~Li, T.~Lin, and X.~Wu.
\newblock New cartesian grid methods for interface problems using the finite
  element formulation.
\newblock {\em Numer. Math.}, 96:61--98, 2003.

\bibitem{lions1972non}
J.-L. Lions and E.~Magenes.
\newblock {\em Non-Homogeneous Boundary Value Problems and Applications}.
\newblock Springer-Verlag, Berlin, 1972.

\bibitem{Osher2001LevelSet}
S.~Osher and R.~P. Fedkiw.
\newblock Level set methods: an overview and some recent results.
\newblock {\em J. Comput. Phys.}, 169(2):463--502, 2001.

\bibitem{Peskin:2002}
C.~S. Peskin.
\newblock The immersed boundary method.
\newblock {\em Acta Numer.}, 11:479--517, 2002.

\bibitem{RAISSI2019686}
M.~Raissi, P.~Perdikaris, and G.~E. Karniadakis.
\newblock Physics-informed neural networks: A deep learning framework for
  solving forward and inverse problems involving nonlinear partial differential
  equations.
\newblock {\em J. Comput. Phys.}, 378:686--707, 2019.

\bibitem{SARMA2024117135}
A.~K. Sarma, S.~Roy, C.~Annavarapu, P.~Roy, and S.~Jagannathan.
\newblock Interface {PINN}s ({I}-{PINN}s): A physics-informed neural networks
  framework for interface problems.
\newblock {\em Comput. Methods Appl. Mech. Eng.}, 429:117135, 2024.

\bibitem{Sun2025DirichletNeumann}
Q.~Sun, X.~Xu, and H.~Yi.
\newblock Dirichlet-{N}eumann learning algorithm for solving elliptic interface
  problems.
\newblock {\em Commun. Comput. Phys.}, 38(1):248--284, 2025.

\bibitem{TsengLinHu:2023}
Y.-H. Tseng, T.-S. Lin, W.-F. Hu, and M.-C. Lai.
\newblock A cusp-capturing {PINN} for elliptic interface problems.
\newblock {\em J. Comput. Phys.}, 491:112359, 16 pp., 2023.

\bibitem{WU2022111588}
S.~Wu and B.~Lu.
\newblock {INN}: Interfaced neural networks as an accessible meshless approach
  for solving interface {PDE} problems.
\newblock {\em J. Comput. Phys.}, 470:111588, 2022.

\bibitem{YaoGu:2023}
Y.~Yao, J.~Guo, and T.~Gu.
\newblock A deep learning method for multi-material diffusion problems based on
  physics-informed neural networks.
\newblock {\em Comput. Methods Appl. Mech. Engrg.}, 417:116395, 21 pp., 2023.

\bibitem{doi:10.1137/24M1632309}
J.~Ying, J.~Hu, Z.~Shi, and J.~Li.
\newblock An accurate and efficient continuity-preserved method based on
  randomized neural networks for elliptic interface problems.
\newblock {\em SIAM J. Sci. Comput.}, 46(5):C633--C657, 2024.

\bibitem{ZhouZhaoWei:2006}
Y.~C. Zhou, S.~Zhao, M.~Feig, and G.~W. Wei.
\newblock High order matched interface and boundary method for elliptic
  equations with discontinuous coefficients and singular sources.
\newblock {\em J. Comput. Phys.}, 213(1):1--30, 2006.

\bibitem{doi:10.1137/22M1517081}
X.~Zhu, X.~Hu, and P.~Sun.
\newblock Physics-informed neural networks for solving dynamic two-phase
  interface problems.
\newblock {\em SIAM J. Sci. Comput.}, 45(6):A2912--A2944, 2023.

\bibitem{Zilian2009}
A.~Zilian and T.-P. Fries.
\newblock A localized mixed-hybrid method for imposing interfacial constraints
  in the extended finite element method ({XFEM}).
\newblock {\em Int. J. Numer. Methods Eng.}, 79(6):733--752, 2009.

\end{thebibliography}

\end{document}